\documentclass[11pt]{article}
\usepackage[T1]{fontenc}
\usepackage{lmodern}
\usepackage[margin=1in]{geometry}
\usepackage{amsmath,amssymb,amsthm,mathtools}
\usepackage{array}
\usepackage{xcolor}
\usepackage{microtype}
\usepackage{hyperref}
\hypersetup{colorlinks=true,linkcolor=blue!45!black,citecolor=blue!45!black,urlcolor=blue!45!black,
 pdftitle={Finite-time blowup of the IPM equation from smooth initial data without forcing},pdfauthor={Fan Zheng}}
\newtheorem{theorem}{Theorem}[section]
\newtheorem{proposition}[theorem]{Proposition}
\newtheorem{lemma}[theorem]{Lemma}
\newtheorem{corollary}[theorem]{Corollary}
\theoremstyle{definition}

\theoremstyle{remark}
\newtheorem{remark}[theorem]{Remark}
\numberwithin{equation}{section}
\allowdisplaybreaks[1]
\title{Finite-time blowup of the IPM equation\\[3pt]from smooth initial data without forcing}
\author{Fan Zheng\\[3pt]
\small Instituto de Ciencias Matem\'aticas\\
\small\href{mailto:fan.zheng@icmat.es}{\texttt{fan.zheng@icmat.es}}}
\date{}
\begin{document}
\maketitle
\begin{abstract}
We prove finite-time gradient blow-up for the
unforced incompressible porous media equation on the two-dimensional
torus, removing the forcing used in the IPM blow-up constructions of
Alp\"oge, Buckmaster, and Coiculescu \cite{lit:ABC} and
C\'ordoba and Mart\'inez-Zoroa \cite{lit:CMIPM}.
The key is an identity that expresses the logarithmic growth of a layer's
gradient as the negative pressure curvature transverse to its transported
phase. A nearly vertical parent layer rotates and amplifies a new phase,
with a gain exponential in the reciprocal parent tilt. An asymmetric
periodic profile gives a uniform upper bound for the pressure Hessian,
preventing rapid decay of the gradient magnitude. Each layer is seeded at time zero;
its phase means are selected by a terminal displacement problem, and a
finite order expansion is corrected to an exact unforced solution.
The scale induction then shows blow-up of two derivative
components at the origin and convergence to a terminal density in every
$C^\eta$, $\eta<1$.
\end{abstract}

\newpage
\begingroup
\small
\tableofcontents
\endgroup
\newpage
\section{Introduction: identities and amplification}
\label{ipm:sec-introduction}

Set $c=2\pi$ and $\mathbb T^2=(\mathbb R/2\pi\mathbb Z)^2$.
Physical vector norms are Euclidean and physical matrix norms are
induced operator norms; their $L^\infty$ norms are spatial suprema.
For mean-zero density, Darcy's law is
\begin{equation}\label{ipm:darcy}
 u(\rho)=-c\rho e_2-\nabla p,\qquad
 \operatorname{div}u(\rho)=0,\qquad \int_{\mathbb T^2}p=0.
\end{equation}
Equivalently, for nonzero $\xi\in\mathbb Z^2$,
\begin{equation}\label{ipm:velocity-symbol}
 \widehat{u(\rho)}(\xi)=\mathfrak m(\xi)\widehat\rho(\xi),
 \qquad
 \mathfrak m(\xi)=c\frac{(\xi_1\xi_2,-\xi_1^2)}{|\xi|^2},
 \qquad \widehat{u(\rho)}(0)=0.
\end{equation}
We seek a solution of
\begin{equation}\label{ipm:unforced}
 \partial_t\rho+u(\rho)\cdot\nabla\rho=0.
\end{equation}
Odd always means odd under the joint spatial reflection $x\mapsto-x$.
Since the Darcy multiplier is even, the velocity of an odd density is odd,
and the origin is a stationary particle.

\begin{theorem}\label{ipm:main}
There are an odd mean-zero $\rho_{\rm in}\in C^\infty(\mathbb T^2)$,
a time $T>0$, and a solution of \eqref{ipm:unforced} in
$C^\infty([0,T)\times\mathbb T^2)$, with $\rho(0)=\rho_{\rm in}$,
such that
\[
 \lim_{t\uparrow T}\partial_2\rho(t,0)=+\infty,
 \qquad
 \lim_{t\uparrow T}\partial_1u_1(\rho)(t,0)=+\infty.
\]
There is a common terminal density $\rho_*$ such that
$\rho(t)\to\rho_*$ in $C^\eta$ for every $0\le\eta<1$.
By replacing $\rho(t,x)$ with $T\rho(Tt,x)$, one can arrange $T=1$.
\end{theorem}

Theorem~\ref{ipm:main} is proved in Theorem~\ref{ipm:cascade} at the end of
Section~\ref{ipm:seventh-physical}.

The construction adds successively finer localized layers to exact smooth
solutions. Each layer is present at time zero with a very small seed.
The ray identity \eqref{ipm:rayidentity} expresses its logarithmic gradient
growth as the negative pressure curvature transverse to the transported
phase. A nearly vertical parent turns the new phase through an angle of
order one and produces a gain exponential in the reciprocal parent tilt.
After normalization to a prescribed target gradient, the upper pressure
bound $\nabla^2p\le I$\footnote{All inequalities between symmetric matrices are quadratic-form inequalities.} makes the layer's earlier history exponentially
small. An asymmetric periodic profile places its large gradient in the
negative part of the pressure Hessian and leaves a summable positive
pressure cost. The child thus retains the pressure bound and reaches the
direction needed to amplify the next layer.

We first derive the increment equation and the leading wave identities,
then construct the phase means and the successive corrections. The means
are selected through a terminal displacement equation, whose backward
energy estimate uses the same upper pressure bound. Correcting the
remaining residual gives an exact unforced child and its particle map.
We then introduce the stage parameters, explain the amplification and
turning clocks, and show how the small history, pressure budget, and
retained time intervals lead to the limiting solution.

\subsection{The increment equation in parent and lifted coordinates}

Let $P$ be an exact smooth parent density, with velocity $u=u(P)$,
pressure $p$, and initial-label flow $X(t,z)$, $X(0,z)=z$.
Write $D_t=\partial_t+u\cdot\nabla$, $r=\nabla P$,
$M=\nabla u$, and $H=\nabla^2p$, with $M_{ij}=\partial_j u_i$.
Differentiating transport~\eqref{ipm:unforced} and Darcy's law~\eqref{ipm:darcy}
gives the exact identities
\begin{equation}\label{ipm:raybasic}
 D_tr=-M^Tr,\qquad M+ce_2r^T=-H,\qquad \operatorname{tr}H=-cr_2.
\end{equation}
The first transports a scalar gradient as a covector. The other two
express the velocity gradient and the trace of the pressure Hessian in
terms of the density gradient.

Choose a label length $\ell$, a large dimensionless frequency $k$, and a
unit vector $m_0$, with $(k/\ell)m_0\in\mathbb Z^2$. Set $z=\ell y$ and
transport the physical phase by the parent:
\begin{equation}\label{ipm:intro-phase}
 \theta(t,X(t,\ell y))=km_0\cdot y,\qquad
 F(t,y)=D_zX(t,\ell y),\qquad
 m(t,y)=F(t,y)^{-T}m_0.
\end{equation}
Thus $D_t\theta=0$ and $(\nabla_x\theta)\circ X=(k/\ell)m$.
Since $F_t=(M\circ X)F$, the normal satisfies $m'=-(M\circ X)^Tm$ along
each parent particle. Throughout, primes on quantities evolving along a parent
trajectory denote physical-time differentiation at fixed parent label.

Write the proposed child as $P^{\rm new}=P+\Theta$, its velocity as
$u+U$, and its pressure as $p+\pi$. By linearity of Darcy's law,
$U=u(\Theta)$, and subtraction of the exact parent equation gives
\begin{equation}\label{ipm:intro-increment}
 D_t\Theta+U\cdot r+U\cdot\nabla\Theta=0,\qquad
 U=-c\Theta e_2-\nabla\pi,\qquad \operatorname{div}U=0.
\end{equation}
To construct all three increments together, now regard $\theta\in\mathbb T$
as an independent angle and set
\begin{align}
 \Theta(t,X(t,\ell y))&=\ell w(t,y,km_0\cdot y),\nonumber\\
 U(t,X(t,\ell y))&=\ell F(t,y)v_{\mathrm{raw}}(t,y,km_0\cdot y),
       \label{ipm:intro-raw-fields}\\
 \pi(t,X(t,\ell y))&=\ell^2q(t,y,km_0\cdot y).\nonumber
\end{align}
Thus $v_{\mathrm{raw}}$ is the velocity increment in rescaled parent labels,
as calculated in \eqref{ipm:label-velocity}--\eqref{ipm:child-flow-identity}.
Define
\[
 G=F^{-1}F^{-T},\qquad a_2=F^{-1}e_2,\qquad
 R_0(y)=\nabla_zP(0,\ell y),\qquad
 \widetilde\nabla=\nabla_y+km_0\partial_\theta.
\]
Transport of $P$ gives $P(t,X(t,z))=P(0,z)$, and then the chain rule gives
$F^T(r\circ X)=R_0$. The graph chain rule and
the volume-preserving Piola identity, as calculated in
\eqref{ipm:physical-piola}, therefore turn
\eqref{ipm:intro-increment} into the sufficient lifted system
\begin{equation}\label{ipm:intro-raw-equation}
 \begin{aligned}
  w_t+v_{\mathrm{raw}}\cdot R_0+v_{\mathrm{raw}}\cdot\widetilde\nabla w&=0,\\
  v_{\mathrm{raw}}&=-ca_2w-G\widetilde\nabla q,\\
  \widetilde\nabla\cdot v_{\mathrm{raw}}&=0.
 \end{aligned}
\end{equation}
It is solved for every independent $(y,\theta)$ and then restricted to
$\theta=km_0\cdot y$. The integer-carrier condition
$n_0=(k/\ell)m_0\in\mathbb Z^2$ makes that graph periodic. Parent transport becomes $\partial_t$ at fixed $y$; each fast
spatial derivative contributes $k\partial_\theta$.

\subsection{Leading density, velocity, pressure, and ray identities}

In this subsection, $r,M,H$ denote the parent fields
$\nabla P,\nabla u,\nabla^2p$ evaluated at $X(t,\ell y)$.
The leading density and velocity increments are
\begin{align}
 \Theta_{\rm lead}(t,X(t,\ell y))
   &=\frac{\ell}{k}\alpha\chi(y)a(t,y)f_\delta(km_0\cdot y),
       \label{ipm:intro-density}\\
 U_{\rm lead}(t,X(t,\ell y))
   &=\frac{\ell}{k}\alpha\chi(y)a(t,y)
                       \mathfrak m(m(t,y))f_\delta(km_0\cdot y).
       \label{ipm:intro-velocity}
\end{align}
Here $\alpha>0$ is a constant seed, $a>0$ is the evolving scalar
amplitude, $\chi$ localizes the layer, and $f_\delta$ is odd and periodic.
The multiplier is homogeneous of degree zero, so it depends on the
direction of $m$, and
\[
 \mathfrak m(m)\cdot m=0.
\]
The same cancellation follows directly from the lifted Darcy equation.
Seek
\[
 w_{\rm lead}=k^{-1} A_1,\qquad
 v_{\mathrm{raw,lead}}=k^{-1} V_1,\qquad
 q_{\rm lead}=k^{-2}\pi_1,\qquad
 A_1=\alpha\chi a f_\delta.
\]
At leading order, the Darcy and incompressibility equations in
\eqref{ipm:intro-raw-equation} give
\begin{equation}\label{ipm:intro-leading-darcy}
 V_1=-ca_2A_1-Gm_0\partial_\theta\pi_1,
 \qquad m_0\cdot V_1=0.
\end{equation}
Taking the dot product of the expression for $V_1$ with $m_0$ gives
\[
 0=-c(m_0\cdot a_2)A_1-(m_0\cdot Gm_0)\partial_\theta\pi_1,
 \qquad \partial_\theta\pi_1=-c\frac{m_0\cdot a_2}{m_0\cdot Gm_0}A_1.
\]
Taking the zero-mean periodic primitive therefore yields
\begin{equation}\label{ipm:intro-leading-pressure}
 \pi_1=-c\alpha\chi a\frac{m_0\cdot a_2}{m_0\cdot Gm_0}\varphi_\delta,
 \qquad \varphi_\delta=\partial_\theta^{-1}f_\delta,
 \qquad \pi_{\rm lead}\circ X=\ell^2q_{\rm lead}^{\rm diag}
       =\frac{\ell^2}{k^2}\pi_1^{\rm diag}.
\end{equation}
Here the superscript ${\rm diag}$ denotes graph evaluation.
Substituting the expression for $\partial_\theta\pi_1$ into
\eqref{ipm:intro-leading-darcy} gives
\[
 V_1=-c\left(a_2-\frac{m_0\cdot a_2}{m_0\cdot Gm_0}Gm_0\right)A_1
     =-c\left(I-\frac{Gm_0\otimes m_0}{m_0\cdot Gm_0}\right)a_2A_1.
\]
Thus $V_1=b_0A_1=-cLa_2A_1$, where
\begin{equation}\label{ipm:intro-projection}
 b_0=-cLa_2,\qquad L=I-\frac{Gm_0\otimes m_0}{m_0\cdot Gm_0}.
\end{equation}
Since $a_2=F^{-1}e_2$, $G=F^{-1}F^{-T}$, and $m=F^{-T}m_0$,
\[
 m_0\cdot a_2=(F^{-T}m_0)\cdot e_2=m\cdot e_2=m_2,
 \qquad m_0\cdot Gm_0=|m|^2.
\]
Using also $Fa_2=e_2$ and $FGm_0=F^{-T}m_0=m$, we obtain
$Fb_0=-ce_2+cm_2m/|m|^2=\mathfrak m(m)$, which proves
\eqref{ipm:intro-velocity}.
The fast quadratic term at order $k^{-1}$ is
$(m_0\cdot V_1)\partial_\theta A_1=0$. The scalar equation at that
order is therefore $(A_1)_t+c_A^{\ell}A_1=0$, where
\begin{equation}\label{ipm:amplitude-coefficient-definition}
 c_A^{\ell}(t,y):=c_A(t,\ell y)
 =b_0(t,y)\cdot R_0(y)
 =\mathfrak m(m(t,y))\cdot r(t,X(t,\ell y)).
\end{equation}
Here $c_A$ denotes the same coefficient in unscaled labels. This gives
\begin{equation}\label{ipm:intro-wave}
 \frac{a'}a=-\mathfrak m(m)\cdot r,
 \qquad a(s,y)=1,
\end{equation}
where $s$ is the activation time. Together with the phase evolution
derived after \eqref{ipm:intro-phase}, this defines the wave on the
whole parent interval, including times before $s$. The slow derivatives,
Darcy errors, and nonlinear residual of the leading increments are
handled by the corrections in Section~\ref{ipm:sec-exact}.

Set $n=m/|m|$ and $n^\perp=(-n_2,n_1)$. The phase evolution and
equations \eqref{ipm:raybasic} and \eqref{ipm:intro-wave} give
\begin{align*}
 (\log|m|)'&=-n^TMn=n^THn+cn_2(r\cdot n),\\
 (\log a)'&=c(n_1^2r_2-n_1n_2r_1).
\end{align*}
Adding them cancels the terms involving $r_1$; the remaining density
term is $cr_2=-\operatorname{tr}H$. Hence
\begin{equation}\label{ipm:rayidentity}
 (\log(a|m|))'=n^THn+cr_2
       =-(n^\perp)^THn^\perp.
\end{equation}
Thus negative pressure curvature in $m^\perp$ produces growth of the gradient size $a|m|$. The upper bound $H\le I$ gives
$(\log(a|m|))'\ge-1$, controlling the layer's earlier history by
\eqref{ipm:backwardray}.

Differentiating \eqref{ipm:intro-density}, \eqref{ipm:intro-velocity}, and
\eqref{ipm:intro-leading-pressure}, respectively, and retaining the leading
fast derivatives gives, after composition with $X(t,\ell y)$,
\begin{align}
 (\nabla\Theta)_{\rm lead}\circ X
   &=\alpha\chi a\,m f'_\delta(km_0\cdot y),\nonumber\\
 (\nabla U)_{\rm lead}\circ X
   &=\alpha\chi a\,\mathfrak m(m)\otimes m\,f'_\delta(km_0\cdot y),
       \label{ipm:intro-leading-matrices}\\
 (\nabla^2\pi)_{\rm lead}\circ X
   &=-c\alpha\chi a\,m_2\frac{m\otimes m}{|m|^2}
                                      f'_\delta(km_0\cdot y).\nonumber
\end{align}
Here $(\nabla_x\theta)\circ X=(k/\ell)m$ and
$\varphi_\delta''=f_\delta'$; derivatives of the slow coefficients
contribute lower-order terms.
Hereafter all parent densities $P$ are odd, as well as their velocities
$u(P)$ and initial-label flows $X(t,z)$, so the origin is a stationary
particle. For any layer, its physical phase satisfies
$\theta(t,0)=(k/\ell)m_0\cdot X^{-1}(t,0)=0$. With
$\chi(0)=1$ and $f'_\delta(0)=\delta^{-1}$, define its leading
density-gradient magnitude at the origin by
\begin{equation}\label{ipm:intro-central-magnitude}
 g(t)=\frac{\alpha}{\delta}a(t,0)|m(t,0)|.
\end{equation}
At the origin the three expressions in
\eqref{ipm:intro-leading-matrices} become
\begin{equation}\label{ipm:intro-central-matrices}
 gn,\qquad g\mathfrak m(n)\otimes n=-cg n_1 n^\perp\otimes n,\qquad
                  -cg n_2 n\otimes n.
\end{equation}
These are the leading density-gradient, velocity-gradient, and
pressure-Hessian increments that the layer supplies to the next stage.

\subsection{The divergence correction and the phase means}

Transversality removes the fast divergence of $V_1$ but leaves its slow
divergence. In two dimensions this is corrected by a stream function.
For any transverse, zero-angle-mean coefficient $V_p$, set
\begin{equation}\label{ipm:intro-curl}
 \psi_p=\partial_\theta^{-1}(m_0^\perp\cdot V_p),\qquad
 C_p=\nabla_y^\perp\psi_p,
 \qquad D=k^{-1}\nabla_y+m_0\partial_\theta=k^{-1}\widetilde\nabla.
\end{equation}
Then $V_p=m_0^\perp\partial_\theta\psi_p$, and
\[
 V_p+k^{-1} C_p=D^\perp\psi_p,\qquad
 D\cdot(V_p+k^{-1} C_p)=0.
\]
The coefficients of $D$ are constant, so this divergence-of-curl identity
holds on the whole lifted space. Throughout, $\nabla^\perp=(-\partial_2,\partial_1)$.

Products of oscillations also generate angle-independent density
coefficients $B_p$. Define the mean Darcy velocity operator
$\mathcal V_{\mathrm{mean}}$ on a phase-independent density $B$ by
\begin{equation}\label{ipm:intro-mean-darcy}
 \mathcal V_{\mathrm{mean}}B:=-ca_2B-G\nabla_y\bar q,\qquad
                    \operatorname{div}_y(\mathcal V_{\mathrm{mean}}B)=0.
\end{equation}
Here $\bar q$ solves the periodic elliptic equation
$\operatorname{div}_y(G\nabla_y\bar q)=-c\operatorname{div}_y(a_2B)$.
The right-hand side has zero integral, and uniform positivity of $G$
gives existence and uniqueness up to an additive constant; hence its
gradient is unique. Thus $\mathcal V_{\mathrm{mean}}B_p$
is the mean velocity associated with $B_p$, with pressure $\bar q_p$.
Set $B_1=\bar q_1=C_0=\pi_0=0$. At finite depth $N$ use
\begin{align}
 w^a&=\sum_{p=1}^Nk^{-p}(A_p+B_p),\nonumber\\
 v_{\mathrm{raw}}^a&=\sum_{p=1}^N
       \bigl[k^{-p}(V_p+\mathcal V_{\mathrm{mean}}B_p)+k^{-p-1}C_p\bigr],
       \label{ipm:intro-expansion}\\
 q^a&=\sum_{p=1}^N
       \bigl[k^{-p}\bar q_p+k^{-p-1}\pi_p\bigr].\nonumber
\end{align}
Every finite velocity sum is exactly $D$-divergence free. In particular,
the last curl term $k^{-N-1}C_N$ must be retained.

Write $\langle\cdot\rangle$ for the normalized angular average.
At order $p$, the preceding pressure and curl coefficients are known.
The oscillatory (angle-dependent) part of the Darcy equation in
\eqref{ipm:intro-raw-equation} at this order, with the transverse choice of $V_p$, is
\begin{equation}\label{ipm:intro-order-p-darcy}
 V_p+C_{p-1}=-ca_2A_p-G\nabla_y\pi_{p-1}
                         -Gm_0\partial_\theta\pi_p,
 \qquad m_0\cdot V_p=0.
\end{equation}
Put $J_{p-1}=G\nabla_y\pi_{p-1}+C_{p-1}$. Taking the dot product
with $m_0$ gives
\[
 0=m_0\cdot(-ca_2A_p-J_{p-1})
                 -(m_0\cdot Gm_0)\partial_\theta\pi_p.
\]
Thus the zero-mean pressure is determined by
\[
 \partial_\theta\pi_p
   =\frac{m_0\cdot(-ca_2A_p-J_{p-1})}{m_0\cdot Gm_0},
 \qquad \langle\pi_p\rangle=0.
\]
Substituting this pressure derivative into \eqref{ipm:intro-order-p-darcy} gives
\[
 \begin{aligned}
 V_p&=-ca_2A_p-J_{p-1}
       -Gm_0\frac{m_0\cdot(-ca_2A_p-J_{p-1})}{m_0\cdot Gm_0}\\
    &=L(-ca_2A_p-J_{p-1})=b_0A_p-LJ_{p-1}.
 \end{aligned}
\]
Thus the reconstruction is
\begin{align}
 \partial_\theta\pi_p
   &=\frac{m_0\cdot(-ca_2A_p-J_{p-1})}{m_0\cdot Gm_0},
       \qquad \langle\pi_p\rangle=0,\nonumber\\
 V_p&=b_0A_p-LJ_{p-1}.
       \label{ipm:intro-reconstruction}
\end{align}
Afterwards \eqref{ipm:intro-curl} defines $\psi_p,C_p$.
The whole velocity
coefficient at order $k^{-p}$ is
\[
 v_{\mathrm{raw},p}=V_p+\mathcal V_{\mathrm{mean}}B_p+C_{p-1}
     =b_0A_p+\mathcal V_{\mathrm{mean}}B_p+E_p,\qquad
 E_p=-LJ_{p-1}+C_{p-1}.
\]
With $f_p=A_p+B_p$, the slow and fast quadratic terms in
$v_{\mathrm{raw}}\cdot\widetilde\nabla w$ at that order are
\[
 \sum_{i+j=p}v_{\mathrm{raw},i}\cdot\nabla_y f_j
       +\sum_{i+j=p+1}(m_0\cdot v_{\mathrm{raw},i})\partial_\theta A_j.
\]
The term with $i=1$ in the fast sum vanishes because
$m_0\cdot v_{\mathrm{raw},1}=m_0\cdot V_1=0$. Its only current-order
unknown is $(m_0\cdot \mathcal V_{\mathrm{mean}}B_p)\partial_\theta A_1$, which has zero
angle mean. Collect the known terms into
\begin{equation}\label{ipm:intro-known-source}
 \begin{aligned}
 Z_p={}&E_p\cdot R_0+\sum_{i+j=p}v_{\mathrm{raw},i}\cdot\nabla_y f_j\\
 &+\sum_{\substack{i+j=p+1\\2\le i\le p-1}}
         (m_0\cdot \mathcal V_{\mathrm{mean}}B_i)\partial_\theta A_j
  +\sum_{\substack{i+j=p+1\\2\le i\le p}}
         (m_0\cdot C_{i-1})\partial_\theta A_j.
 \end{aligned}
\end{equation}
All unqualified indices are positive. Matching the scalar equation is
now a linear triangular sequence:
\begin{align}
 (B_p)_t+(\mathcal V_{\mathrm{mean}}B_p)\cdot R_0&=-\langle Z_p\rangle,
        \label{ipm:intro-mean-step}\\
 (A_p)_t+c_A^{\ell}A_p&=-Z_p+\langle Z_p\rangle
                   -(m_0\cdot \mathcal V_{\mathrm{mean}}B_p)\partial_\theta A_1,
        \qquad A_p(0)=0\quad(p\ge2).
        \label{ipm:intro-oscillatory-step}
\end{align}
First solve for the mean density $B_p$, with its mean velocity
$\mathcal V_{\mathrm{mean}}B_p$ and mean pressure $\bar q_p$ determined by
\eqref{ipm:intro-mean-darcy}. Then solve for the oscillatory density $A_p$,
determine the oscillatory pressure $\pi_p$ and velocity $V_p$ from
\eqref{ipm:intro-reconstruction}, and finally form the stream function
$\psi_p$ and velocity correction $C_p$ using \eqref{ipm:intro-curl}.
The initial values of $B_p$ are selected by the inverse
described in Section~\ref{ipm:intro-mean-inverse}. They need not vanish.

The mean source has the required spatial compatibility. For each $i$,
the oscillatory velocity $V_i$ and its correction $k^{-1}C_i$ form the
divergence-free pair $V_i+k^{-1}C_i=D^\perp\psi_i$ from
\eqref{ipm:intro-curl}. Thus
\[
 \operatorname{div}_yV_i+m_0\cdot\partial_\theta C_i=0,
 \qquad \operatorname{div}_yC_i=0.
\]
Since $v_{\mathrm{raw},i}=V_i+\mathcal V_{\mathrm{mean}}B_i+C_{i-1}$,
and $\operatorname{div}_y(\mathcal V_{\mathrm{mean}}B_i)=0$ by
\eqref{ipm:intro-mean-darcy}, we obtain
\[
 \operatorname{div}_yv_{\mathrm{raw},i}
 =\operatorname{div}_yV_i=-m_0\cdot\partial_\theta C_i.
\]
In \eqref{ipm:intro-known-source}, $\langle E_p\rangle=0$, and the
terms $(m_0\cdot\mathcal V_{\mathrm{mean}}B_i)\partial_\theta A_j$ have zero angle mean
because this velocity is angle independent and
$\langle\partial_\theta A_j\rangle=0$. Using
$\partial_\theta A_j=\partial_\theta f_j$, integrating by parts in
$\theta$, and shifting the fast index from $i$ to $i+1$ gives
\begin{equation}\label{ipm:intro-mean-flux}
 \begin{aligned}
 \langle Z_p\rangle
   &=\sum_{i+j=p}\langle v_{\mathrm{raw},i}\cdot\nabla_y f_j\rangle
     +\sum_{\substack{i+j=p+1\\2\le i\le p}}
        \langle(m_0\cdot C_{i-1})\partial_\theta A_j\rangle\\
   &=\sum_{i+j=p}\Bigl\langle
        v_{\mathrm{raw},i}\cdot\nabla_y f_j
        -(m_0\cdot\partial_\theta C_i)f_j
      \Bigr\rangle\\
   &=\sum_{i+j=p}\Bigl\langle
        v_{\mathrm{raw},i}\cdot\nabla_y f_j
        +(\operatorname{div}_yv_{\mathrm{raw},i})f_j
      \Bigr\rangle\\
   &=\operatorname{div}_y\sum_{i+j=p}\langle v_{\mathrm{raw},i} f_j\rangle.
 \end{aligned}
\end{equation}
Its spatial integral is zero. This uses the cancellation between the
slow and fast terms; it does not assign zero integral to each product.

\subsection{The mean inverse and the selected initial density}
\label{ipm:intro-mean-inverse}

The forced construction of Alp\"oge, Buckmaster, and Coiculescu
\cite{lit:ABC} amplifies successively finer localized layers. Its time cutoff
and retained parametrix residual contribute to the force. Here the small
oscillatory seed and the phase-independent corrections have to be both
included in the child's initial density due to smooth uniqueness. The unforced Euler construction
\cite{lit:OpenAI,lit:ZhengEuler} motivates this choice of initial data followed by an
exact residual correction.

To solve \eqref{ipm:intro-mean-step} at a fixed order $p$, define the
physical density $\varrho$ and source $f$ by
\[
 \varrho(t,X(t,\ell y))=\ell B_p(t,y),\qquad
 f(t,X(t,\ell y))=-\ell\langle Z_p\rangle(t,y).
\]
Here $r,M,H$ are Eulerian parent fields, as in Section~1.1.
Comparing physical Darcy's law \eqref{ipm:darcy} with
\eqref{ipm:intro-mean-darcy}, using
$\varrho(t,X(t,\ell y))=\ell B_p(t,y)$, gives
\[
 u(\varrho)(t,X(t,\ell y))=\ell F(t,y)\mathcal V_{\mathrm{mean}}B_p(t,y).
\]
Transport of the parent density $P$ gives
$P(t,X(t,\ell y))=P(0,\ell y)$; the chain rule then yields
\[
 F(t,y)^T r(t,X(t,\ell y))=R_0(y),\qquad
 r(t,X(t,\ell y))=F(t,y)^{-T}R_0(y).
\]
Consequently, \eqref{ipm:intro-mean-step} translates into
\[
 \begin{aligned}
 (D_t\varrho+u(\varrho)\cdot r)(t,X(t,\ell y))
   &=\ell\bigl[(B_p)_t+(\mathcal V_{\mathrm{mean}}B_p)\cdot R_0\bigr](t,y)\\
   &=-\ell\langle Z_p\rangle(t,y)=f(t,X(t,\ell y)),
 \end{aligned}
\]
that is, the full linearized IPM equation $D_t\varrho+u(\varrho)\cdot r=f$.
To construct its inverse, introduce a displacement $\eta$ with terminal
condition $\eta(S)=0$ at the retained endpoint $S$, and set
\begin{equation}\label{ipm:intro-displacement-definition}
 D_t\eta-M\eta=u(\varrho),
 \qquad Q=\varrho+r\cdot\eta.
\end{equation}
The pressure $\pi$ associated with $\varrho$ is related to the mean
pressure $\bar q_p$ by
\[
 \pi(t,X(t,\ell y))=\ell^2\bar q_p(t,y),\qquad
 (\nabla_x\pi)(t,X(t,\ell y))
     =\ell F(t,y)^{-T}\nabla_y\bar q_p(t,y),
\]
where the additive function of time in $\pi$ is chosen to make the
first equality exact. Thus \eqref{ipm:intro-mean-darcy} is precisely
the physical Darcy law
\[
 u(\varrho)=-c\varrho e_2-\nabla\pi.
\]
By \eqref{ipm:raybasic}, $D_tr=-M^Tr$ and $M+H=-ce_2r^T$, so
\[
 \begin{aligned}
 D_tQ
   &=D_t\varrho-(M^Tr)\cdot\eta+r\cdot(M\eta+u(\varrho))
     =D_t\varrho+u(\varrho)\cdot r=f,\\
 D_t\eta+H\eta+\nabla\pi
   &=(M+H)\eta-c\varrho e_2
     =-c(\varrho+r\cdot\eta)e_2=-cQe_2.
 \end{aligned}
\]
To justify the divergence constraint, the Leibniz rule, with $M=\nabla u$ and $\operatorname{div}u=0$, gives
\[
 \begin{aligned}
 \operatorname{div}(D_t\eta)
   &=D_t(\operatorname{div}\eta)+\operatorname{tr}(M\nabla\eta),\\
 \operatorname{div}(M\eta)
   &=\eta\cdot\nabla(\operatorname{div}u)
       +\operatorname{tr}(M\nabla\eta)
     =\operatorname{tr}(M\nabla\eta).
 \end{aligned}
\]
Subtracting cancels the trace terms:
\[
 D_t(\operatorname{div}\eta)
   =\operatorname{div}(D_t\eta-M\eta)
   =\operatorname{div}u(\varrho)=0.
\]
Thus $\operatorname{div}\eta$ is transported by the parent flow, and
$\eta(S)=0$ implies $\operatorname{div}\eta=0$ throughout the interval.
Hence the displacement system is
\begin{equation}\label{ipm:intro-displacement}
 D_tQ=f,\qquad D_t\eta+H\eta+\nabla\pi=-cQe_2,
 \qquad\operatorname{div}\eta=0.
\end{equation}
For the inverse, first solve $D_tQ=f$ forward with $Q(0)=0$, then solve
the displacement equation backward from $\eta(S)=0$, using
$\operatorname{div}\eta=0$ as the constraint determining $\pi$.
Recover the density and velocity by
\[
 \varrho=Q-r\cdot\eta,\qquad
 u(\varrho)=D_t\eta-M\eta=-c\varrho e_2-\nabla\pi.
\]
The divergence identity above gives
$\operatorname{div}(D_t\eta-M\eta)=D_t(\operatorname{div}\eta)=0$,
so the recovered velocity satisfies Darcy's law and incompressibility.
The identity for $D_tQ$ then gives the linearized density equation.
The initial density is selected by
$\varrho(0)=-r(0)\cdot\eta(0)$. Pairing the displacement equation with
$\eta$ gives
\begin{equation}\label{ipm:intro-displacement-energy}
 \frac12\frac{d}{dt}\|\eta\|_2^2
   =-\int\eta^TH\eta-c\int Q\eta_2.
\end{equation}
The upper bound $H\le I$ gives
$-\int\eta^TH\eta\ge-\|\eta\|_2^2$. Integrating the energy inequality
backward from $\eta(S)=0$ controls $\eta$, then
$\varrho=Q-r\cdot\eta$, and hence $B_p$, throughout $[0,S]$.
In parent coordinates, this gives
Proposition~\ref{ipm:meaninverse}, with polynomial cost in the parent
regularity bounds. This estimate handles the forced phase-mean equation \eqref{ipm:intro-mean-step}.
The earlier history of the leading wave is controlled separately by the
pointwise ray identity \eqref{ipm:rayidentity}, as explained below.
Both backward estimates follow directly from evolution identities;
the IPM argument here uses no variational minimization.

Products of oscillations generate phase means, so the coefficient
construction solves a mean equation first at each order, then a supported
oscillatory equation, and finally reconstructs the velocity and pressure.
The mean inverse, the scalar inverse, and the lifted Darcy inverse all
preserve one common Gevrey growth constant, even when the input has an
arbitrarily large factorial shift. This is the quantitative property
needed to repeat the inversion many times.

\subsection{The residual and its exact correction}

The divergence of $v_{\mathrm{raw}}^a$ vanishes exactly, but the finite reconstruction
leaves one Darcy tail:
\begin{equation}\label{ipm:intro-darcy-tail}
 v_{\mathrm{raw}}^a+ca_2w^a+G\widetilde\nabla q^a
       =k^{-N-1}H_N,\qquad H_N=C_N+G\nabla_y\pi_N.
\end{equation}
The symbol $H_N$ here denotes this vector tail, not the parent's Hessian
$H$. To remove it, let $\mathcal H_D$ be the closed space of periodic
$D$-gradients, let $\Pi_D$ be its orthogonal projection, and define
\[
 \mathcal T_G=\Pi_DG|_{\mathcal H_D},\qquad
 \mathcal P_G=I-G\mathcal T_G^{-1}\Pi_D.
\]
The uniform positivity of $G=F^{-1}F^{-T}$ makes $\mathcal T_G$
coercive and invertible on vector gradients.
The projection $\mathcal P_G$ fixes $D$-divergence-free vectors and
annihilates $G\mathcal H_D$. Define the exact lifted Darcy operator by
\begin{equation}\label{ipm:intro-exact-darcy}
 \mathcal V_D A=-ca_2A-G\boldsymbol p,\qquad
 D\cdot \mathcal V_D A=0,\qquad \boldsymbol p\in\mathcal H_D.
\end{equation}
As in the pressure construction after \eqref{ipm:intro-mean-darcy},
uniform positivity of $G$ gives existence and uniqueness: the divergence
constraint is equivalent to
$\mathcal T_G\boldsymbol p=-c\Pi_D(a_2A)$, which uniquely determines
$\boldsymbol p\in\mathcal H_D$ and hence $\mathcal V_D A$.
The subscript $D$ records the dependence on
$D=k^{-1}\nabla_y+m_0\partial_\theta$.
For angle-independent $B$, this operator reduces to the mean Darcy
operator: $\mathcal V_DB=\mathcal V_{\mathrm{mean}}B$.
Applying $\mathcal P_G$ to \eqref{ipm:intro-darcy-tail} gives
$\mathcal V_D w^a-v_{\mathrm{raw}}^a=-k^{-N-1}\mathcal P_GH_N$.
Thus Darcy's law can be restored before estimating the scalar residual.

For an exact solution of the lifted system \eqref{ipm:intro-raw-equation},
set $A=kw$ and define the normalized velocity by
\[
 v=kv_{\mathrm{raw}}=\mathcal V_D A,
 \qquad v_{\mathrm{raw}}=k^{-1} v.
\]
Throughout the layer construction, $v$ denotes this normalized velocity;
the raw velocity is always written $v_{\mathrm{raw}}$.
For the finite approximation put $A^a=kw^a$.
Its normalized finite velocity $kv_{\mathrm{raw}}^a$ differs from
$\mathcal V_D A^a$ by the Darcy tail just computed.
The exact normalized scalar equation and the residual of the approximation are
\begin{align}
 A_t+\mathcal V_D A\cdot R_0+\mathcal V_D A\cdot DA&=0,
       \label{ipm:intro-normalized-equation}\\
 r^a=(A^a)_t+\mathcal V_D A^a\cdot R_0
                      +\mathcal V_D A^a\cdot DA^a.&
       \label{ipm:intro-residual}
\end{align}
Each step of the coefficient recursion loses a fixed number of derivatives,
so truncation at order $N$ incurs a factorial shift of order $N$.
In Gevrey class $2$, this gives bounds of the form
\[
 \|\text{coefficient at order }p\|_\rho
    \le(\mathcal P^C N^C)^{p+1},\qquad 1\le p\le N,
\]
at a radius $\rho$ independent of $N$ and bounded below by a fixed
inverse power of $\mathcal P$. Here $\mathcal P$ is the input size
parameter defined in \eqref{ipm:packet-parameter}; the parent density
is denoted by $P$. Each new order also brings an inverse power of $k$.
The coefficient equations cancel all raw scalar orders through $N$,
so the raw residual starts at $k^{-N-1}$. Restoring Darcy's law adds
a tail of the same order, and the normalization $A=kw$ multiplies
the raw residual by $k$. Accounting for the remaining products and
their number gives \eqref{ipm:residual-before-separation}:
\[
 \|r^a\|_\rho\le\mathcal P^C k^{-N}\mathcal B_N^{N+6},
 \qquad \mathcal B_N\le\mathcal P^C N^C,
\]
where $\mathcal B_N$ is the coefficient majorant defined in
\eqref{ipm:expanded-hierarchy-BN}. Choosing
\[
 N=\lfloor k^\vartheta\rfloor,\qquad
 \mathcal P^Q\le k^{\vartheta/100}
\]
with $\vartheta>0$ sufficiently small and $Q$ sufficiently large gives
$\mathcal B_N\le k^{1/50}$ by \eqref{ipm:expanded-hierarchy-ratio-small}.
The inverse powers of $k$ therefore dominate the derivative losses:
\[
 \|r^a\|_\rho
 \le\mathcal P^C k^{-N+(N+6)/50}
 \le k^{-0.9N}
 \le\exp(-0.6k^\vartheta\log k)
\]
for sufficiently large $k$, since $N\ge0.9k^\vartheta$.

Write $A=A^a+e$. Subtraction gives the exact error equation
\begin{equation}\label{ipm:intro-exact-error}
 \begin{aligned}
 e_t+b\cdot\nabla_{y,\theta}e
      +\mathcal V_D e\cdot(R_0+DA^a)&=-r^a,\qquad e(0)=0,\\
 b&=(k^{-1} \mathcal V_D(A^a+e),\,
                       m_0\cdot \mathcal V_D(A^a+e)).
 \end{aligned}
\end{equation}
The vector field $b$ is divergence free. In its approximate part, the
first two components have an explicit factor $k^{-1}$, and the leading
angle component vanishes because $m_0\cdot V_1=0$. Consequently its
Gevrey norm is $O(k^{-1}\mathcal P^C)$. The remaining linear term in
\eqref{ipm:intro-exact-error} has no derivative on $e$.

This small transport controls the loss of Gevrey radius. The residual
also absorbs the polynomial-coefficient Gronwall factor
$\exp(C\mathcal P^C S)$: the frequency separation
\eqref{ipm:correction-frequency-separation}, with $Q$ large enough, gives $\mathcal P^C\ll N$. Proposition~\ref{ipm:correction} constructs an exact
correction with size at most
$\mathcal P^C\exp(-k^\vartheta/2)$ on the whole retained interval.
Since $e(0)=0$, the finite hierarchy has already fixed all of the child's
initial data before this correction is made.

\subsection{The physical solution and its particle map}

Use the normalized velocity $v=\mathcal V_D A$, so that
$v_{\mathrm{raw}}=k^{-1} v$. Since $\boldsymbol p\in\mathcal H_D$
and $n_0=(k/\ell)m_0\in\mathbb Z^2$, we can recover a smooth periodic scalar
$\zeta$ with $D\zeta=\boldsymbol p$. Indeed, at Fourier frequency
$(\nu,j)\in\mathbb Z^2\times\mathbb Z$, the multiplier of $D$ is
$i\eta$, where $\eta=(\ell/k)(\nu+jn_0)$. Thus $|\eta|\ge\ell/k$ whenever
$\eta\ne0$, so recovering the potential via
$\widehat\zeta(\nu,j)=\eta\cdot\widehat{\boldsymbol p}(\nu,j)/(i|\eta|^2)$
for $\eta\ne0$, with $\widehat\zeta(\nu,j)=0$ for $\eta=0$,
preserves smoothness.

The exact physical increments are
\begin{equation}\label{ipm:intro-exact-physical}
 \Theta\circ X=\ell k^{-1} A^{\rm diag},\qquad
 U\circ X=\ell k^{-1} Fv^{\rm diag},\qquad
 \pi\circ X=\ell^2k^{-2}\zeta^{\rm diag},
\end{equation}
where all compositions use $X(t,\ell y)$. The graph chain rule proves
Darcy's law, incompressibility, and the scalar equation for $P+\Theta$.
Joint oddness under $(y,\theta)\mapsto(-y,-\theta)$ gives zero physical
mean and keeps the origin stationary.

The same normalized velocity produces the actual child flow. Let
$\Phi$ solve
\[
 \dot y=k^{-1} v(t,y,\theta),\qquad
 \dot\theta=m_0\cdot v(t,y,\theta),\qquad \Phi(0)=\mathrm{id}.
\]
Then $\theta-km_0\cdot y$ is constant along trajectories. With
$\pi_y$ denoting the first two components, set
\begin{equation}\label{ipm:intro-child-flow}
 Y(t,z)=\ell\pi_y\Phi(t,z/\ell,(k/\ell)m_0\cdot z),\qquad
 X^{\rm new}(t,z)=X(t,Y(t,z)).
\end{equation}
Differentiation gives
$\partial_tX^{\rm new}=(u+U)(t,X^{\rm new})$, see
\eqref{ipm:child-flow-identity}. Also,
$\operatorname{div}_Dv=0$ implies $\det D_zY=1$.
The small lifted transport bounds $Y$ and its first two time derivatives;
composition with the known parent map gives polynomial Gevrey bounds for
$X^{\rm new}-\mathrm{id}$, $X^{\rm new}_t$, and $X^{\rm new}_{tt}$.
Together with the initial density bound, these are precisely the inputs
recovered for the next stage by Proposition~\ref{ipm:packet}.

\subsection{Successive increments and their parameters}

The preceding construction gives one exact child and the bounds needed
to use it as the next parent. We now describe the successive stages.
Start from the exact stationary plane-wave parent constructed in
Section~\ref{ipm:sec-base}. Successive exact densities telescope as
\[
 P^L=P^{J-1}+\sum_{j=J}^L\Theta_j,\qquad
 \rho_{\rm in}=P^{J-1}(0)+\sum_{j=J}^\infty\Theta_j(0).
\]
The convergence of these series is established after the stage estimates.
At stage $j$, the preceding leading density-gradient increment at the
origin has target magnitude $h_{j-1}$ and angle $\pi/2+1/x_{j-1}$. The new layer activates at $t_{j-1}$ and
reaches angle $\pi/2+1/x_j$ at $t_j$. Its positive seed is normalized by
\[
 \frac{\alpha_j}{\delta_j}a_j(t_j,0)|m_j(t_j,0)|=h_j.
\]
The parameter choice is
\begin{align}
 x_j&=j^2x_{j-1},\qquad h_j=\exp(x_{j-1}/j^5),\nonumber\\
 \widehat k_j&=\exp(x_{j-1}/j^2),\qquad
 \ell_j=\exp(-x_{j-1}/j^{7/2}),\qquad
 \delta_j=\exp(-x_{j-1}/j^3).
       \label{ipm:intro-scales}
\end{align}
To satisfy the integer-carrier condition, the actual $k_j$ and $m_{0,j}$
are rounded, which introduces a small error in the perpendicularity
condition at activation; Lemma~\ref{ipm:amplification} allows this error. The gain $e^{c_0x_{j-1}}$ dominates both the
prescribed target size and every fixed derivative cost in the initial
oscillation. The much larger new frequency pays for the coefficient
hierarchy and exact correction, while the narrow label support controls
variation of the parent coefficients. The phase means have no fast
angular derivatives and begin at order $k_j^{-2}$, which makes their
initial series summable as well.

All fixed analytic exponents are chosen first, then a sufficiently large
starting index $J$, and finally a large base parameter $x_{J-1}$.
The following frozen calculation explains the clocks behind this choice;
Lemma~\ref{ipm:scales} verifies all inequalities simultaneously.

\subsection{The amplification and turning clocks}

We first freeze the leading parent geometry at the origin. Let
\[
 n=(\cos\phi,\sin\phi),\qquad
 \phi=\frac\pi2+\frac1x,\qquad x\gg1,
\]
and suppose the parent has constant leading density-gradient magnitude $G>0$.
In this frozen calculation, $G$ denotes that scalar magnitude.
The model retains precisely the rank-one fields from
\eqref{ipm:intro-central-matrices}:
\begin{equation}\label{ipm:intro-frozen-parent}
 r=Gn,\qquad M=-cG\cos\phi\,n^\perp\otimes n,
 \qquad H=-cG\sin\phi\,n\otimes n.
\end{equation}
These equalities define the frozen model. For the actual parent, an older
solution and the preceding layer's correction provide controlled
remainders, and $G,\phi$ vary slowly.

Write the new unit normal as $N=(\cos\beta,\sin\beta)$.
At the activation time $s$, choose $N\perp n$, in the
orientation $\beta(s)=\phi-\pi/2=1/x$. Its normal equation and
\eqref{ipm:rayidentity} give, in the frozen model,
\begin{align}
 \beta'&=-(N^\perp)^TM^TN=cG\sin(1/x)\sin^2(\beta-\phi),
       \label{ipm:intro-angular-speed}\\
 (\log(a|m|))'&=-(N^\perp)^THN^\perp
       =cG\sin\phi\,\sin^2(\beta-\phi).
       \label{ipm:intro-angular-growth}
\end{align}
The shear rotates the new normal counterclockwise from a nearly
horizontal direction. As it approaches the parent normal, both the
angular speed and the growth rate decrease by the same factor
$\sin^2(\beta-\phi)$. Their ratio therefore stays large:
\begin{equation}\label{ipm:intro-growth-per-angle}
 \frac{d\log(a|m|)}{d\beta}=-\tan\phi=\cot(1/x)\sim x.
\end{equation}
A turn through an angle of order one gives an exponential gain in $x$.

The time scale is just as explicit. Put
\[
 z=-\cot(\beta-\phi),\qquad z(s)=0.
\]
Equation~\eqref{ipm:intro-angular-speed} reduces to $z'=cG\sin(1/x)$.
Choose a larger next parameter $x_{\rm new}=j^2x$ and stop at
$\beta(t_*)=\pi/2+1/x_{\rm new}$. Then
\begin{align}
 z(t_*)&=\cot(1/x-1/x_{\rm new}),\nonumber\\
 t_*-s&=\frac{\cot(1/x-1/x_{\rm new})}{cG\sin(1/x)}
       =\frac{x^2}{cG}\bigl[1+O(j^{-2}+x^{-2})\bigr],
       \label{ipm:intro-frozen-time}\\
 \log\frac{a(t_*)|m(t_*)|}{a(s)|m(s)|}
       &=\cot(1/x)\left(\frac\pi2-\frac1x+
                                      \frac1{x_{\rm new}}\right)
         =\frac\pi2x+O(1).
       \label{ipm:intro-frozen-gain}
\end{align}
Thus a parent with gradient size $G$ produces a gain comparable to
$e^{c_0x}$ during a physical time comparable to $x^2/G$. The new
normal ends just beyond vertical, with the smaller tilt $1/x_{\rm new}$.
After its amplitude is normalized to a larger target magnitude, this
new layer supplies the same geometry for the next amplification.

Lemma~\ref{ipm:amplification} shows that these
time and gain estimates persist for the actual parent. Let $E$ bound
the density-gradient, velocity-gradient, and Hessian remainders,
together with $|\phi'|$ and $|(\log G)'|$. It suffices that
$Ej^2x^3/G$ and the angular error at activation caused by carrier rounding
be sufficiently small.
No further spatial derivatives of these remainder fields are needed.
At the target, $|\beta-\phi|=1/x-1/x_{\rm new}\sim x^{-1}$,
so \eqref{ipm:intro-angular-speed} gives
$\beta'=cG\sin(1/x)\sin^2(\beta-\phi)\sim G/x^3$.
Thus $Ej^2x^3/G\ll1$ makes the perturbation small even relative to
this slower angular speed, ensuring that the target crossing persists.

\subsection{The pressure bound and the small earlier history}
\label{ipm:intro-history}

The construction maintains an upper bound $H\le I$ globally, which
translates to the lower bound $(\log(a|m|))'\ge-1$ at each fixed
parent label via \eqref{ipm:rayidentity}.
Integrating from $t$ to the activation time $s$ yields the exact
pointwise history estimate
\begin{equation}\label{ipm:backwardray}
 a(t)|m(t)|\le e^{s-t}a(s)|m(s)|\qquad(t\le s).
\end{equation}

Normalize a new layer by $g(t_*)=h$. If its activation-to-target gain
is at least $e^{c_0x}$, \eqref{ipm:backwardray} gives, for $t\le s$ on
an interval of length at most one,
\begin{equation}\label{ipm:intro-small-history}
 g(t)\le e^{s-t}g(s)\le eh\,e^{-c_0x}.
\end{equation}
This is how a layer that later has a large gradient can already be
present at time zero with a very small seed. Localization makes the
amplitude and direction on its supported tube comparable to their
values at the origin. By $\operatorname{tr}H=-cr_2$ and
$M+H=-ce_2r^T$ in \eqref{ipm:raybasic},
\begin{equation}\label{ipm:matrixbounds}
 \|H\|\le1+c|r|,\qquad \|M\|\le1+2c|r|.
\end{equation}
These bounds control growth and rotation on a
short retained interval after $t_*$, preserving the layer's target direction.

To preserve the upper Hessian bound, choose the asymmetric phase profile
of Subsection~\ref{ipm:sec-profile}, with
\[
 f'_\delta(0)=\delta^{-1},\qquad f'_\delta\ge-C_f,
 \qquad \|f_\delta\|_\infty\le C_f.
\]
During activation, the new normal turns counterclockwise from almost
horizontal to almost vertical, remaining in the upper half-plane.
It therefore has $m_2>0$ throughout the supported tube during activation
and the retained future, when its target direction is preserved. In
\eqref{ipm:intro-leading-matrices}, the large positive slope of the
profile therefore creates a negative semidefinite Hessian. Its positive
part comes from the bounded negative slope and costs only $C\delta h$.
Before activation, when no direction sign is imposed, the small-history
bound supplies the cost $Ch e^{-c_0x}$. With a global correction error
$\varepsilon$, the resulting estimate is
\begin{equation}\label{ipm:intro-hessian-budget}
 H(P+\Theta)\le H(P)+
       \bigl[C(\delta h+h e^{-c_0x})+\varepsilon\bigr]I.
\end{equation}
The scales make these costs summable. Thus the pressure bound used to
control a layer's history is recovered for the next layer.

\subsection{The stage timeline and the limiting solution}

The activation time is $t_{j-1}$, the target time is $t_j$, and the retained
endpoint for the exact child is $S_j$. Put
\begin{equation}\label{ipm:intro-horizons}
 W_j=\frac{2x_{j-1}^2}{c h_{j-1}},\qquad
 \frac{W_j}{4}\le t_j-t_{j-1}\le\frac{3W_j}{4},\qquad
 S_j=t_j+2W_{j+1}.
\end{equation}
The parent is already available through
$S_{j-1}=t_{j-1}+2W_j>t_j$. Starting at the activation time $t_{j-1}$,
evolve the new phase normal under the parent flow and choose $t_j$ when
its angle at the origin reaches $\beta_j(t_j)=\pi/2+1/x_j$,
where $x_j=j^2x_{j-1}$ is prescribed in \eqref{ipm:intro-scales}.
Next set $S_j=t_j+2W_{j+1}$, using the already prescribed width
$W_{j+1}=2x_j^2/(c h_j)$. With this endpoint fixed, construct the child
on $[0,S_j]$; in particular, select its initial mean corrections by
solving the displacement problems with terminal condition $\eta(S_j)=0$.
Thus both $t_j$ and $S_j$ are determined before the child is constructed.

The intervals have distinct roles: $[0,t_{j-1}]$ is the layer's small
history, $[t_{j-1},t_j]$ is its amplification interval, and
$[t_j,S_j]$ retains the new geometry long enough to act as a parent.
For all $j\ge J$, the scale choices give $W_{j+1}\le W_j/16$
once $J$ and then $x_{J-1}$ are sufficiently large, as established in
Lemma~\ref{ipm:scales}. Hence
\[
 S_j\le t_j+\frac{W_j}{8}
 \le t_{j-1}+\frac{7W_j}{8}
 <t_{j-1}+2W_j=S_{j-1}.
\]
Moreover,
\begin{equation}\label{ipm:intro-accumulation}
 T=\lim_jt_j,\qquad
 0<T-t_j\le\frac45W_{j+1}<2W_{j+1}=S_j-t_j.
\end{equation}
Thus every finite stage is smooth past the common accumulation time $T$.
On every compact interval before $T$, the later layers are still in
their small-history regime. Their derivative bounds yield smooth
convergence to one exact unforced solution with a smooth limiting initial
density.

At the origin, the leading density-gradient increment of layer $j$ has
magnitude $g_j$ and angle $\beta_j$. The corresponding components are
\begin{equation}\label{ipm:intro-positive-components}
 (\partial_2\Theta_j)_{\rm lead}=g_j\sin\beta_j,\qquad
 (\partial_1u_1(\Theta_j))_{\rm lead}
       =c g_j\cos^2\beta_j\sin\beta_j.
\end{equation}
Both are nonnegative during activation and the retained future. After
its target, layer $j$ contributes at least fixed positive multiples of
$h_j$ and $h_j/x_j^2$, respectively. Later active layers preserve these
lower bounds up to summable errors, and later layers still in history
have summably small gradients. Since $h_j$ and $h_j/x_j^2$ tend to
infinity, the lower bounds hold at every sufficiently late time and give
the two actual limits in Theorem~\ref{ipm:main}. Finally, the small
uniform density increments and their controlled first derivatives give
summability in each $C^\eta$, $\eta<1$, and hence the common terminal
density.

\subsection{Related work on the incompressible porous media equation}
\label{ipm:lit-ipm}

\subsubsection{Local theory and the role of stratification}
The IPM system \eqref{ipm:darcy}--\eqref{ipm:unforced} couples density transport to Darcy's law. Its velocity is a Fourier multiplier of order zero applied to the density, whereas the two-dimensional Euler Biot--Savart operator gains one derivative from vorticity to velocity. Smooth transport preserves density values, but the gradient equation \eqref{ipm:raybasic} allows deformation to amplify derivatives. This distinction underlies both the local theory and the singularity problem.

C\'ordoba, Gancedo, and Orive established foundational local existence and uniqueness results, continuation criteria, and singularity formation in a special infinite-energy class \cite{lit:CGO}. In the usual finite-energy setting, local well-posedness holds above the two-dimensional Sobolev threshold $H^s$, $s>2$; the local theory is also discussed in \cite{lit:BJPW}. Such local results and transport bounds leave open the possibility of unbounded deformation at later times.

Stable stratification supplies a contrasting regime. Around the equilibrium $\rho_s=-x_2$, write $\rho=\rho_s+\vartheta$. Then the perturbation equation is $\partial_t\vartheta+u(\vartheta)\cdot\nabla\vartheta=u_2(\vartheta)$, and the Fourier symbol of its linear right-hand side is $-c\xi_1^2/|\xi|^2$. This anisotropic damping can overcome the nonlinear transport for small, sufficiently regular perturbations. Elgindi proved nonlinear asymptotic stability of stratified states \cite{lit:ElgindiIPM}. Castro, C\'ordoba, and Lear treated a confined strip, including the boundary conditions needed by the energy estimates \cite{lit:CCL}.

Subsequent work sharpened both regularity assumptions and decay estimates. Bianchini, Crin-Barat, and Paicu obtained stability on $\mathbb R^2$ through a relaxation from a damped Boussinesq system \cite{lit:BCP}. Jo and Kim quantified stability and sharp decay on several domains \cite{lit:JK}. Park developed an approach using the nonlinear potential-energy structure in a periodic channel \cite{lit:Park}. Bianchini, Jo, Park, and Wang proved stability for small $H^k$ perturbations with every real $k>2$ on $\mathbb T\times\mathbb R$ \cite{lit:BJPW}. These theorems concern perturbations of a stably stratified background; their domains and background densities should be retained when comparing them with a smooth periodic density on $\mathbb T^2$.

The endpoint behaves differently. Bianchini, C\'ordoba, and Mart\'inez-Zoroa proved nonexistence in the critical regularity class and strong ill-posedness for arbitrarily small $H^2(\mathbb R^2)$ perturbations of $-x_2$ \cite{lit:BCM}. Their construction locally neutralizes the stabilizing background and creates strong deformation. This is a critical-regularity phenomenon, and does not assert finite-time breakdown of an initially everywhere-smooth, unforced solution.

\subsubsection{Growth and singularities in other settings}
Kiselev and Yao constructed smooth data for which global smoothness implies unbounded growth of derivatives, and proved nonlinear instability of a class of stratified states \cite{lit:KY}. Their results include periodic settings and demonstrate that bounded transported density is compatible with the formation of arbitrarily fine scales. The conclusion is conditional growth over an unbounded time interval, rather than a finite-time singularity theorem.

Unforced finite-time blowup is known in classes that differ from smooth finite-energy flow without boundary. The infinite-energy examples of \cite{lit:CGO} and the related reduction of Castro, C\'ordoba, Gancedo, and Orive \cite{lit:CCGO} exploit a special spatial structure. Collot, Prange, and Tan established stability of an explicit self-similar singularity within this infinite-energy reduction, using its relation to a Proudman--Johnson equation \cite{lit:CPT}. More recently, Dembski constructed singularities for Lipschitz IPM solutions in wedge domains, smooth away from the origin, with compactly supported density vanishing on the boundary \cite{lit:Dembski}. The boundary geometry and the initial regularity are essential qualifications when placing this result next to a periodic smooth-data construction.

Weak-solution nonuniqueness is another distinct issue. C\'ordoba, Faraco, and Gancedo constructed bounded nonunique weak solutions of IPM by convex integration \cite{lit:CFG}. Such flexibility below the classical uniqueness regime does not itself imply breakdown of the unique smooth solution from smooth data. Likewise, interface singularities in the Muskat problem concern densities with a jump across an evolving curve. They belong to the broader IPM literature, but address a different regularity problem from the continuously varying smooth densities considered here.

\subsubsection{Successive amplification with a smooth source}
The closest predecessor of the present layer construction is the work of C\'ordoba and Mart\'inez-Zoroa \cite{lit:CMIPM}. They constructed finite-time singularities for IPM on $\mathbb R^2$ from smooth compactly supported initial density, with a compactly supported source whose spatial derivatives remain uniformly bounded. Their theorem controls the source in $L^\infty_t C^\infty_x$; joint smoothness through the singular time is discussed as a possible extension, rather than established by that theorem. Both the density gradient and the velocity gradient become unbounded.

Their mechanism repeatedly uses the field already constructed to amplify a new localized oscillation. The phase is transported by an approximation of the preceding velocity; the direction of its gradient controls the growth of the new layer. The Darcy multiplier admits a high-frequency expansion on a localized oscillation. Transversality cancels the largest quadratic self-interaction, and corrections of increasing order make the residual sufficiently small in progressively higher spatial norms. This permits infinitely many amplification stages to accumulate in finite time while the resulting source remains uniformly smooth in space \cite[\S1.2]{lit:CMIPM}.

Alp\"oge, Buckmaster, and Coiculescu extended this construction to $\mathbb T^2$ with a source in $C^\infty([0,1]\times\mathbb T^2)$ \cite{lit:ABC}. Their density converges in every $C^\eta$, $0\leq\eta<1$, while the density and velocity gradients diverge as $t\uparrow1$. The proof uses Fourier truncations of the preceding velocity and simultaneous control of mixed time and space derivatives. It retains the oscillatory amplification and correction strategy of \cite{lit:CMIPM}. Our construction on $\mathbb T^2$ removes the forcing used in both blow-up constructions \cite{lit:ABC,lit:CMIPM}. The parent is an exact unforced solution and its full velocity transports the new phase; no Fourier truncation of that parent is used. The terminal selection of the phase means and the exact residual correction are described in Sections~\ref{ipm:intro-mean-inverse} and \ref{ipm:sec-inverses}--\ref{ipm:sec-map}.

\subsection{Related work on three-dimensional Euler}
\label{ipm:lit-euler}

\subsubsection{Continuation and lower-regularity singularities}
For the incompressible Euler equations,
\[
 \partial_tu+(u\cdot\nabla)u+\nabla p=0,\qquad \operatorname{div} u=0,
\]
vorticity satisfies $D_t\omega=(\omega\cdot\nabla)u$. The stretching term has no counterpart in two-dimensional Euler. The Beale--Kato--Majda criterion states that a smooth three-dimensional solution can be continued while $\int_0^T\|\omega(t)\|_{L^\infty}\,dt$ remains finite \cite{lit:BKM}. Thus a finite-time singularity must involve sufficiently strong vorticity growth; large finite amplification alone does not establish breakdown.

Elgindi proved finite-time singularity formation on $\mathbb R^3$ for axisymmetric solutions without swirl with velocity in $C^{1,\alpha}$ for sufficiently small positive $\alpha$ \cite{lit:ElgindiEuler}. Elgindi, Ghoul, and Masmoudi proved stability of the self-similar mechanism and obtained finite-energy examples \cite{lit:EGM}. These solutions lie in a classical local uniqueness class, but their initial velocities are not everywhere smooth. C\'ordoba, Mart\'inez-Zoroa, and Zheng subsequently constructed finite-energy unforced solutions in $C^{1,\alpha}$ that are smooth outside the origin \cite{lit:CMZ}. Their construction uses an accumulation of spatial scales, providing a non-self-similar route to singularity formation.

Elgindi and Pasqualotto developed a different instability mechanism for the Boussinesq system and for Euler with swirl \cite{lit:EP}. In the Euler case their finite-energy $C^{1,\alpha}$ data are smooth away from a circle. The mechanism overcomes the effect of transport through an instability related to Taylor--Couette flow. The 2026 preprints of Shkoller and Chen treat finite-energy axisymmetric Euler without swirl for every $0<\alpha<1/3$ \cite{lit:Shkoller,lit:Chen}. Shao, Wei, Zhang, and Zhang construct asymptotically self-similar, finite-energy axisymmetric blowup without swirl in every dimension $d\geq3$ for $0<\alpha<1-2/d$, with initial velocity in $C^{1,\alpha}\cap C^\infty(\mathbb R^d\setminus\{0\})\cap L^2$ and compactly supported initial vorticity \cite[Theorem~1.3]{lit:SWZZ}. Taking $d=3$ gives precisely the range $0<\alpha<1/3$, denoted $C^{1,1/3-}$; the endpoint $\alpha=1/3$ is not included. These results extend the range of H\"older exponents for which singularities are constructed, while preserving the distinction between $C^{1,\alpha}$ and everywhere-smooth initial velocity.

\subsubsection{Smooth data, boundaries, and forcing}
Chen and Hou proved finite-time blowup from smooth finite-energy data in an axially periodic cylinder with an impermeable wall \cite{lit:CH1,lit:CH2}. Their proof combines analytic stability estimates for a nearly self-similar profile with rigorous numerical verification. The solid boundary is part of the theorem. Consequently, this result and a smooth-data theorem on all of $\mathbb R^3$ answer different versions of the Euler singularity problem.

Forced constructions form a second line of development. C\'ordoba and Mart\'inez-Zoroa constructed finite-energy Euler solutions with preterminal velocity in $C^{3,1/2}$ and force uniformly bounded in $C^{1,1/2-\varepsilon}\cap L^2$ for each fixed $0<\varepsilon<1/2$ \cite{lit:CMEuler}. In a published generalization, C\'ordoba, Mart\'inez-Zoroa, and Zheng extended this construction to the hypodissipative Navier--Stokes equations, obtaining finite-time blowup of smooth preterminal finite-energy solutions for sufficiently small positive dissipation orders, with force in $L^1_tC_x^{1,\varepsilon}\cap L^\infty_tL_x^2$ for some $\varepsilon>0$ \cite{lit:CMZHypo}. Alp\"oge and Buckmaster's later manuscript constructs axisymmetric blowup with swirl from smooth compactly supported data and a force smooth in both variables through the singular time \cite{lit:ABE}. Their companion Boussinesq manuscript develops the corresponding amplification mechanism with smooth forcing in both equations \cite{lit:ABB}. In these constructions the preceding layers amplify finer oscillations and high-order corrections control the residual assigned to the force.

The unforced Euler manuscript and its exposition \cite{lit:OpenAI,lit:ZhengEuler} state finite-time breakdown on $\mathbb R^3$ for smooth compactly supported divergence-free initial velocity. Its construction passes through exact smooth Euler solutions obtained by adding localized oscillations and corrections to earlier solutions. A uniform upper bound on the pressure Hessian permits a controlled choice of the earlier history of each layer. The initial increments are summable in every fixed Sobolev norm, while the finite-stage gradients become arbitrarily large. It is the direct methodological precedent for the proposed unforced IPM construction here.

\subsubsection{Norm inflation and weak solutions}
High-frequency amplification also appears in problems that do not require a finite-time singularity from smooth data. Bourgain and Li showed that large Lagrangian deformation produces strong ill-posedness in the borderline Sobolev spaces $H^{d/2+1}$ for velocity, in dimensions two and three \cite{lit:BL}. Elgindi and Jeong gave a related critical Sobolev ill-posedness construction \cite{lit:EJ}. These works explain how a background flow can amplify a small perturbation through deformation, but critical norm inflation is a different conclusion from the breakdown of an everywhere-smooth solution.

Convex integration addresses still weaker solutions. De Lellis and Sz\'ekelyhidi constructed nonunique bounded weak Euler solutions \cite{lit:DLS}; Isett proved the existence of nonconservative H\"older solutions below the Onsager exponent $1/3$ \cite{lit:Isett}. Oscillatory corrections occur in both this theory and layer constructions, but their purposes differ. Here each finite stage must remain a classical solution, and the limiting initial datum must be smooth. Weak-solution flexibility does not supply those requirements.

\subsection{Transported waves and the relation to earlier constructions}
\label{ipm:lit-waves}

\subsubsection{Geometric optics and exact waves}
The phase and amplitude equations have an older origin than the recent blowup constructions. Lifschitz and Hameiri, and Friedlander and Vishik, studied short-wavelength perturbations of incompressible flows through geometric optics \cite{lit:LH,lit:FV}. If $X$ is the background flow, $F=D_zX$, and $M=\nabla u\circ X$, then a transported wavevector with initial value $m_0$ has the form $m=F^{-T}m_0$ and obeys $m'=-M^Tm$, with time derivatives taken along the trajectory. For linearized Euler the leading velocity amplitude $v$ satisfies
\[
 v'=-Mv+2\frac{m\cdot Mv}{|m|^2}m,\qquad m\cdot v=0.
\]
The second term is the pressure contribution required to preserve transversality. These equations describe leading high-frequency growth along trajectories; using that growth in an iterative nonlinear construction additionally requires a field defined and controlled on an entire spatial region.

Craik and Criminale constructed exact finite-amplitude waves over affine background flows \cite{lit:CC}. For a velocity wave $v(t)f(m(t)\cdot x)$, its quadratic self-advection vanishes when $m\cdot v=0$. This cancellation is a useful model for the leading order calculation. Fabijonas and Holm proposed iterated wave superpositions \cite{lit:FH}, but Le Diz\`es and Leblanc showed that the proposed iteration generally solves the equations only along a trajectory \cite{lit:LL}. That limitation explains why an exact affine wave calculation alone does not justify using the resulting field as the parent of the next localized layer. Cheverry's nonlinear geometric-optics analysis develops oscillatory expansions and the additional phase corrections generated by nonlinear evolution \cite{lit:Cheverry}.

\subsubsection{The IPM identities in this framework}
For IPM, Darcy's law determines the velocity instantaneously from the density. The symbol \eqref{ipm:velocity-symbol}, the transported phase equation \eqref{ipm:raybasic}, and the amplitude equation \eqref{ipm:intro-wave} lead to the pressure-curvature identity \eqref{ipm:rayidentity}. This connects the transported-phase method of \cite{lit:CMIPM,lit:ABC} with the formulation used here: the density-gradient amplitude grows where the parent pressure has negative curvature transverse to the phase. The calculation and its transverse cancellation were given in Section~1.2.

\subsubsection{From amplification to an unforced construction}
An approximate layer with a small residual gives a forced solution if that residual is assigned to the external source. An unforced construction must instead remove it by solving correction equations. Moreover, a new layer cannot be inserted at a later time by a time cutoff without producing a source. Its oscillatory seed and selected phase means must already belong to the time-zero data; the exact residual correction vanishes initially. Large amplification near a target time must therefore be compatible with controlled evolution during the layer's entire earlier history.

The Euler construction in \cite{lit:OpenAI,lit:ZhengEuler} deals with this history by displacement boundary-value problems. The quadratic form governing a displacement $\eta$ contains $\int(|\eta_t|^2-\eta\cdot H\eta)\,dt$. Its coercivity on an appropriately short interval follows from an upper bound on $H$ and the time Poincar\'e inequality. This is related to the pressure-augmented least-action principle and short-time minimization discussed by Brenier \cite{lit:Brenier}; the variational estimate selects a controlled history, rather than suppressing the later amplification.

The present construction retains successive amplification, transverse cancellation, and the upper pressure-Hessian bound, but its displacement problem has a different differential order. The variables $\eta$ and $Q=\varrho+\nabla P\cdot\eta$ from Section~\ref{ipm:intro-mean-inverse} reduce the forced linearized density equation to $D_tQ=f$. The resulting first-order projected terminal problem selects the initial phase mean. Proposition~\ref{ipm:meaninverse} controls that history by backward energy estimates. The second-order Euler variational problem is its methodological precedent; the IPM inverse uses the first-order displacement equation instead.

The proposed extension combines the earlier amplification mechanisms with control of the full layer history, exact unforced correction, and summability of all initial derivatives. The domain and regularity in Theorem~\ref{ipm:main} distinguish its conclusion from forced blowup, critical ill-posedness, and singularities that use a boundary or infinite energy.

\subsection*{Organization and notation}

The proof begins with the geometric estimates, then constructs and
iterates an exact layer. Sections~\ref{ipm:sec-transfer}--\ref{ipm:sec-history} establish
quantitative amplification and turning, and control of the wave history
and pressure. Sections~\ref{ipm:calculus}--\ref{ipm:sec-inverses}
develop the fixed calculus and the linear inverses.
Sections~\ref{ipm:sec-finite}--\ref{ipm:sec-initial-increment} then construct
the finite approximation, correct it to an exact solution, recover the
particle-map class, and identify the initial increment.
Section~\ref{ipm:sec-packet} collects these results into the complete
conditional layer theorem. Sections~\ref{ipm:seventh-cascade-interface}--\ref{ipm:sec-cascade}
verify its geometric inputs, choose the scales and stationary base, and
close the induction. Finally,
Sections~\ref{ipm:sec-initial-limit}--\ref{ipm:seventh-physical} construct
the limiting initial datum and smooth solution, prove the component
blow-up, and establish uniqueness and the terminal limit. Appendix~\ref{ipm:seventh-appendices}
records the uniform operator estimates.

Constants denoted by $C$ are independent of the stage. A subscript, such as
$C_m$, permits dependence on the indicated fixed derivative order.
The parent fields $r=\nabla P$, $M=\nabla u(P)$, and $H=\nabla^2p(P)$
have Eulerian arguments $(t,x)$; composition with $X$ is understood only
where explicitly stipulated.

\section{Quantitative amplification and the next direction}
\label{ipm:sec-transfer}

In this section we quantify the gain and turning produced by a nearly
vertical parent, under hypotheses on its density gradient $r_P$,
velocity gradient $M_P$, pressure Hessian $H_P$, and their remainders
at the origin. The layer construction summarized in
Section~\ref{ipm:sec-packet} will supply an exact layer with these
leading fields; Section~\ref{ipm:seventh-cascade-interface} verifies
the resulting transfer hypotheses.

\begin{lemma}[Transfer with value errors]
\label{ipm:amplification}
Fix $c>0$. There are constants $x_0$, $j_0$, and $\epsilon_0>0$
with the following property. Suppose that, on $[s,s+2W]$, a parent has
the following decompositions at the origin:
\begin{align}
 r_P&=G n+R_r,&
 M_P&=-cG\cos\alpha\,n^\perp\otimes n+R_M,
       \label{ipm:transfer-decomposition}\\
 H_P&=-cG\sin\alpha\,n\otimes n+R_H,&
 n&=(\cos\alpha,\sin\alpha),
       \nonumber
\end{align}
where $G(s)=H_0\ge1$, $\alpha(s)=\pi/2+1/x$, and
\begin{equation}
 |R_r|+\|R_M\|+\|R_H\|
       +|\alpha'|+|(\log G)'|\le E.
 \label{ipm:transfer-error}
\end{equation}
Let
\begin{equation}
 W=\frac{2x^2}{cH_0},\qquad
 \eta=\frac1{j^2x},\qquad
 j\ge j_0,\quad x\ge x_0,
 \qquad \frac{E j^2x^3}{H_0}\le\epsilon_0.
 \label{ipm:transfer-smallness}
\end{equation}
Choose a child phase with initial angle
$\beta(s)=\alpha(s)-\pi/2+\tau$, where
$|\tau|j^2x\le\epsilon_0$. Its first crossing
$\beta(t_*)=\pi/2+\eta$ occurs with
\begin{equation}
 s+W/4\le t_*\le s+3W/4.
 \label{ipm:transfer-time}
\end{equation}
The phase angle at the origin satisfies
\begin{equation}\label{ipm:transfer-angle-range}
 \beta(t)\in[1/(2x),\pi-1/(2x)],\qquad s\le t\le t_*.
\end{equation}
The child amplitude satisfies
\begin{equation}
 e^{c_0x}\le
 \frac{a(t_*,0)|m(t_*,0)|}{a(s,0)|m(s,0)|}
 \le e^{C_0x}
 \label{ipm:transfer-gain}
\end{equation}
for absolute constants $c_0,C_0>0$.
\end{lemma}

\begin{proof}
We give a stopping argument in which no target crossing is assumed.
Throughout this proof constants may depend on the fixed number $c$.
Set
\[
 \mu=\frac{Ex^3}{H_0}\le\epsilon_0j^{-2}.
\]
Since
$2W=4x^2/(cH_0)$, integration of the two derivative bounds yields, for
$s\le t\le s+2W$,
\begin{equation}
 |\alpha(t)-(\pi/2+x^{-1})|+|\log(G(t)/H_0)|
       \le2EW=\frac{4\mu}{cx}\le\frac{4\epsilon_0}{c}\eta.
 \label{ipm:expanded-cascade-parent-drift}
\end{equation}
Fix $\epsilon_0$ so small that this is at most $\eta/100$.
By increasing $x_0,j_0$, we also arrange
that $G/H_0\in[9/10,11/10]$ and
$\alpha-\pi/2\in[9/(10x),11/(10x)]$.

Let $N=(\cos\beta,\sin\beta)$ be the child unit phase. Its equation is
\[
 \beta'=(N^\perp)^T(-M_P^T)N.
\]
For $n=(\cos\alpha,\sin\alpha)$,
$n^\perp\cdot N=\sin(\beta-\alpha)$ and
$N^\perp\cdot n=-\sin(\beta-\alpha)$. Set $q=-cG\cos\alpha$,
$d=\beta-\alpha$, and $z=-\cot d$ whenever $d\in(-\pi,0)$.
Substituting the rank-one part of $M_P$ then gives
\begin{equation}
 \beta'=q(t)\sin^2d+O(E),\qquad
 d'=q(t)\sin^2d+O(E),\qquad
 z'=q(t)+O(E(1+z^2)),
 \label{ipm:transfer-z}
\end{equation}
The identity for $d'$
follows from $d'=\beta'-\alpha'$, the identity for $\beta'$, and
$|\alpha'|\le E$ from \eqref{ipm:transfer-error}.
The identity for $z'$ follows from
$z'=(\sin d)^{-2}d'$ and
$(\sin d)^{-2}=1+z^2$.

Since $-\cos\alpha=\sin(\alpha-\pi/2)
=x^{-1}[1+O(x^{-2}+\mu)]$, \eqref{ipm:expanded-cascade-parent-drift} gives
\begin{equation}
 q(t)=\frac{cH_0}{x}
       [1+O(x^{-2}+\mu)].
 \label{ipm:transfer-q}
\end{equation}

Initially $d(s)=-\pi/2+\tau$ and $z(s)=\tan\tau$.
Consider the connected interval beginning at $s$ on which $d\in(-\pi,0)$
and $|z|<2x$. On this interval the last equation in
\eqref{ipm:transfer-z} has error at most $CEx^2$. Relative to
$cH_0/x$ this is at most $C\mu$. Hence
\begin{equation}
 z'=\frac{cH_0}{x}
       [1+O(x^{-2}+\mu)]>0.
 \label{ipm:transfer-monotone}
\end{equation}
We may make the relative error here at most $1/20$. The same choices
give $|z(s)|=|\tan\tau|\le2|\tau|\le2\epsilon_0/(j^2x)\le x/100$.
A first exit before $s+3W/4$ is impossible:
monotonicity excludes the boundary $z=-2x$, while integration gives
\[
 z(t)\le x/100+(21/20)(cH_0/x)(3W/4)
             <2x
\]
for $t\le s+3W/4$. The angle cannot leave $(-\pi,0)$ while $z$
remains bounded, because the two endpoints correspond to poles of
$-\cot d$. Thus all the preceding estimates hold on the complete
interval $[s,s+3W/4]$, whether or not a target has yet been crossed.
\label{ipm:expanded-cascade-stopping}

The target is the moving value
\begin{equation}
 z_{\rm tar}(t)=
    \cot\bigl(\alpha(t)-\pi/2-\eta\bigr).
 \label{ipm:expanded-cascade-moving-target}
\end{equation}
Its argument is positive, since
$\alpha-\pi/2-\eta=x^{-1}(1-j^{-2})+O(\mu/x)$.
The Taylor estimate $\cot u=u^{-1}+O(u)$ for small positive $u$
therefore gives, uniformly in the complete stopped interval,
\[
 z_{\rm tar}(t)=x\,[1+O(j^{-2}+x^{-2}+\mu)].
\]
Choose $j_0,x_0,\epsilon_0$ so that it lies between $0.9x$ and
$1.1x$. On $[s,s+W/4]$, integration of
\eqref{ipm:transfer-monotone} gives $z(t)\le0.54x$, whereas at
$s+3W/4$ it gives $z\ge1.41x$. Continuity forces the first intersection
with \eqref{ipm:expanded-cascade-moving-target} between $s+W/4$ and
$s+3W/4$, proving \eqref{ipm:transfer-time}.

Since $z$ is increasing and $d\in(-\pi,0)$, the angle $d$ is
increasing, so $d(t)\ge d(s)=-\pi/2+\tau$. Together with
$9/(10x)\le\alpha-\pi/2\le11/(10x)$, this gives
\[
 \beta=\alpha+d\ge\frac{9}{10x}-|\tau|\ge\frac1{2x},
 \qquad \beta<\alpha\le\frac\pi2+\frac{11}{10x}\le\pi-\frac1{2x}.
\]
Here we use the smallness of $|\tau|$ and sufficiently large $x_0$.
This proves the asserted angle range on $[s,t_*]$.

By \eqref{ipm:rayidentity} and the Hessian decomposition,
\begin{equation}
 (\log(a|m|))'
     =cG\sin\alpha\,\sin^2d+O(E).
 \label{ipm:transfer-log-g}
\end{equation}
To estimate this growth, use $|z|\le2x$ and \eqref{ipm:transfer-q}:
\begin{equation}
 \sin^2d=\frac1{1+z^2}\ge\frac1{1+4x^2},\qquad
 q\sin^2d\ge c_1H_0/x^3.
 \label{ipm:expanded-cascade-angle-denominator}
\end{equation}
Write
$B=q\sin^2d$ and $D=cG\sin\alpha\sin^2d$. Then
$D/B=-\tan\alpha$, $B\ge c_1H_0/x^3$, and the quotient of either
additive error $O(E)$ by $B$ is at most $C\mu$. Therefore
\[
 \frac{(\log(a|m|))'}{d'}
 =-\tan\alpha(t)
       [1+O(\mu)]
       +O(\mu),
\]
where the denominator factor $(1+O(\mu))^{-1}$ was expanded only
after ensuring $|O(\mu)|\le1/4$. The number $-\tan\alpha(t)$ lies
between fixed positive multiples of $x$ by
\eqref{ipm:expanded-cascade-parent-drift}. Moreover,
\[
 d(t_*)-d(s)
 =\pi/2-\tau-[\alpha(t_*)-\pi/2-\eta]
\]
lies, for the fixed choices above, between $1$ and $2$. Integration
with respect to the strictly increasing variable $d$ proves
\eqref{ipm:transfer-gain} with fixed $c_0,C_0>0$.
\end{proof}

\section{Wave history and pressure control}
\label{ipm:sec-history}

The history estimate \eqref{ipm:backwardray} bounds the earlier gradient by its
value at activation. We now combine it with localization and an asymmetric
profile to control the positive pressure curvature. The estimates here concern the leading fields
\eqref{ipm:intro-leading-matrices}; their application to the exact child
uses the global remainder bounds \eqref{ipm:packet-gradient-errors}. They are used after the exact layer theorem is proved in
Section~\ref{ipm:sec-packet}. The Gevrey product and composition estimates
used below are established independently in Section~\ref{ipm:calculus}.

\subsection{The parent regularity class}

Throughout, an ordered derivative word is a tuple
$I=(i_1,\ldots,i_n)\in\{1,\ldots,d\}^n$, with length $|I|=n$ and
$\partial^I=\partial_{i_1}\cdots\partial_{i_n}$.
The notation $\sum_{|I|=n}$ sums over all such tuples; for a specified
variable $z$, write $\partial_z^I=\partial_{z_{i_1}}\cdots\partial_{z_{i_n}}$.
For $J\subseteq\{1,\ldots,n\}$, $I_J$ is the subword in those positions,
in their original order, and $J^c$ denotes the complementary positions.
The Sobolev base index $\beta$ is an ordinary multi-index.
For a field $Z$ in unscaled physical labels $z$, write
$[Z]_n=\sum_{|I|=n}\|\partial_z^IZ\|_\infty$.

Let $X(t,z)$ be the initial-label flow of a smooth odd solution $P$ of
\eqref{ipm:darcy}--\eqref{ipm:unforced} on $[0,S]\times\mathbb T^2$, where $S\le1$.
Thus $X(0,z)=z$, $\det D_zX=1$, and $X(t,-z)=-X(t,z)$ for a lift to $\mathbb R^2$. Recall that $F=D_zX$ in these unscaled labels.
A bound by the parameter $K\ge2$ means
\begin{equation}\label{ipm:parent-class}
 \sup_{0\le t\le S}[Z(t)]_n
 \le K^{n+1}(n!)^2\quad(n\ge0),
 \qquad
 Z\in\{P(0),X-\operatorname{id},X_t,X_{tt}\}.
\end{equation}
A fixed power enlargement of $K$ absorbs the distinction between equivalent
finite-dimensional norm conventions.

\subsection{The asymmetric phase profile}
\label{ipm:sec-profile}
\label{ipm:profile}

\begin{lemma}[A narrow, large positive slope and a bounded negative slope]
\label{ipm:profilelemma}
There is a family of odd $2\pi$-periodic Gevrey-two functions $f_\delta$,
$0<\delta\le\delta_0$, with fixed constants $C_f\ge1$ and $\delta_0>0$, such that
\begin{equation}\label{ipm:profile-bounds}
 f_\delta'(0)=\delta^{-1},\qquad f_\delta'\ge-C_f,
 \qquad \|f_\delta\|_\infty\le C_f,
 \qquad
 \|f_\delta^{(n)}\|_\infty\le C_f^{n+1}\delta^{-n}(n!)^2.
\end{equation}
There is also a fixed nonnegative even Gevrey-two cutoff $\chi$ supported
in $B_{1/2}(0)$ and equal to one on $B_{1/4}(0)$.
\end{lemma}
\begin{proof}
Choose the even bump
\[
 \psi(z)=\begin{cases}\exp(1-(1-z^2)^{-1}),&|z|<1,\\0,&|z|\ge1.
 \end{cases}
\]
It satisfies $\psi(0)=1$ and $\|\psi^{(n)}\|_\infty\le C^{n+1}(n!)^2$.
To wit, at a point with $t=1-z^2>0$, a complex disk of radius
$c_0t$ keeps $1-w^2$ in a fixed right sector of size comparable to $t$.
Cauchy's estimate gives
$|\psi^{(n)}(z)|\le n!(C/t)^n e^{-c_1/t}$.
Taking the supremum in $t$ and using $n^ne^{-n}\le n!$ gives the claim.
The same estimate at the endpoints implies that all derivatives extend
by zero there.

Periodize $\psi(\theta/\delta)$ and put
\[
 \mu_\delta=\frac1{2\pi}\int_{-\pi}^{\pi}\psi(\theta/\delta)\,d\theta
             =a_0\delta,\qquad
 g_\delta=\frac{\psi(\theta/\delta)-\mu_\delta}
                  {\delta(1-\mu_\delta)}.
\]
Take $\delta_0$ small enough that $\mu_\delta\le1/2$.
Then $g_\delta$ is even, has zero mean, satisfies $g_\delta(0)=\delta^{-1}$,
and has negative part bounded by $2a_0$.
The primitive $f_\delta(\theta)=\int_0^\theta g_\delta(v)\,dv$ is odd
and periodic. Its supremum is bounded by the integral of $|g_\delta|$,
which is uniform. Its positive-order derivative bounds follow from those
of $\psi$ and the displayed scaling. Choose one fixed $C_f\ge1$
to cover these profile bounds. A fixed cutoff $\chi$ is obtained
by integrating and normalizing $\psi$ to a smooth step and composing
that step with an affine function of $|y|^2$. Lemma~\ref{ipm:composition}
gives its Gevrey bound; the affine function is chosen to place the two
cutoff radii at $1/4$ and $1/2$.
\end{proof}

\begin{corollary}[One-sided pressure cost of the profile]
\label{ipm:profile-pressure-cost}
Let $m\ne0$, $\alpha,a\ge0$, and $\chi\ge0$, and consider the leading
pressure Hessian $(\nabla^2\pi)_{\rm lead}$ in
\eqref{ipm:intro-leading-matrices}, evaluated at $X(t,\ell y)$.
If $m_2>0$, then $(\nabla^2\pi)_{\rm lead}\le C_fc\alpha\chi a|m|I$.
Without a sign assumption on $m_2$, $\|(\nabla^2\pi)_{\rm lead}\|\le
C_f^2c\alpha\chi a|m|/\delta$. Thus wherever
$\alpha\chi a|m|/\delta\le C_*h$, the positive-Hessian cost
in the first case is at most $C_fC_*c\delta h$.
\end{corollary}
\begin{proof}
The conclusions follow from $0\le(m\otimes m)/|m|^2\le I$,
$|m_2|\le|m|$, and the bounds $f'_\delta\ge-C_f$ and
$\|f'_\delta\|_\infty\le C_f^2/\delta$ from \eqref{ipm:profile-bounds}.
\end{proof}

\subsection{Localization, normalization, and the pressure budget}

\begin{lemma}[Parent coefficients and localization of a transported layer]
\label{ipm:tube-localization}
Assume the four-field parent class \eqref{ipm:parent-class} with
parameter $K$. In physical labels let $m=F^{-T}m_0$, $|m_0|=1$,
and normalize the transported scalar amplitude by $a(s,\cdot)=1$.
Recall from \eqref{ipm:amplitude-coefficient-definition} that
$c_A(t,z)=\mathfrak m(F(t,z)^{-T}m_0)\cdot\nabla P(t,X(t,z))$.
Then all the following conclusions hold uniformly for $t\in[0,S]$.
For an absolute exponent $A$,
\begin{equation}
 [c_A]_n\le A_c B_c^n(n!)^2,\qquad
 A_c+B_c\le K^{A},\qquad
 [m]_n\le K^{A(n+1)}(n!)^2.
 \label{ipm:expanded-cascade-amplitude-coefficient}
\end{equation} Whenever
$\ell K^{2A+3}\le\varepsilon_{\rm loc}\le1$,
the normalized amplitude $a(t,\ell y)/a(t,0)=e^{\Lambda(t,y)}$
satisfies, on $|y|\le1$,
\begin{equation}\label{ipm:centered-amplitude-gevrey}
 [e^\Lambda-1]_{n,y}
       \le C\varepsilon_{\rm loc}C^n(n!)^2\quad(n\ge0).
\end{equation}
If $\varepsilon_{\rm loc}\le1/(20\pi)$,
then for $|y|\le1$,
\begin{equation}
 \frac12\le
 \frac{a(t,\ell y)|m(t,\ell y)|}{a(t,0)|m(t,0)|}
 \le2,
 \qquad
 \left|\frac{m(t,\ell y)}{|m(t,\ell y)|}
          -\frac{m(t,0)}{|m(t,0)|}\right|
       \le2\ell K^{2A+3}.
 \label{ipm:cascade-tube-comparison}
\end{equation}
\end{lemma}
\begin{proof}
Using \eqref{ipm:parent-class}, $(n+1)^2\le4^n$,
and $K\ge2$, one extra spatial derivative gives
\[
 [F]_n\le\mathbf1_{\{n=0\}}+K^{n+2}((n+1)!)^2
       \le K^{n+3}((n+1)!)^2
       \le K^{n+3}4^n(n!)^2
       \le K^{3(n+1)}(n!)^2.
\]
The same holds for $F_t$ and $F_{tt}$, using the bounds on $X_t$ and
$X_{tt}$. Since $\det F=1$,
$F^{-1}=\operatorname{adj}F=(\operatorname{cof}F)^T$. In dimension two,
its entries are the entries of $F$ times $\pm1$, so $F^{-1}$ has the same bound.

Transport of the parent density $P$ gives $r\circ X=F^{-T}\nabla P(0)$.
The product bounds of Lemmas~\ref{ipm:word-products} and~\ref{ipm:shiftcalculus}
show that a product of a fixed number of fields of this class has
size and growth constants bounded by a fixed power of $K$.
Thus $r\circ X$ has that class.

The reciprocal bound has the same quantitative
property. If $g\ge b>0$ and
$[g]_n\le U V^n(n!)^2$, differentiating $g g^{-1}=1$ gives, for $n\ge1$,
\[
 [g^{-1}]_n\le b^{-1}\sum_{l=1}^n
          \binom nl[g]_l[g^{-1}]_{n-l}.
\]
Choose $R\ge4V\max(1,U/b)$. Induction using \eqref{ipm:shiftabsorption} and $\sum_{l\ge1}(V/R)^l\le2V/R$ proves
$[g^{-1}]_n\le2b^{-1}R^n(n!)^2$.
This is also a direct special case of Lemma~\ref{ipm:composition}.
Its constants are polynomial in $U,V,b^{-1}$, not exponential in
these quantities.

For $m=F^{-T}m_0$ with $|m_0|=1$, the identity $F^Tm=m_0$ gives
$|m|\ge\|F\|^{-1}\ge K^{-3}$.
Apply the preceding reciprocal argument to $|m|^2$ and then the
product bounds of Lemmas~\ref{ipm:word-products} and~\ref{ipm:shiftcalculus} to the rational symbol
$\mathfrak m$ in \eqref{ipm:velocity-symbol}.
These estimates give the coefficient bounds
\eqref{ipm:expanded-cascade-amplitude-coefficient}.
The fixed exponent $A$ is uniform in $m_0$, the child carrier, the seed,
and the stage. This verifies the coefficient interface with the exact
layer proposition using only its inherited four-field class.

Write $a(s,z)=1$. Since $a_t+c_A a=0$, the growth at the origin
cancels in the identity
\begin{equation}
 \Lambda(t,y):=\log\frac{a(t,\ell y)}{a(t,0)}
       =-\int_s^t[c_A(\sigma,\ell y)
                              -c_A(\sigma,0)]\,d\sigma.
 \label{ipm:cascade-normalization}
\end{equation}
On $|y|\le1$, the mean-value theorem and $S\le1$ give
\begin{equation}
 \|\Lambda\|_\infty\le A_cB_c\ell,\qquad
 [\Lambda]_{n,y}\le A_c(\ell B_c)^n(n!)^2\quad(n\ge1).
 \label{ipm:expanded-cascade-centered-derivatives}
\end{equation}
The bound $\ell K^{2A+3}\le\varepsilon_{\rm loc}$ and
\eqref{ipm:expanded-cascade-amplitude-coefficient} imply $\ell B_c\le1$ and
$A_cB_c\ell\le\varepsilon_{\rm loc}$. At this step we only require $\varepsilon_{\rm loc}\le1$. Then for $n\ge1$ the
last bound is at most $\varepsilon_{\rm loc}(n!)^2$.
The exponential composition estimate of Lemma~\ref{ipm:composition}
gives
\eqref{ipm:centered-amplitude-gevrey}.
One can retain the factor $\varepsilon_{\rm loc}$ directly in its
proof: every positive-order derivative term in the exponential
chain rule contains at least one derivative of $\Lambda$, and all
those derivatives have this factor; the remaining factorial
partition sum is bounded by $C^n(n!)^2$. The $L^\infty$ estimate is $|e^\Lambda-1|\le e^{\|\Lambda\|_\infty}\|\Lambda\|_\infty$.
The same calculation works with any lower time limit
$\sigma\in[0,S]$ in place of $s$. It therefore controls the full
centered propagator, in both directions of time.

For the phase comparison, \eqref{ipm:expanded-cascade-amplitude-coefficient}
and $|m|\ge K^{-3}$ give
\begin{equation}\label{ipm:phase-spatial-derivative-bounds}
 |\nabla_{z}\log|m||\le\frac{|\nabla_{z}m|}{|m|}
       \le K^{2A+3},\qquad
 \left|\nabla_{z}\frac m{|m|}\right|
       \le\frac{2|\nabla_{z}m|}{|m|}\le2K^{2A+3}.
\end{equation}
Apply the mean-value theorem between $0$ and $\ell y$.
On $|y|\le1$, \eqref{ipm:expanded-cascade-centered-derivatives}
and the choice $A_cB_c\ell\le\varepsilon_{\rm loc}$ give
\[
 \left|\log\frac{a(t,\ell y)}{a(t,0)}\right|\le\varepsilon_{\rm loc},
 \qquad
 \left|\log\frac{|m(t,\ell y)|}{|m(t,0)|}\right|
 \le\ell K^{2A+3}\le\varepsilon_{\rm loc}.
\]
Thus $\varepsilon_{\rm loc}\le1/(20\pi)$ makes their sum at most
$2\varepsilon_{\rm loc}\le1/(10\pi)<\log2$. Together with the direction derivative bound in \eqref{ipm:phase-spatial-derivative-bounds}, this proves
\eqref{ipm:cascade-tube-comparison}.
\end{proof}

\begin{proposition}[History and retained geometry]
\label{ipm:packet-history-geometry}
Let $P$ be an exact parent on $[0,S]$, $S\le1$, with
\begin{equation}\label{ipm:history-parent-bounds}
 H(P)\le I,\qquad \|\nabla P\|_\infty\le C_{\nabla}h_{\rm par}.
\end{equation}
Suppose a layer activates at $s$, reaches angle
$\pi/2+1/x_{\rm tar}$ at the child target time $t_*$ with gain at least $e^{c_0x}$,
and is retained until $S=t_*+2W_{\rm next}$.
Assume its angle at the origin satisfies
$\beta(t)\in[1/(2x),\pi-1/(2x)]$ for $s\le t\le t_*$,
where $x_{\rm tar}\ge x$ is sufficiently large.
Normalize its leading density-gradient increment at the origin to magnitude
$h$ at the target time by
\begin{equation}
 \alpha a(t_*,0)|m(t_*,0)|/\delta=h.
 \label{ipm:cascade-target-normalization}
\end{equation}
Then the history estimate is
\begin{equation}
 \alpha a(t,0)|m(t,0)|/\delta
       \le e h e^{-c_0x},\qquad 0\le t\le s.
 \label{ipm:cascade-small-history}
\end{equation}
Assume in addition that
\begin{equation}
 2(1+2cC_{\nabla})(1+h_{\rm par})W_{\rm next}
 \le\frac1{4x_{\rm tar}},
 \qquad x_{\rm tar}\ell K^{2A+3}\le\frac1{20\pi}.
 \label{ipm:cascade-future-localization}
\end{equation}
Then $m_2>0$ throughout $\operatorname{supp}\chi$
on $[s,S]$.
With $\gamma(t)=\alpha a(t,0)$,
\begin{equation}
 \sup_{[0,S]}\gamma\le\Gamma:=\max(1,CK^A\delta h),
 \label{ipm:cascade-gamma-choice}
\end{equation}
and its earlier history satisfies
\begin{equation}
 \gamma(t)\le CK^A\delta h e^{-c_0x},\qquad t\le s.
 \label{ipm:cascade-gamma-history}
\end{equation}
All exponents are fixed, and the constants are independent of the stage.
\end{proposition}
\begin{proof}
The target normalization \eqref{ipm:cascade-target-normalization} and the
activation-to-target gain give
$\alpha a(s,0)|m(s,0)|\le\delta h e^{-c_0x}$.
By \eqref{ipm:backwardray} and $S\le1$, for $0\le t\le s$,
\[
 \alpha a(t,0)|m(t,0)|
 \le e^{s-t}\alpha a(s,0)|m(s,0)|
 \le e\delta h e^{-c_0x}.
\]
This proves \eqref{ipm:cascade-small-history} without the additional
smallness conditions.

Now assume \eqref{ipm:cascade-future-localization}. Equations
\eqref{ipm:history-parent-bounds}, \eqref{ipm:backwardray},
\eqref{ipm:rayidentity}, and the matrix bounds
\eqref{ipm:matrixbounds} give
\begin{align}
 \alpha a(t,0)|m(t,0)|&\le e\delta h,
                   &&0\le t\le t_*,
                   \label{ipm:cascade-before-target}\\
 \alpha a(t,0)|m(t,0)|&\le\delta h
              \exp[C(1+h_{\rm par})W_{\rm next}],
                   &&t_*\le t\le t_*+2W_{\rm next}.
                   \label{ipm:cascade-after-target}
\end{align}
From the phase equation,
$|\beta'|\le\|M_P\|\le1+2cC_{\nabla}h_{\rm par}$.

For $t_*\le t\le S$, the angle at the origin satisfies
\[
 |\beta(t)-\beta(t_*)|
 \le\int_{t_*}^t|\beta'(\tau)|\,d\tau
 \le2(1+2cC_{\nabla}h_{\rm par})W_{\rm next}
 \le\frac1{4x_{\rm tar}}.
\]
Since $\beta(t_*)=\pi/2+1/x_{\rm tar}$,
\[
 \frac\pi2+\frac3{4x_{\rm tar}}
 \le\beta(t)\le\frac\pi2+\frac5{4x_{\rm tar}},
 \qquad
 \sin\beta(t)\ge\cos\frac5{4x_{\rm tar}}\ge\frac12
\]
for sufficiently large $x_{\rm tar}$. This same requirement bounds the exponential in
\eqref{ipm:cascade-after-target} by a fixed constant, which we may
take to be $2$ after tightening the fixed threshold.

For $s\le t\le t_*$, the assumed activation-angle range gives
\[
 \frac1{2x}\le\beta(t)\le\pi-\frac1{2x},
 \qquad \sin\beta(t)\ge\sin\frac1{2x}\ge\frac1{\pi x}.
\]
The second bound in \eqref{ipm:cascade-future-localization} controls the spatial direction error.

By \eqref{ipm:cascade-tube-comparison}, for $y\in\operatorname{supp}\chi$
and $s\le t\le S$,
\[
 \left|\frac{m_2(t,\ell y)}{|m(t,\ell y)|}-\sin\beta(t)\right|
 \le2\ell K^{2A+3}
 \le\frac1{10\pi x_{\rm tar}}\le\frac1{10\pi x}.
\]
Consequently, for $s\le t\le t_*$,
\[
 \frac{m_2(t,\ell y)}{|m(t,\ell y)|}
 \ge\frac1{\pi x}-\frac1{10\pi x}
 =\frac9{10\pi x}>0.
\]
For $t_*\le t\le S$, the retained-future angle bound instead gives
\[
 \frac{m_2(t,\ell y)}{|m(t,\ell y)|}
 \ge\frac12-\frac1{10\pi x_{\rm tar}}>0.
\]
This proves $m_2>0$ throughout $\operatorname{supp}\chi$ on
$[s,S]$. We use no phase sign during earlier history $[0,s]$.
\label{ipm:expanded-cascade-supported-sign}

The amplitude parameter also has a polynomial bound. Since
$\gamma(t)=\alpha a(t,0)$ and $|m(t,0)|^{-1}\le CK^A$,
\eqref{ipm:cascade-before-target} and \eqref{ipm:cascade-after-target}
justify the choice \eqref{ipm:cascade-gamma-choice}.
The history bound \eqref{ipm:cascade-small-history} similarly gives
\eqref{ipm:cascade-gamma-history}.
Neither $\alpha^{-1}$ nor the large raw amplification enters $\mathcal P$.
\end{proof}

\begin{corollary}[Leading pressure and gradient bounds]
\label{ipm:packet-budget}
Assume the hypotheses of Proposition~\ref{ipm:packet-history-geometry},
with the profile of Lemma~\ref{ipm:profilelemma}, and set
$C_\chi=\max(1,\|\chi\|_\infty)$. For the leading fields
in \eqref{ipm:intro-leading-matrices}, for all $t\in[0,S]$ and globally in space,
\begin{align}
 (\nabla^2\pi)_{\rm lead}&\le C_H(\delta h+he^{-c_0x})I,
 \label{ipm:cascade-h-budget}\\
 \|(\nabla\Theta)_{\rm lead}\|_\infty&\le C_{\rm lead}h,
 \qquad C_{\rm lead}:=2eC_\chi C_f^2.
 \label{ipm:cascade-gradient-increment}
\end{align}
The constants depend only on the fixed bump $\psi$ and cutoff $\chi$,
and are independent of the stage.
\end{corollary}
\begin{proof}
Proposition~\ref{ipm:packet-history-geometry} and the localization bound
\eqref{ipm:cascade-tube-comparison} give, on the support of $\chi$,
$\alpha\chi a|m|/\delta\le Ch$ throughout $[0,S]$ and
$\alpha\chi a|m|/\delta\le Ch e^{-c_0x}$ on $[0,s]$.
On $[s,S]$, $m_2>0$, so Corollary~\ref{ipm:profile-pressure-cost}
gives $(\nabla^2\pi)_{\rm lead}\le C\delta hI$. On $[0,s]$, its sign-free bound
gives $\|(\nabla^2\pi)_{\rm lead}\|\le Ch e^{-c_0x}$. This proves
\eqref{ipm:cascade-h-budget}. The formula
$(\nabla\Theta)_{\rm lead}\circ X=\alpha\chi a\,m f'_\delta$ and
$\|f'_\delta\|_\infty\le C_f^2/\delta$ from \eqref{ipm:profile-bounds}, together with the target normalization
\eqref{ipm:cascade-target-normalization}, give
\eqref{ipm:cascade-gradient-increment}. Both leading fields vanish
outside the transported support, so the bounds hold globally.
\end{proof}

\section{Analytic inputs and a fixed Gevrey calculus}
\label{ipm:calculus}

\subsection{Ordered derivatives, products, and composition}
\label{ipm:general-calculus}

This subsection records product, composition, and flow estimates in Gevrey norms. All tori used below have periods bounded
below by a fixed positive constant. Their measure is unnormalized Lebesgue
measure. By Lemma~\ref{ipm:seventh-foundations-sobolev}, the constants in
$H^s\hookrightarrow W^{1,\infty}$, the $H^s$ algebra estimate, and
multiplication $W^{s,\infty}\times H^s\to H^s$ are independent of the
torus side lengths when $s>d/2+1$. We fix $s=3$, which suffices on both
the two-dimensional label torus and the three-dimensional lifted torus.

For a function with finitely many components, let
\[
 D_n(f)=\sum_{|I|=n}\|\partial^I f\|_{L^\infty},
 \qquad
 E_n(f)=\sum_{|I|=n}\|\partial^I f\|_{H^s}.
\]
In these derivative norms, component norms are summed. The empty word
gives the corresponding component sum for the function.
Fix the $l$ positions of the differentiated word falling on the first
factor. Summing the coordinate choices at these positions and their
complement gives the product of the two shorter word sums, yielding
\begin{lemma}[Ordered-word product estimates]\label{ipm:word-products}
The preceding word norms satisfy
\begin{equation}\label{ipm:wordproduct}
 E_n(fg)\le C_s\sum_{l=0}^n\binom nl E_l(f)E_{n-l}(g).
\end{equation}
For $D_n$ and coefficient multiplication, respectively,
\[
 D_n(fg)\le\sum_{l=0}^n\binom nl D_l(f)D_{n-l}(g),\qquad
 E_n(af)\le C_s\sum_{l=0}^n\binom nl
 \left(\sum_{|I|=l}\|\partial^Ia\|_{W^{s,\infty}}\right)E_{n-l}(f).
\]
The constants are independent of the torus periods.
\end{lemma}

On a product torus with a $2\pi$-periodic angle, let $\langle f\rangle$
be the normalized angular average, viewed as an angle-independent field.
For $N_n=D_n$ or $E_n$,
\[
 N_n(\langle f\rangle)\le N_n(f),\qquad
 N_n(\partial_\theta^{-1}f)\le C N_n(f)\quad\text{if }\langle f\rangle=0.
\]
The primitive has zero angular mean, and $C$ is independent of $n$ and
the other torus periods.

For $R\ge1$, write $E_d(n)=R^{n+d}((n+d)!)^2$.
Unless stated otherwise, the growth constant $R$ is unchanged in the estimates.
\begin{lemma}[Uniform shifted products]\label{ipm:shiftcalculus}
For $n,d_1,d_2\ge0$,
\begin{equation}\label{ipm:shiftproduct}
 \sum_{l=0}^n\binom nl (l+d_1)!^2(n-l+d_2)!^2
 \le 3(n+d_1+d_2)!^2.
\end{equation}
For $0\le l\le n$ and every $d\ge0$,
\begin{equation}\label{ipm:shiftabsorption}
 \binom nl(l!)^2\frac{(n-l+d)!^2}{(n+d)!^2}\le\binom nl^{-1}\le1.
\end{equation}
Consequently, if $E_n(f)\le A_fE_{d_1}(n)$ and
$E_n(g)\le A_gE_{d_2}(n)$ for every $n\ge0$, then
\[
 E_n(fg)\le3C_sA_fA_gE_{d_1+d_2}(n),\qquad
 E_n(\nabla f)=E_{n+1}(f)\le A_fE_{d_1+1}(n).
\]
For an integer $q\ge0$ and a prefactor $B\le R^q$,
\[
 B E_d(n)\le E_{d+q}(n),\qquad
 \frac{E_{d+q}(n)}{E_d(n)}
 =R^q\left(\frac{(n+d+q)!}{(n+d)!}\right)^2\ge R^q.
\]
In particular, a prefactor $B\le Q^b$ is absorbed when $R\ge Q^r$
and $rq\ge b$, with $Q\ge1$ and $r>0$.
For a finite sum, if $E_n(f_j)\le A E_d(n)$ for $1\le j\le M$ and
$M\le R(d+1)^2$, then
\[
 E_n\!\left(\sum_{j=1}^M f_j\right)
 \le MAE_d(n)\le AE_{d+1}(n),\qquad
 \frac{E_{d+1}(n)}{E_d(n)}=R(n+d+1)^2.
\]
The same conclusions hold for $D_n$ using its product constant.
\end{lemma}
\begin{proof}
Increasing either $d_i$ shows
\[
 \binom{n+d_1+d_2}{l+d_1}\ge\binom nl.
\]
After division by $(n+d_1+d_2)!^2$, each summand in
\eqref{ipm:shiftproduct} is therefore at most $\binom nl^{-1}$.
For $n\ge4$, the sum of these reciprocals is bounded by
$2+2/n+2(n-3)/(n(n-1))<3$; the cases $n\le3$ are direct.
For \eqref{ipm:shiftabsorption}, the factorial quotient
$(n+d)!/(n-l+d)!$ increases with $d$. At $d=0$ the left side is
$\binom nl^{-1}$. The norm bounds follow from Lemma~\ref{ipm:word-products},
the identity $E_n(\nabla f)=E_{n+1}(f)$, and the displayed weight ratios.
\end{proof}

At a radius $\rho>0$ put
\[
 \|f\|_\rho=\sum_{n\ge0}\frac{\rho^n}{(n!)^2}E_n(f).
\]
The truncated sum through any $m$ is denoted by $\|f\|_{\rho,m}$.
Lemma~\ref{ipm:word-products} and \eqref{ipm:shiftabsorption} with $d=0$ give
\[
 \|fg\|_\rho\le C_s\|f\|_\rho\|g\|_\rho,
 \qquad
 \|fg\|_{\rho,m}\le C_s\|f\|_{\rho,m}\|g\|_{\rho,m}.
\]

\begin{lemma}[Shifted bounds and a fixed radius decrease]
\label{ipm:radius-calculus}
If $E_n(f)\le A R^{n+d}(n+d)!^2$, then, for $4\rho R\le1/2$,
\begin{equation}\label{ipm:shifttoradius}
 \|f\|_\rho\le 2A(4R)^d(d!)^2.
\end{equation}
A fixed decrease of the radius also gives
\begin{equation}\label{ipm:radiusderivative}
 \|\nabla f\|_{\rho/2}
 \le C\rho^{-1}\|f\|_\rho,
\end{equation}
\end{lemma}
\begin{proof}
The factorial inequality $(n+d)!\le2^{n+d}n!d!$ reduces
\eqref{ipm:shifttoradius} to a geometric series. Reindexing the
derivative sum proves \eqref{ipm:radiusderivative}, since
$\sup_n2^{-n}(n+1)^2<\infty$.
\end{proof}

\begin{lemma}[Reciprocals and composition]\label{ipm:composition}
Suppose $a\ge\mu>0$ and $D_n(a)\le A R_0^n(n!)^2$.
Then
\[
 D_n(a^{-1})\le C\mu^{-1}
       [C R_0(1+A/\mu)]^n(n!)^2.
\]
If $D_n(f)\le A_oR_o^n(n!)^2$ and
$D_n(h)\le A_iR_i^n(n!)^2$, with $R_i,R_o\ge1$, then
for $g=\mathrm{id}+h$,
\begin{equation}\label{ipm:compositionbound}
 D_n(f\circ g)\le C A_o R_{\rm new}^n(n!)^2,
 \quad
 R_{\rm new}=C\max\{R_i,R_o(d+2A_iR_i)\}.
\end{equation}
The constants depend only on the dimensions and norm conventions.
\end{lemma}
\begin{proof}
Differentiate $aa^{-1}=1$. For $n\ge1$, the derivative of the inverse
is bounded by $\mu^{-1}$ times the sum of terms with at least one
derivative on $a$. After division by
$\mu^{-1}R^n(n!)^2$, the induction for lower derivatives is closed by
\[
 \frac A\mu\sum_{l\ge1}(R_0/R)^l\le\frac12.
\]
One may take $R=C R_0(1+A/\mu)$. This argument applies directly to
word sums, so it introduces no order-dependent change of radius.

For composition, use the partition form of the chain rule. After
normalizing the $n$th derivative by $(n!)^2$, the term involving $q$
outer derivatives and positive block lengths $j_1+\cdots+j_q=n$
is bounded by
\[
 A_o R_o^q
 \frac{q!\prod_{i=1}^q j_i!}{n!}
 \prod_{i=1}^q\frac{D_{j_i}(g)}{(j_i!)^2}.
\]
Summing over ordered block lengths is an upper bound for the chain-rule
sum; the factor $1/q!$ for unordered blocks has already been included.
For positive $j_1,\ldots,j_q$ with sum $n$,
\begin{equation}\label{ipm:composition-ratio}
 \frac{q!\prod_{i=1}^q j_i!}{n!}\le1,
\end{equation}
since the multinomial coefficient $n!/\prod j_i!$ is at least $q!$.
The coordinate and component sums are bounded by the corresponding
products of word sums.

Set
\[
 U(\zeta)=\sum_{j\ge1}\frac{D_j(g)}{(j!)^2}\zeta^j.
\]
Coefficientwise, the positive-order derivative generating series of
$f\circ g$ is bounded by
\begin{equation}\label{ipm:composition-series}
 A_o\frac{R_oU(\zeta)}{1-R_oU(\zeta)},
 \qquad R_oU(\zeta)<1.
\end{equation}
Moreover
\[
 U(\zeta)\le d\zeta+
               \frac{A_iR_i\zeta}{1-R_i\zeta}.
\]
At $\zeta=R_{\rm new}^{-1}$ with the fixed constant in
\eqref{ipm:compositionbound} large enough, $R_i\zeta\le1/2$ and
$R_oU\le1/2$. The value of the generating series is bounded by
$C A_o$, which bounds each nonnegative coefficient. The zeroth
derivative is bounded directly by $D_0(f)$.
\end{proof}

\begin{lemma}[A small Gevrey vector field has a controlled flow]
\label{ipm:smallflow}
Let $b(t,z)$ be smooth and periodic in $z\in\mathbb R^d$, and assume
\[
 \sup_{0\le t\le S}D_n(b(t))\le B R^n(n!)^2,
 \qquad SBR\le c_d,
\]
where $R\ge1$ and $c_d>0$ is a sufficiently small fixed constant.
Its flow from $\Phi(0,z)=z$ satisfies
\begin{equation}\label{ipm:smallflowbound}
 \sup_t D_n(\Phi(t)-\mathrm{id})
       \le C_dBS(C_dR)^n(n!)^2.
\end{equation}
The first time derivative satisfies
\begin{equation}\label{ipm:smallflow-first-time}
 \sup_{0\le t\le S}D_n(\Phi_t(t))\le C_dB(C_dR)^n(n!)^2.
\end{equation}
If also $\sup_{0\le t\le S}D_n(b_t(t))\le B_tR_t^n(n!)^2$, with $R_t\ge1$, then
\begin{equation}\label{ipm:smallflow-second-time}
 \sup_{0\le t\le S}D_n(\Phi_{tt}(t))
 \le C_d(B_t+B^2R)\bigl(C_d\max\{R,R_t\}\bigr)^n(n!)^2.
\end{equation}
These estimates hold for every $n\ge0$, with $C_d$ depending only on $d$.
\end{lemma}
\begin{proof}
For a finite derivative cutoff $J$, define
\[
 U_J(t,\zeta)=d\zeta+
       \sum_{n=1}^J\frac{D_n(\Phi(t)-\mathrm{id})}{(n!)^2}\zeta^n.
\]
The preceding partition calculation and $\Phi_t=b\circ\Phi$ give
\[
 U_J(t,\zeta)\le d\zeta+
       \int_0^t\frac{BRU_J(\tau,\zeta)}{1-RU_J(\tau,\zeta)}\,d\tau
\]
as long as $RU_J<1$. This holds first on a short interval, since
$U_J(0,\zeta)=d\zeta$. Take $\zeta=(4dR)^{-1}$.
On the bootstrap interval $RU_J\le1/2$, Gronwall gives
$U_J\le d\zeta\exp(2BRS)$. Choosing $c_d$ small closes
$U_J\le2d\zeta$ with a strict margin. Substitution into the integral
inequality yields $U_J-d\zeta\le C BS$.
This proves \eqref{ipm:smallflowbound} for positive orders, uniformly
in $J$; the zeroth order follows from $|\Phi-z|\le BS$.

For $h=\Phi-\mathrm{id}$, \eqref{ipm:smallflowbound} gives size
$A_i=C_dBS$ and growth $R_i=C_dR$, so $A_iR_i\le C_d^2c_d$.
Thus \eqref{ipm:compositionbound} applied to $\Phi_t=b\circ\Phi$
gives \eqref{ipm:smallflow-first-time}.
Also, $(n+1)^2\le4^n$ and Lemmas~\ref{ipm:word-products}
and~\ref{ipm:shiftcalculus} give, uniformly in time,
\[
 D_n(\nabla_zb)=D_{n+1}(b)\le BR(4R)^n(n!)^2,
 \qquad
 D_n((\nabla_zb)b)\le C_dB^2R(4R)^n(n!)^2.
\]
Consequently
\[
 D_n\bigl(b_t+(\nabla_zb)b\bigr)
 \le C_d(B_t+B^2R)\bigl(4\max\{R,R_t\}\bigr)^n(n!)^2.
\]
Applying \eqref{ipm:compositionbound} to
$\Phi_{tt}=(b_t+(\nabla_zb)b)\circ\Phi$ proves
\eqref{ipm:smallflow-second-time}.
\end{proof}

\subsection{Parent coefficient bounds}

Let $P,X,K$ satisfy the parent class \eqref{ipm:parent-class}.
Choose $0<\ell<1$ and a unit vector $m_0$, and set $y=z/\ell$ on
$\mathbb T_\ell^2=(\mathbb R/(2\pi/\ell)\mathbb Z)^2$.
Recall the coefficient fields
\begin{equation}\label{ipm:label-coefficients}
 \begin{aligned}
 F(t,y)&=D_zX(t,\ell y),& G&=F^{-1}F^{-T},& a_2&=F^{-1}e_2,\\
 R_0(y)&=\nabla_zP(0,\ell y),&&\\
 L&=I-\frac{Gm_0\otimes m_0}{m_0\cdot Gm_0},&
 b_0&=-cLa_2,&c_A^{\ell}(t,y)&=c_A(t,\ell y)=b_0\cdot R_0.
 \end{aligned}
\end{equation}
\begin{lemma}
\label{ipm:packet-calculus}
Under \eqref{ipm:parent-class}, the fields in
\eqref{ipm:label-coefficients} and $F$, $F_t$, $F_{tt}$, $F^TF$, $G_t$, $(a_2)_t$, $(F^TF)_t$
expressed in the unscaled labels $z$ satisfy, for one fixed exponent $A_0$,
\begin{equation}\label{ipm:derived-coefficients}
 \sum_{|I|=n}\|\partial_z^IC\|_{W_z^{s,\infty}}
 \le K^{A_0}(K^{A_0})^n(n!)^2,
 \qquad G\ge K^{-6}I.
\end{equation}
\end{lemma}
\begin{remark}
Derivatives in $y$ are more favorable because they bring powers of $\ell<1$.
\end{remark}
\begin{proof}
From \eqref{ipm:parent-class}, a fixed derivative shift and
$(n+j)!\le2^{n+j}n!j!$ show that $F,F_t,F_{tt}$ have size and growth
bounded by fixed powers of $K$. Since $\det F=1$ in dimension two,
$F^{-1}=\operatorname{adj}F=(\operatorname{cof}F)^T$ has the same bounds.
The product and reciprocal estimates of Lemmas~\ref{ipm:word-products},
\ref{ipm:shiftcalculus}, and~\ref{ipm:composition} give
\eqref{ipm:derived-coefficients} for all the listed fields, using
$(F^{-1})_t=-F^{-1}F_tF^{-1}$ to bound $G_t$ and $(a_2)_t$, and
$(F^TF)_t=F_t^TF+F^TF_t$.

The proof of Lemma~\ref{ipm:tube-localization} gives $\|F\|\le K^3$, so
$v^TGv=|F^{-T}v|^2\ge\|F\|^{-2}|v|^2\ge K^{-6}|v|^2$.
\end{proof}

\section{The mean equation and selected initial data}
\label{ipm:sec-mean}

We select initial data for the phase-mean equation \eqref{ipm:intro-mean-step} by imposing a terminal
condition on the displacement. Assume the parent class
\eqref{ipm:parent-class} and $\nabla^2p\le I$ on $[0,S]$.
On $\mathbb T_\ell^2$, use the coefficients in
\eqref{ipm:label-coefficients} and define
\begin{equation}\label{ipm:mean-darcy}
 \mathcal V_{\mathrm{mean}}B:=-ca_2B-G\nabla_y\bar q,
 \qquad \operatorname{div}_y(\mathcal V_{\mathrm{mean}}B)=0,
\end{equation}
where $\bar q$ is periodic and its gradient is uniquely determined.

All integrals use unnormalized Lebesgue measure. Write
$|v|_n=\sum_{|I|=n}\|\partial_y^Iv\|_{H^3}$ and
$E_d(n)=R^{n+d}((n+d)!)^2$, with ordered derivative words $I$.

\begin{proposition}\label{ipm:meaninverse}
Under the parent hypotheses \eqref{ipm:parent-class} and
$\nabla^2p\le I$ on $[0,S]$, there are fixed exponents $A_1,A_2$ such
that the following holds for every $R\ge K^{A_1}$ and every integer
$d\ge0$. Suppose $f(t,y)$ is smooth, has zero spatial integral for
every $t$, and satisfies
\[
 \sup_{0\le t\le S}|f(t)|_n\le A E_d(n)\qquad(n\ge0).
\]
Then there are a choice of initial data and a solution $B$ of
\begin{equation}\label{ipm:mean-equation}
 B_t+(\mathcal V_{\mathrm{mean}}B)\cdot R_0=f
\end{equation}
both depending linearly on $f$, such that
\begin{equation}\label{ipm:mean-bound}
 \sup_t|(B,B_t,\mathcal V_{\mathrm{mean}}B,\nabla_y\bar q)(t)|_n
 \le K^{A_2}A E_d(n)\qquad(n\ge0).
\end{equation}
The same shift $d$ occurs on both sides. The selected
$B$ has zero spatial integral at every time. For an odd parent and odd
source it is odd. The constants are uniform as the torus is enlarged.
\end{proposition}

\begin{proof}[Proof of Proposition~\ref{ipm:meaninverse}]
\textbf{Step 1. Coefficient constants and multiplication.}
We first fix constants which will make the proof quantitative.
By the previously derived parent bounds, choose
$\mathsf M,\mathsf R\ge2$, both bounded by fixed powers of $K$, such that
\begin{equation}\label{ipm:expanded-mean-coefficients}
 \sum_{|I|=n}\|\partial^IC(t)\|_\infty
 \le\mathsf M\mathsf R^n(n!)^2,
 \quad C\in\{F,F^TF,cF^Te_2R_0^T,R_0,G,a_2,c_A^{\ell}\},
 \quad \|F^{-1}\|_\infty\le\mathsf M.
\end{equation}
Increase them, still by fixed powers of $K$, to cover the first time
derivatives $G_t,(a_2)_t$ needed later.

To bound coefficient fields in $W^{s,\infty}$, use the elementary
fixed-shift estimate $(n+j)!=\binom{n+j}{n}n!j!\le2^{n+j}n!j!$ to obtain
\begin{equation}\label{ipm:expanded-mean-multiplier-class}
 \sum_{|I|=n}\|\partial^IC\|_{W^{s,\infty}}
 \le L_s R_s^n(n!)^2,
 \quad R_s=4\mathsf R,
 \quad L_s=C_s\mathsf M(4\mathsf R)^s(s!)^2.
\end{equation}
Here $C_s$ depends only on $s$ and dimension; a finite sum over the
listed coefficient fields may be included by increasing $C_s$.
Write $\Lambda_s=C_sL_s$, increasing $C_s$ to cover all fixed component sums.
For every $R\ge2R_s$, every $d\ge0$, and every field with shifted input bound
$|v(t)|_n\le A_*E_d(n)$, the coefficient product estimate of
Lemma~\ref{ipm:word-products} and \eqref{ipm:shiftabsorption} give
\begin{equation}\label{ipm:expanded-mean-multiplication}
 \begin{split}
 |Cv(t)|_n
 &\le\Lambda_s A_*E_d(n)
       \sum_{l=0}^n(R_s/R)^l
       \binom nl\frac{(l!)^2((n-l+d)!)^2}{((n+d)!)^2}\\
 &\le2\Lambda_s A_*E_d(n).
 \end{split}
\end{equation}
This formula shows explicitly that coefficient multiplication does
not enlarge the chosen $R$, even for an arbitrarily shifted input bound.

\textbf{Step 2. Scaling and the exact displacement identities.}
The relevant physical domain for the label variable $y$ is the
enlarged torus. On it introduce the rescaled parent
\[
 \widetilde P(t,\widetilde x)=\ell^{-1}P(t,\ell\widetilde x),
 \quad \widetilde u(t,\widetilde x)=\ell^{-1}u(P)(t,\ell\widetilde x),
 \quad \widetilde p(t,\widetilde x)=\ell^{-2}p(t,\ell\widetilde x).
\]
Its flow is $\widetilde X(t,y)=\ell^{-1}X(t,\ell y)$. Direct differentiation gives
\[
 \nabla_y\widetilde X(t,y)=\nabla_zX(t,\ell y)=F(t,y),\quad
 \nabla_{\widetilde x}\widetilde P(t,\widetilde x)
 =\nabla_xP(t,\ell\widetilde x),\quad
 \nabla_{\widetilde x}^2\widetilde p(t,\widetilde x)
 =\nabla_x^2p(t,\ell\widetilde x).
\]
Thus the rescaled equation has the same coefficient $c$, the same
upper Hessian bound, and initial gradient $R_0$, with no extra factor
of $\ell$. We suppress the tildes in the following identities only.

With the notation of Section~1.1, the parent identities
\eqref{ipm:raybasic} still hold after rescaling. For a linearized density
$\varrho$, introduce the displacement $\eta$ and
$Q=\varrho+r\cdot\eta$ as in \eqref{ipm:intro-displacement-definition}.
The calculation there gives the displacement system
\eqref{ipm:intro-displacement}.
We will construct a particular solution by imposing $Q(0)=0$ and
$\eta(S)=0$. These conditions select the initial density
$\varrho(0)=-r(0)\cdot\eta(0)$.

In parent labels, write $\eta\circ\widetilde X=Fz$ and denote the
pullbacks of $Q,f,\pi$ by $\widetilde X$ by the same letters. Since
$F_t=(M\circ\widetilde X)F$ and
$r\circ\widetilde X=F^{-T}R_0$, system
\eqref{ipm:intro-displacement} becomes
\begin{equation}\label{ipm:terminal-mean-system}
 \begin{aligned}
 Q_t&=f,\qquad Q(0)=0,\\
 \mathcal Kz_t-\mathcal Bz+\nabla_y\pi&=g,
 \qquad\operatorname{div}_yz=0,\qquad z(S)=0,\\
 \mathcal K&=F^TF,\qquad
 \mathcal B=cF^Te_2R_0^T,\qquad g=-cF^Te_2Q.
 \end{aligned}
\end{equation}
To check the middle line directly, multiply the displacement equation
by $F^T$ after pullback. The coefficient of $z$ becomes
$F^T((M+H)\circ\widetilde X)F=-cF^Te_2R_0^T=-\mathcal B$.
The Piola identity gives
$(\operatorname{div}_x\eta)\circ\widetilde X
 =\operatorname{div}_y\bigl(F^{-1}(\eta\circ\widetilde X)\bigr)
 =\operatorname{div}_yz$, since $\det F=1$.

\textbf{Step 3. Existence and smoothness of the solution $z$.}
Let $\Pi$ be the $L^2$-orthogonal projection from
$L^2(\mathbb T_\ell^2;\mathbb R^2)$ onto the closed subspace
\[
 \mathcal H=\{v\in L^2(\mathbb T_\ell^2;\mathbb R^2):
 \operatorname{div}_yv=0\text{ in distributions}\}.
\]
This space includes constant vector fields; no zero-mean condition is imposed.
On $\mathcal H$ define
\[
 T_{\mathcal K}(t)v=\Pi(\mathcal K(t)v).
\]
For each $t$, $T_{\mathcal K}(t)$ is a bounded self-adjoint operator on $\mathcal H$, and
\begin{equation}\label{ipm:expanded-mean-coercivity}
 \langle v,T_{\mathcal K}v\rangle
 =\|Fv\|_2^2\ge\mathsf M^{-2}\|v\|_2^2.
\end{equation}
Coercivity gives an inverse on $\mathcal H$ of norm at most
$\mathsf M^2$: its range is closed by the lower bound, and its
orthogonal complement is its kernel, which is zero. Since
$\Pi\nabla_y\pi=0$, applying $\Pi$ to the equation for $z$ gives
\begin{equation}\label{ipm:expanded-mean-terminal-ode}
 z_t=T_{\mathcal K}^{-1}\Pi(\mathcal Bz+g),\qquad z(S)=0.
\end{equation}
This is an ODE with bounded time-continuous operators on $\mathcal H$.
Standard local well-posedness for ODEs in $\mathcal H$, followed by
continuation over $[0,S]$ using linearity and boundedness of the
coefficients, gives a unique $z\in C^1([0,S];\mathcal H)$.

We justify spatial regularity of $z$ before differentiating it.
The projection commutes with translations. If $T_{\mathcal K}v=h$,
the difference quotient in any coordinate satisfies
\[
 T_{\mathcal K}\delta_av
 =\delta_ah-\Pi((\delta_a\mathcal K)\tau_av),
\]
where $\tau_a$ is translation by $a$ in that coordinate. The $L^2$ operator bound for $T_{\mathcal K}^{-1}$ and the
$W^{1,\infty}$ bound for $\mathcal K$ give uniform
difference quotient bounds when $h\in H^1$, and hence $v\in H^1$.
Iterating this argument, or differentiating after each established
regularity step, proves $T_{\mathcal K}^{-1}:H^j\cap\mathcal H\to H^j$
for every fixed $j$. Its norm is a fixed polynomial in the coefficient
bounds through order $j$. The inverse is continuous in time, as
follows from the resolvent identity. Solving
\eqref{ipm:expanded-mean-terminal-ode} in each $H^j$ and using $L^2$
uniqueness identifies these solutions with $z$. Thus $z$ is smooth.

For this $z$, $g+\mathcal Bz-\mathcal Kz_t$ has zero Leray
projection (in particular, zero spatial mean), so it is a periodic
gradient. This gives $\nabla_y\pi=\nabla_y\bar q$.
Only the gradient is estimated in \eqref{ipm:mean-bound}; the zero-mean scalar potential exists
by Fourier series, but its norms are not needed.

\textbf{Step 4. Backward energy, including arbitrary finite families.}
The identities of the exact parent imply
\begin{equation}\label{ipm:expanded-mean-matrix-energy}
 \operatorname{sym}\mathcal B+\tfrac12\mathcal K_t
 =F^T\operatorname{sym}\bigl((M+ce_2r^T)\circ\widetilde X\bigr)F
 =-F^T(H\circ\widetilde X)F.
\end{equation}
For an arbitrary vector source $g$, testing the equation for $z$ by
$z$ therefore gives
\begin{equation}\label{ipm:backward-energy}
 \frac12\frac d{dt}\|Fz\|_2^2
 =-\int(Fz)^T(H\circ\widetilde X)(Fz)+\int z\cdot g.
\end{equation}
The pressure term vanishes by the divergence constraint. Since $H\le I$,
$E(t)=\|Fz(t)\|_2$ satisfies the lower differential inequality
$E'\ge-E-\mathsf M\|g(t)\|_2$, interpreted by regularizing the norm
at zero. Integrating backward from $E(S)=0$ gives
\begin{equation}\label{ipm:expanded-mean-pointwise-energy}
 \|z(t)\|_2\le
 \mathsf M^2\int_t^S e^{\tau-t}\|g(\tau)\|_2\,d\tau.
\end{equation}
The projected ODE also gives, at the same time,
\begin{equation}\label{ipm:expanded-mean-pointwise-time}
 \|z_t(t)\|_2\le
 \mathsf M^2\bigl(\mathsf M\|z(t)\|_2+\|g(t)\|_2\bigr).
\end{equation}

We will apply the estimates to families of derivatives of $z$.
The required norm is $\sup_t\sum_\alpha$, whereas estimating each
member separately would give the larger $\sum_\alpha\sup_t$ on the
source. For a finite family $z^\alpha$ with sources $g^\alpha$, indexed
by $\mathcal A$, sum \eqref{ipm:expanded-mean-pointwise-energy} before
taking the time supremum. Together with
\eqref{ipm:expanded-mean-pointwise-time} and $S\le1$, this gives
\begin{equation}\label{ipm:expanded-mean-family-L2}
 \sup_t\sum_{\alpha\in\mathcal A}
   (\|z^\alpha(t)\|_2+\|z_t^\alpha(t)\|_2)
 \le I_0\sup_t\sum_{\alpha\in\mathcal A}\|g^\alpha(t)\|_2,
 \qquad I_0=4e\mathsf M^5.
\end{equation}
The constant is independent of the cardinality of $\mathcal A$.

\textbf{Step 5. A fixed Sobolev base with explicit polynomial constants.}
The norm $\sum_{j=0}^s\sum_{|I|=j}\|\partial^Iv\|_2$ is equivalent
to $\|v\|_{H^s}$ with constants depending only on $s$ and dimension,
not on the torus period. This follows by comparing their Fourier
weights; the final finite $\ell^1$--$\ell^2$ comparison concerns only
the fixed number of base derivatives, since $s$ and the dimension are fixed.

Commuting a word $I$ through the equation for $z$ gives
\begin{equation}\label{ipm:expanded-mean-commuted-system}
 \mathcal K\partial^Iz_t-\mathcal B\partial^Iz
       +\nabla\partial^I\pi
 =\partial^Ig+
 \sum_{\emptyset\ne J\subseteq\{1,\ldots,|I|\}}
 \bigl[(\partial^{I_J}\mathcal B)\partial^{I_{J^c}}z
       -(\partial^{I_J}\mathcal K)\partial^{I_{J^c}}z_t\bigr].
\end{equation}
Grouping the position subsets by their cardinality produces $\binom{|I|}{|J|}$. Terminal zero data and the divergence
constraint are preserved by differentiation.

For a finite family put
\[
 J_j=\sup_t\sum_\alpha\sum_{|I|=j}
  (\|\partial^Iz^\alpha(t)\|_2+
   \|\partial^Iz_t^\alpha(t)\|_2),
 \quad F_s=\sup_t\sum_\alpha\|g^\alpha(t)\|_{H^s}.
\]
Let $M_s=C_s\mathsf M\mathsf R^s(s!)^2$, chosen to dominate all
coefficient derivative sums through order $s$.
For $0\le j\le s$, the family estimate and the commutator identity give
\[
 J_j\le I_0\left(C_sF_s+
 M_s\sum_{l=1}^j\binom jl J_{j-l}\right).
\]
With $V_j=\sum_{i=0}^jJ_i$, it follows that
$V_j\le(1+2^sI_0M_s)V_{j-1}+C_sI_0F_s$.
Thus a valid joint $H^s$ inverse constant is
\begin{equation}\label{ipm:expanded-mean-base-constant}
 I_s=C_s(s+1)I_0(1+2^sI_0M_s)^s,
\end{equation}
after increasing the fixed norm-equivalence constant $C_s$.
Explicitly, uniformly over finite families,
\begin{equation}\label{ipm:expanded-mean-family-Hs}
 \sup_t\sum_\alpha
   (\|z^\alpha(t)\|_{H^s}+\|z_t^\alpha(t)\|_{H^s})
 \le I_s\sup_t\sum_\alpha\|g^\alpha(t)\|_{H^s}.
\end{equation}
Because $s$ is fixed, $I_s$ is a polynomial in
$\mathsf M,\mathsf R$ and hence a fixed power of $K$.

\textbf{Step 6. Arbitrary additional derivatives at the same growth rate.}
Apply \eqref{ipm:expanded-mean-family-Hs} to
\eqref{ipm:expanded-mean-commuted-system} for all $|I|=n$, and use
\eqref{ipm:expanded-mean-multiplier-class} to bound the commutator sum.
Writing
$U_n=\sup_t|(z,z_t)(t)|_n$ and $G_n=\sup_t|g(t)|_n$, we obtain
\begin{equation}\label{ipm:mean-triangular}
 U_n\le I_s\left[G_n+
 \Lambda_s\sum_{l=1}^n\binom nl R_s^l(l!)^2U_{n-l}\right].
\end{equation}
The binomial coefficient follows from Lemma~\ref{ipm:word-products}.

Choose
\begin{equation}\label{ipm:expanded-mean-common-R}
 R\ge R_{\rm mean}:=\max\{2R_s,8I_s\Lambda_sR_s\}.
\end{equation}
Suppose $G_n\le A_gE_d(n)$ for an arbitrary $d\ge0$.
We claim
\begin{equation}\label{ipm:mean-all-order-bound}
 U_n\le2I_sA_gE_d(n)\qquad(n\ge0).
\end{equation}
The case $n=0$ follows
from \eqref{ipm:mean-triangular}. In the induction step normalize a
commutator by $E_d(n)$ and use \eqref{ipm:shiftabsorption}. Its total
contribution is at most
\[
 2I_s^2\Lambda_s A_g
       \sum_{l=1}^n(R_s/R)^l
 \le\tfrac12 I_s A_g,
\]
after the harmless numerical enlargement in
\eqref{ipm:expanded-mean-common-R}. Adding the direct source gives a
strict bound below $2I_sA_g$. This induction is uniform in $d$ and uses
one fixed $R$, with no radius reduction at each inversion.

\textbf{Step 7. Reconstruction, zero mean, and parity.}
Define
\begin{equation}\label{ipm:mean-reconstruction}
 B=Q-R_0\cdot z,\qquad
 \mathcal V_{\mathrm{mean}}B=z_t,\qquad
 \nabla_y\bar q=\nabla_y\pi=g+\mathcal Bz-\mathcal Kz_t.
\end{equation}
To verify the second identity, use
$g=-cF^Te_2Q$ in the equation for $z$ to get
$\mathcal Kz_t=-cF^Te_2B-\nabla_y\bar q$.
Multiplication by $G=\mathcal K^{-1}$ gives precisely
\eqref{ipm:mean-darcy}, and $\operatorname{div}z_t=0$ gives its
constraint. Since $R_0$ is time independent,
$B_t=f-R_0\cdot z_t$, which proves \eqref{ipm:mean-equation}.

For the bounds, $Q(t)=\int_0^tf(\tau)\,d\tau$ satisfies
\[
 \sup_{0\le t\le S}|Q(t)|_n\le SA E_d(n)\qquad(n\ge0).
\]
Multiplication \eqref{ipm:expanded-mean-multiplication} gives
\[
 \sup_{0\le t\le S}|g(t)|_n\le 2c\Lambda_sSA E_d(n)\qquad(n\ge0).
\]
The arbitrary-shift estimate \eqref{ipm:mean-all-order-bound} gives
\[
 \sup_{0\le t\le S}|(z,z_t)(t)|_n
 \le 4cI_s\Lambda_sSA E_d(n)\qquad(n\ge0).
\]
Use \eqref{ipm:expanded-mean-multiplication} again in each term of
\eqref{ipm:mean-reconstruction} and in $B_t=f-R_0\cdot z_t$.
The total output amplitude is bounded, for example, by
\[
 C_c(1+\Lambda_s)^2(1+I_s)A,
\]
after a fixed enlargement of the numerical constant. It is a
polynomial in the prescribed parent constants. The same argument
also controls $G\nabla_y\bar q$ if desired. It requires no source time
derivative.

Since $\int f(t)=0$, one has $\int Q(t)=0$. The vector $R_0$ is the
gradient of the periodic rescaled initial density
$\widetilde P(0,y)=\ell^{-1}P(0,\ell y)$. Hence
\[
 \int R_0\cdot z
 =-\int\widetilde P(0,y)\operatorname{div}z\,dy=0.
\]
It follows that $B$ has zero mean, including at the selected initial
time, when $B(0)=-R_0\cdot z(0)$. The periodic physical Darcy velocity
then has zero mean, as follows by integrating its Darcy law.

For an odd parent $F,\mathcal K,\mathcal B,R_0$ are even under
$y\mapsto-y$. An odd $f$ makes $Q$ and $g$ odd. The Leray projection
commutes with reflection, and the coercive inverse preserves its
parity subspaces. Uniqueness of the ODE for $z$ with $z(S)=0$ gives odd $z,z_t$;
then $B$ is odd. The pressure gradient is odd and its zero-mean scalar
potential is even. This completes the proof of
Proposition~\ref{ipm:meaninverse}.
\end{proof}

\section{The localized scalar and lifted Darcy inverses}
\label{ipm:sec-inverses}

Choose $k\ge2$ satisfying the integer-carrier condition
\begin{equation}\label{ipm:carrier-lattice}
 n_0:=\frac{k}{\ell}m_0\in\mathbb Z^2.
\end{equation}
On the lifted domain $\mathbb T_\ell^2\times\mathbb T$, set
$D=k^{-1}\nabla_y+m_0\partial_\theta$.
Angular averages use normalized angular measure.

\subsection{The localized scalar equation and its estimates}

Use the fixed cutoff $\chi$ and profiles $f_\delta$ of
Lemma~\ref{ipm:profilelemma}, with $0<\delta\le\delta_0$.
Constants may depend on these fixed choices, but not on the expansion
order, derivative order, amplitude seed, or layer index.

Fix $s_0\in[0,S]$ and solve
\begin{equation}\label{ipm:leading-amplitude}
 a_t+c_A^{\ell}(t,y)a=0,\quad a(s_0,y)=1,
 \qquad \gamma(t)=\alpha a(t,0),\quad \alpha>0.
\end{equation}
Set
\begin{equation}\label{ipm:packet-parameter}
 \mathcal P=2+K+\ell^{-1}+\delta^{-1}+\Gamma,
 \qquad \Gamma\ge\max(1,\sup_{t\in[0,S]}\gamma(t)).
\end{equation}

\begin{proposition}
\label{ipm:scalar-inverse}
Assume the parent class \eqref{ipm:parent-class} and $\ell K^{2A+3}\le\varepsilon_*$,
where $\varepsilon_*>0$ is sufficiently small and $A$ is the exponent
in Lemma~\ref{ipm:tube-localization}.
Let $R_E=\max\{8,R_\chi\}$ be the extended propagator growth rate
constructed below. It depends only on the fixed cutoff and the Sobolev
index $s=3$, and is independent of $K$.
For every $R\ge2R_E$ and $d\ge0$, a smooth source supported in
$\operatorname{supp}\chi$ with
$|f(t)|_n\le\gamma(t)A_*E_d(n)$ has a unique zero-initial solution of
\[
 A_t+c_A^{\ell}A=f,\qquad A(0)=0,
\]
satisfying $|(A,A_t)(t)|_n\le K^C\gamma(t)A_*E_d(n)$.
The solution preserves support and zero angular mean.
\end{proposition}
\begin{proof}
Normalize by $\gamma(t)=\alpha a(t,0)$. Since
$\gamma'/\gamma=-c_A^{\ell}(t,0)$, the equation becomes
\begin{equation}\label{ipm:normalized-scalar-equation}
 \partial_t\!\left(\frac{A}{\gamma}\right)
 +\bigl(c_A^{\ell}(t,y)-c_A^{\ell}(t,0)\bigr)\frac{A}{\gamma}
 =\frac{f}{\gamma}.
\end{equation}
For $|y|\le1$ and any two times $\sigma,t\in[0,S]$, in either order,
set
\[
 L_{\sigma,t}(y)
 =-\int_\sigma^t[c_A^{\ell}(\tau,y)-c_A^{\ell}(\tau,0)]\,d\tau,
 \qquad E_{\sigma,t}=e^{L_{\sigma,t}}.
\]
By \eqref{ipm:expanded-cascade-amplitude-coefficient}, the localization
condition with sufficiently small $\varepsilon_*$ gives
\[
 \ell B_c\le1,\qquad
 \ell A_cB_c\le\varepsilon_*K^{-3}\le\epsilon_0,\qquad
 \epsilon:=S A_c\ell B_c\le\epsilon_0,
\]
where $\epsilon_0$ is a sufficiently small absolute constant.
Using $c_A^{\ell}(t,y)=c_A(t,\ell y)$ and
\eqref{ipm:expanded-cascade-amplitude-coefficient}, the mean value theorem gives
\begin{equation}\label{ipm:expanded-mean-log-propagator}
 \|L_{\sigma,t}\|_\infty\le\epsilon,
 \qquad
 \sum_{|I|=n}\|\partial_y^I L_{\sigma,t}\|_\infty
 \le S A_c(\ell B_c)^n(n!)^2
 \le\epsilon(n!)^2\quad(n\ge1).
\end{equation}

We next bound the spatial derivatives of $E_{\sigma,t}=e^{L_{\sigma,t}}$
uniformly in $\sigma,t$. Take $\epsilon_0\le\log2$. In the normalized
derivative generating series, the positive derivatives of $L$ are
bounded by $U(\zeta)=\epsilon\zeta/(1-\zeta)$.
The partition estimate \eqref{ipm:composition-ratio} bounds the
positive derivative series of $e^L$ by
\[
 \frac{2U(\zeta)}{1-U(\zeta)}
 =\frac{2\epsilon\zeta}{1-(1+\epsilon)\zeta},
\]
since the outer derivatives of the exponential are at most two on
$|L|\le\epsilon$. Comparing coefficients gives
\[
 \sum_{|I|=n}\|\partial_y^IE_{\sigma,t}\|_\infty
 \le2\epsilon(1+\epsilon)^{n-1}(n!)^2
 \le2\cdot2^n(n!)^2\qquad(n\ge1).
\]
The zeroth-order bound is $E_{\sigma,t}\le2$. For the base derivatives
$0\le j\le s$, use $((n+j)!)^2\le4^{n+j}(n!)^2(j!)^2$, as in
\eqref{ipm:expanded-mean-multiplier-class}, to obtain
\begin{equation}\label{ipm:centered-propagator}
 \tfrac12\le E_{\sigma,t}\le2,
 \qquad
 \sum_{|I|=n}\|\partial_y^IE_{\sigma,t}\|_{W^{s,\infty}}
 \le C_s8^n(n!)^2\qquad(n\ge0)
\end{equation}
on $|y|\le1$.
To obtain global multiplier bounds, choose a fixed Gevrey cutoff
$0\le\chi_*\le1$, equal to one near $\operatorname{supp}\chi$ and
supported in $B_1$, with
\[
 \sum_{|I|=n}\|\partial_y^I\chi_*\|_{W^{s,\infty}}
 \le C_{\chi,s}R_\chi^n(n!)^2\qquad(n\ge0).
\]
Define the extended propagator
\[
 \widetilde E_{\sigma,t}(y)
 :=1+\chi_*(y)\bigl(E_{\sigma,t}(y)-1\bigr).
\]
It agrees with $E_{\sigma,t}$ near $\operatorname{supp}\chi$ and equals
one outside $B_1$. The product bounds of
Lemmas~\ref{ipm:word-products} and \ref{ipm:shiftcalculus} give
\eqref{ipm:centered-propagator} globally for $\widetilde E_{\sigma,t}$,
with $8$ replaced by $R_E:=\max\{8,R_\chi\}$ and an enlarged $C_s$.
These constants depend only on $s$ and the fixed cutoff, and are
independent of the parent, seed, time pair, and expansion depth.

For a source supported in $\operatorname{supp}\chi$, direct integration of
\eqref{ipm:normalized-scalar-equation} gives the solution with zero initial data:
\begin{equation}\label{ipm:oscillatory-inverse}
 A(t,y,\theta)=\gamma(t)\int_0^t
 E_{\sigma,t}(y)\frac{f(\sigma,y,\theta)}{\gamma(\sigma)}\,d\sigma.
\end{equation}
Since $E_{\sigma,t}f(\sigma)=\widetilde E_{\sigma,t}f(\sigma)$ on the source
support, the multiplier estimate applied inside the time integral gives,
for $|f(t)|_n\le\gamma(t)A_*E_d(n)$ and $R\ge2R_E$ with
$R_E=\max\{8,R_\chi\}$,
\begin{equation}\label{ipm:expanded-mean-scalar-bound}
 |A(t)|_n\le
 C_sS\gamma(t)A_*E_d(n)\sum_{l=0}^n(R_E/R)^l
 \le2C_sS\gamma(t)A_*E_d(n).
\end{equation}
To bound $A_t=-c_A^{\ell}A+f$, the scale condition
$\ell K^{2A+3}\le\varepsilon_*$ and the global coefficient bounds
\eqref{ipm:expanded-cascade-amplitude-coefficient} give, on $\mathbb T_\ell^2$,
\[
 \sum_{|I|=n}\|\partial_y^Ic_A^{\ell}\|_\infty
 \le A_c(\ell B_c)^n(n!)^2\le(n!)^2\qquad(n\ge1),
\]
while $\|c_A^{\ell}\|_\infty\le A_c\le K^A$.
Including the fixed base derivatives by factorial splitting yields
\[
 \sum_{|I|=n}\|\partial_y^Ic_A^{\ell}\|_{W^{s,\infty}}
 \le C_sK^A4^n(n!)^2\qquad(n\ge0).
\]
The multiplication argument of \eqref{ipm:expanded-mean-multiplication}
therefore applies with coefficient growth rate $4$ whenever $R\ge8$.
Since $R\ge2R_E\ge16$, the bound for $A_t$ follows with the
same $R,d$ and only a fixed polynomial cost in $K$.
The operator is local in $y$, so support and zero angular mean
are preserved. The calculation
also bounds the leading homogeneous coefficient
$A_1=\gamma(t)\chi E_{s_0,t}f_\delta$ over its entire history.
\end{proof}

\subsection{The lifted Darcy equation and its (time derivative) estimates}

\begin{proposition}
\label{ipm:lifted-inverse}
For the parent coefficients on $\mathbb T_\ell^2\times\mathbb T$, let
$D=k^{-1}\nabla_y+m_0\partial_\theta$. For each smooth scalar $A$, the system
\begin{equation}\label{ipm:lifted-darcy}
 \mathcal V_D A=-ca_2A-Gp,\qquad
 \operatorname{div}_D \mathcal V_D A=0,\qquad
 p\in\mathcal H_D:=
 \overline{\{Dq:q\text{ smooth periodic}\}}^{L^2}
\end{equation}
has a unique pressure vector $p\in\mathcal H_D$ and velocity
$\mathcal V_D A$. Let $\Pi_D$ be the $L^2$-orthogonal projection onto
$\mathcal H_D$, set $T_G=\Pi_DG|_{\mathcal H_D}$, and define
\begin{equation}\label{ipm:darcy-projection}
 \mathbb P_G=I-GT_G^{-1}\Pi_D.
\end{equation}
For $R\ge K^C$ and $d\ge0$, the inverse and $\mathbb P_G$ preserve
bounds of the form $|v|_n\le A_*E_d(n)$, with the same $R,d$ and
with $A_*$ multiplied by at most $K^C$. Explicitly, for all $n\ge0$,
\[
 \begin{aligned}
 |A|_n\le A_*E_d(n)&\quad\Longrightarrow\quad
 |(p,\mathcal V_D A)|_n\le K^CA_*E_d(n),\\
 |v|_n\le A_*E_d(n)&\quad\Longrightarrow\quad
 |\mathbb P_Gv|_n\le K^CA_*E_d(n).
 \end{aligned}
\]
For $0<\rho\le K^{-C}$ and every integer $N\ge0$, they satisfy
\begin{equation}\label{ipm:darcy-gevrey}
 \|\mathcal V_D A\|_{\rho,N}\le K^C\|A\|_{\rho,N},
 \qquad
 \|\mathbb P_Gv\|_{\rho,N}\le K^C\|v\|_{\rho,N}.
\end{equation}
after a fixed enlargement of $C$.
Under the same condition $R\ge K^C$, if
$|(A,A_t)|_n\le A_*E_d(n)$ for all $n\ge0$, then
$|(p_t,(\mathcal V_D A)_t)|_n\le K^CA_*E_d(n)$ with the same $R,d$.
All constants are uniform in $\ell,k$ and in the resonances of $D$.
\end{proposition}
\begin{proof}
\textbf{Step 1. Coercive inversion and derivative bounds.}
At a lifted Fourier frequency $(\xi,j)$ put
$\eta=k^{-1}\xi+jm_0$. The orthogonal projection $\Pi_D$ onto
$\mathcal H_D$ has symbol $\eta\otimes\eta/|\eta|^2$ when
$\eta\ne0$, and zero when $\eta=0$. It has norm at most one on
$L^2$ and on every ordinary $H^s$, commutes with ordinary derivatives
and translations, and is independent of time. All these facts are
uniform in $k$, $\ell$, and the possible resonances.
Distributional $D$-divergence-free vectors are exactly
$\ker\Pi_D$. The pressure equation is therefore
\begin{equation}\label{ipm:coercive-darcy}
 T_Gp=-c\Pi_D(a_2A).
\end{equation}
For $p\in\mathcal H_D$,
$\langle p,T_Gp\rangle=\int p\cdot Gp\ge\lambda\|p\|_2^2$,
where $\lambda^{-1}\le K^C$. The closed-range/self-adjoint argument
used in \eqref{ipm:expanded-mean-coercivity} gives
$\|T_G^{-1}\|_{L^2\to L^2}\le\lambda^{-1}$.
This is an inverse on vector gradients; no scalar potential or
division by $|\eta|$ appears.

The difference quotient proof used above applies unchanged, because
$\Pi_D$ commutes with translations. Thus the inverse maps $H^j$ to
$H^j$ for every fixed $j$, and smooth right sides give smooth pressure
vectors. More quantitatively, for $T_Gp=\Pi_Dh$ one has
\[
 T_G\partial^Ip=\Pi_D\partial^Ih-
 \Pi_D\sum_{\emptyset\ne J\subseteq\{1,\ldots,|I|\}}
       (\partial^{I_J}G)\partial^{I_{J^c}}p.
\]
At the fixed base $s$, the same finite induction gives a valid
constant
\begin{equation}\label{ipm:expanded-mean-elliptic-base}
 J_s=C_s(s+1)\lambda^{-1}
        (1+2^s\lambda^{-1}M_s)^s
\end{equation}
for the map $h\mapsto p$ on $H^s$, where $M_s$ is enlarged to cover
the fixed derivatives of $G$. The assertion holds for arbitrary
finite sums of input norms with no family-size factor, since the
estimate is instantaneous and can simply be summed.
Consequently the additional derivative norms satisfy
\begin{equation}\label{ipm:expanded-mean-pressure-triangular}
 |p|_n\le J_s\left[|h|_n+
 \Lambda_s\sum_{l=1}^n\binom nlR_s^l(l!)^2|p|_{n-l}\right].
\end{equation}
With $R\ge\max(2R_s,8J_s\Lambda_sR_s)$, exactly the same induction
as for \eqref{ipm:mean-triangular} gives
$|p|_n\le2J_s A_*E_d(n)$ whenever $|h|_n\le A_*E_d(n)$.
Apply the multiplication estimate to $h=-ca_2A$ and to $Gp$ in
\eqref{ipm:lifted-darcy}. This proves an arbitrary-shift bound for
$\mathcal V_D$ with the same $R,d$ and only a fixed polynomial prefactor.

The operator $\mathbb P_G$ in \eqref{ipm:darcy-projection} maps onto $\ker\Pi_D$, because $\Pi_D\mathbb P_G=0$, is the
identity on that kernel, and annihilates $G\mathcal H_D$.
The same estimates bound it on every shifted class. It need not be
orthogonal when $G$ is variable.

\textbf{Step 2. Finite-cutoff Gevrey bounds for the lifted inverse.}
The correction argument needs bounds for truncated derivative sums,
not just pointwise factorial estimates. Put
$\|v\|_{\rho,N}=\sum_{n=0}^N\rho^n|v|_n/(n!)^2$.
In the pressure commutator, the relevant weight identity is
\[
 \frac{\rho^n}{(n!)^2}\binom nl(l!)^2
 \frac{((n-l)!)^2}{\rho^{n-l}}
 =\frac{\rho^l}{\binom nl}\le\rho^l.
\]
Multiply \eqref{ipm:expanded-mean-pressure-triangular} by the weights
and sum. Each commutator index satisfies $n-l\le N$, so
\[
 \|p\|_{\rho,N}\le J_s\|h\|_{\rho,N}
 +J_s\Lambda_s\sum_{l=1}^N(\rho R_s)^l\|p\|_{\rho,N}.
\]
Choose $0<\rho\le(8J_s\Lambda_sR_s)^{-1}$, reducing the fixed upper
bound if necessary. The second term is absorbed. Multiplication by
a coefficient is bounded in the same finite sum, by Leibniz and
$\binom nl^{-1}\le1$. This proves \eqref{ipm:darcy-gevrey} uniformly in $N$.
Monotone convergence gives the full Gevrey estimates. The permissible
upper radius is a fixed inverse power of $K$; in the construction a
radius above a fixed inverse power of $\mathcal P$ is chosen beneath
it. This formulation separates the interval of valid radii from the
single radius eventually selected.

\textbf{Step 3. One time derivative without differentiating the hierarchy.}
The space $\mathcal H_D$ and the projection $\Pi_D$ do not vary in
time. Applying \eqref{ipm:seventh-foundations-time-inverse} from
Lemma~\ref{ipm:time-projected-inverse} with coefficient $G$, unknown $p$,
projection $\Pi_D$, and forcing $-ca_2A$ to
\eqref{ipm:coercive-darcy} gives
\begin{equation}\label{ipm:darcy-time-derivative}
 \Pi_D(Gp_t)=-c\Pi_D((a_2)_tA+a_2A_t)-\Pi_D(G_tp).
\end{equation}
The multiplier class was chosen to include $G_t,(a_2)_t$.
If $A,A_t$ have amplitude at most $A_*$ in one shifted class, its
right side has amplitude at most a fixed polynomial times $A_*$,
using the already established estimate for $p$.
Applying the same inverse once more gives $p_t$, and
\[
 (\mathcal V_D A)_t
 =-c(a_2)_tA-ca_2A_t-G_tp-Gp_t
\]
has the same $R,d$ and a fixed polynomial prefactor. The same assertion
holds for finite-cutoff Gevrey norms. It needs no $A_{tt}$ and no time
derivative of a recursive oscillatory pressure coefficient.
\end{proof}

\subsection{Uniformity of growth constants \texorpdfstring{$R$}{R}}

This subsection summarizes the coefficient multiplication and inverse estimates
from Sections~\ref{ipm:sec-mean}--\ref{ipm:sec-inverses} and chooses one
growth constant $R$ valid for all inverses in each step of the hierarchy.

\begin{corollary}
\label{ipm:common-growth}
There is a fixed exponent $r$ such that $R=\mathcal P^r$ is
admissible simultaneously for coefficient multiplication, the selected
mean inverse, the localized scalar inverse, the lifted Darcy inverse,
and its time derivative. Each operation returns the same $R$ and
the same arbitrary input shift, up to a fixed polynomial prefactor.
The choice is independent of the hierarchy depth.
\end{corollary}
\begin{proof}
The preceding input and output requirements can be summarized as follows.
All polynomial factors below are independent of $n,d$ and the number
of times an operator is called.
\begin{center}
\begin{tabular}{>{\raggedright\arraybackslash}p{0.22\linewidth}>{\raggedright\arraybackslash}p{0.16\linewidth}>{\raggedright\arraybackslash}p{0.24\linewidth}>{\raggedright\arraybackslash}p{0.25\linewidth}}
\hline
Operation and reference & Required $R$ & Input bound & Output and cost\\
\hline
Coefficient multiplication, \eqref{ipm:expanded-mean-multiplication}
& $R\ge2R_s$
& $|v|_n\le A_*E_d(n)$; coefficient in $W^{s,\infty}$
& Same $R,d$; factor $2\Lambda_s$\\
Selected mean inverse, Proposition~\ref{ipm:meaninverse}
& $R\ge K^{A_1}$
& $|f|_n\le A_*E_d(n)$; zero spatial integral
& $B,B_t,\mathcal V_{\mathrm{mean}}B,\nabla\bar q$; same $R,d$, factor $K^{A_2}$\\
Supported scalar inverse, Proposition~\ref{ipm:scalar-inverse}
& $R\ge2R_E$ (independent of $K$)
& $|f(t)|_n\le\gamma(t)A_*E_d(n)$; supported in $\operatorname{supp}\chi$
& $A,A_t$; $A(0)=0$; same $\gamma,R,d$, factor $K^C$\\
Lifted Darcy and projection, Proposition~\ref{ipm:lifted-inverse}
& $R\ge K^C$
& $|A|_n\le A_*E_d(n)$ or $|v|_n\le A_*E_d(n)$
& $p,\mathcal V_D A$ or $\mathbb P_Gv$; same $R,d$, factor $K^C$\\
Time derivative of Darcy, Proposition~\ref{ipm:lifted-inverse}
& $R\ge K^C$
& $|(A,A_t)|_n\le A_*E_d(n)$
& $p_t,(\mathcal V_D A)_t$; same $R,d$, factor $K^C$\\
\hline
\end{tabular}
\end{center}
The finite-cutoff bounds \eqref{ipm:darcy-gevrey} hold for
$0<\rho\le K^{-C}$, including the time-derivative bounds of
Proposition~\ref{ipm:lifted-inverse}.
Choose $R=\mathcal P^r$ to satisfy all thresholds in the table and
to dominate the fixed profile and cutoff growth constants.
Each threshold is a fixed polynomial in the prescribed parameters.
Increasing $r$ also permits fixed polynomial prefactors to be absorbed
into a fixed number of factorial shifts. This selection is made once,
before choosing a hierarchy order. Every inverse thereafter accepts
and returns precisely this $R$.
\end{proof}

\begin{remark}
No global $H^s$ norm of a nondecaying coefficient has been used in
these arguments. If such a norm were substituted on the enlarged
two-dimensional torus it could introduce a factor $\ell^{-1}$;
the multiplier formulation avoids that factor altogether. Pressure
potentials also have not been estimated: all coercive inverses act on
vector gradients. These two choices are what make the stated period
uniformity literal rather than an unrecorded scaling convention.
\end{remark}

\section{A finite-order approximation}
\label{ipm:sec-finite}

At depth $N$, write the finite expansions
\begin{equation}\label{ipm:raw-expansions}
 \begin{aligned}
 w^a&=\sum_{p=1}^Nk^{-p} f_p,& f_p&=A_p+B_p,\\
 q^a&=\sum_{p=1}^N(k^{-p}\bar q_p+k^{-p-1}\pi_p),\\
 v_{\mathrm{raw}}^a&=\sum_{p=1}^N
       [k^{-p}(V_p+\mathcal V_{\mathrm{mean}}B_p)+k^{-p-1}C_p].
 \end{aligned}
\end{equation}
Here $q^a$ is a scalar approximate pressure, and $v_{\mathrm{raw}}^a$
is the raw label velocity. For exact fields, $v_{\mathrm{raw}}=k^{-1}v$
with $v=\mathcal V_D A$ and $A=kw$.
The fields $A_p,V_p,\pi_p,\psi_p,C_p$ have
zero angular mean, while $B_p,\bar q_p$ are independent of $\theta$.
Given $A_p$ and the preceding coefficients, reconstruct the velocity
and pressure for $1\le p\le N$ by
\begin{equation}\label{ipm:curl-recurrence}
 \begin{aligned}
 J_{p-1}&=G\nabla_y\pi_{p-1}+C_{p-1},\\
 V_p&=b_0A_p-LJ_{p-1},\\
 \partial_\theta\pi_p
   &=\frac{m_0\cdot(-ca_2A_p-J_{p-1})}{m_0\cdot Gm_0},
       &\langle\pi_p\rangle&=0,\\
 \psi_p&=\partial_\theta^{-1}(m_0^\perp\cdot V_p),
       &C_p&=\nabla_y^\perp\psi_p.
 \end{aligned}
\end{equation}
Write the coefficient of $k^{-p}$ in $v_{\mathrm{raw}}^a$ as
\begin{equation}\label{ipm:velocity-coefficient}
 v_{\mathrm{raw},p}=V_p+\mathcal V_{\mathrm{mean}}B_p+C_{p-1}
     =b_0A_p+\mathcal V_{\mathrm{mean}}B_p+E_p,\qquad
 E_p=-LJ_{p-1}+C_{p-1}.
\end{equation}
Set $v_{\mathrm{raw},N+1}=C_N$. For $2\le p\le N$, define the known source
\begin{equation}\label{ipm:full-source}
 \begin{aligned}
 Z_p={}&E_p\cdot R_0+
       \sum_{\substack{i+j=p\\i,j\ge1}}v_{\mathrm{raw},i}\cdot\nabla_y f_j\\
 &+\sum_{\substack{i+j=p+1\\2\le i\le p-1}}
          (m_0\cdot \mathcal V_{\mathrm{mean}}B_i)\partial_\theta A_j\\
 &+\sum_{\substack{i+j=p+1\\2\le i\le p}}
          (m_0\cdot C_{i-1})\partial_\theta A_j .
 \end{aligned}
\end{equation}
The raw scalar equation to be solved is
\begin{equation}\label{ipm:raw-scalar}
 w_t+v_{\mathrm{raw}}\cdot R_0+
 v_{\mathrm{raw}}\cdot(\nabla_y+km_0\partial_\theta)w=0.
\end{equation}

\begin{proposition}[Triangular coefficient construction and mean compatibility]
\label{ipm:hierarchy-construction}
Under the layer hypotheses, prescribe
$A_1(t,y,\theta)=\alpha\chi(y)a(t,y)f_\delta(\theta)$,
$B_1=\bar q_1=0$, and $\pi_0=C_0=0$.
At every finite depth $N$, smooth coefficients in
\eqref{ipm:raw-expansions} can be selected recursively as follows.
At each order $2\le p\le N$, first select the mean by solving
\begin{equation}\label{ipm:mean-recurrence}
 (B_p)_t+(\mathcal V_{\mathrm{mean}}B_p)\cdot R_0=-\langle Z_p\rangle
\end{equation}
then solve the oscillatory equation with zero initial data,
\begin{equation}\label{ipm:oscillatory-recurrence}
 (A_p)_t+c_A^{\ell}A_p
   =-Z_p+\langle Z_p\rangle
      -(m_0\cdot \mathcal V_{\mathrm{mean}}B_p)\partial_\theta A_1,\qquad A_p(0)=0,
\end{equation}
and finally reconstruct the velocity and pressure by
\eqref{ipm:curl-recurrence}. Every mean source has zero spatial integral.
The finite velocity $v_{\mathrm{raw}}^a$ is exactly $D$-divergence free,
and substitution of the finite expansions into \eqref{ipm:raw-scalar}
cancels every power $k^{-p}$, $1\le p\le N$.
\end{proposition}
\begin{proof}
\textbf{Step 1. Darcy reconstruction and incompressibility.}
In \eqref{ipm:curl-recurrence}, every coefficient multiplying an angular
function is independent
of $\theta$. Hence the right side defining $\pi_p$ has zero mean.
Also $m_0\cdot L=0$, so $m_0\cdot V_p=0$. Since $|m_0|=1$,
$V_p=m_0^\perp\partial_\theta\psi_p$. It follows that
\begin{equation}\label{ipm:exact-curl-pairs}
 k^{-p}V_p+k^{-p-1}C_p
      =k^{-p}D^\perp\psi_p.
\end{equation}
These pairs are exactly $D$-divergence free. The mean velocity is
also $D$-divergence free, since
$D\cdot \mathcal V_{\mathrm{mean}}B_p=k^{-1}\operatorname{div}_y \mathcal V_{\mathrm{mean}}B_p=0$.
Thus the finite $v_{\mathrm{raw}}^a$ is exactly incompressible, including its last
term $k^{-N-1}C_N$.

By \eqref{ipm:velocity-coefficient}, the field $E_p$ is known from the preceding oscillatory bundle,
has zero angular mean, and satisfies
\[
 m_0\cdot E_p=m_0\cdot C_{p-1}.
\]
In particular $E_1=0$. Substitution in
\eqref{ipm:curl-recurrence} also gives the coefficientwise Darcy identity
\begin{equation}\label{ipm:expanded-hierarchy-darcy-coefficient}
 V_p+C_{p-1}
   =-ca_2A_p-Gm_0\partial_\theta\pi_p
                      -G\nabla_y\pi_{p-1}.
\end{equation}
Indeed the right side is
$L(-ca_2A_p-J_{p-1})+C_{p-1}=V_p+C_{p-1}$.

\textbf{Step 2. Triangular recursion and scalar cancellation.}
In \eqref{ipm:raw-scalar}, the slow quadratic term has order $k^{-i-j}$ and the fast term
has order $k^{1-i-j}$. Thus at order $k^{-p}$ their sum is
\begin{equation}\label{ipm:expanded-hierarchy-order-count}
 \sum_{\substack{i+j=p\\i,j\ge1}}v_{\mathrm{raw},i}\cdot\nabla_y f_j
 +\sum_{\substack{i+j=p+1\\i,j\ge1}}
                 (m_0\cdot v_{\mathrm{raw},i})\partial_\theta A_j.
\end{equation}
The angular derivative annihilates $B_j$.
For $p=1$ the only possible quadratic term is
$(m_0\cdot v_{\mathrm{raw},1})\partial_\theta A_1=0$, so
$(A_1)_t+c_A^{\ell}A_1=0$, as prescribed.

For $p\ge2$ all indices in the slow sum are at most $p-1$.
In the fast sum the $i=1$ term is zero. For $i\ge2$,
\[
 m_0\cdot v_{\mathrm{raw},i}=m_0\cdot \mathcal V_{\mathrm{mean}}B_i+m_0\cdot C_{i-1}.
\]
The only current unknown in this expression is the term with
$i=p,j=1$ containing $B_p$. Thus every term of $Z_p$ in \eqref{ipm:full-source} is known before
solving order $p$. The unknown fast term has zero angular average.
First solve \eqref{ipm:mean-recurrence} using the selected mean inverse,
then \eqref{ipm:oscillatory-recurrence} using the localized scalar inverse,
and finally reconstruct the curl and pressure bundle by
\eqref{ipm:curl-recurrence}.
Angular averaging of \eqref{ipm:oscillatory-recurrence} preserves
$\langle A_p\rangle=0$. These equations cancel the coefficient of
$k^{-p}$ in \eqref{ipm:raw-scalar}.

\textbf{Step 3. Compatibility of every mean source.}

The curl pairs imply
\begin{equation}\label{ipm:coefficient-divergence}
 \operatorname{div}_y v_{\mathrm{raw},i}+
                  m_0\cdot\partial_\theta v_{\mathrm{raw},i+1}=0.
\end{equation}
In the first term only $V_i$ contributes: the mean velocity and
$C_{i-1}$ have zero slow divergence. In the second only $C_i$
contributes: $m_0\cdot V_{i+1}=0$ and the mean is angle independent.
The identity consequently holds without knowing the unknown
$A_{i+1},B_{i+1}$. Set $v_{\mathrm{raw},0}=0$; the $i=0$ identity is
$m_0\cdot\partial_\theta v_{\mathrm{raw},1}=0$.

In fact the compatibility holds locally in the slow variable:
\begin{equation}\label{ipm:expanded-hierarchy-local-flux}
 \langle Z_p\rangle
   =\operatorname{div}_y
       \sum_{\substack{i+j=p\\i,j\ge1}}\langle v_{\mathrm{raw},i} f_j\rangle.
\end{equation}
To check this identity, angular integration by parts changes the
fast sum into
$-\sum_{i+j=p+1}\langle(m_0\cdot\partial_\theta v_{\mathrm{raw},i})f_j\rangle$.
The divergence identity converts it to
$\sum_{i+j=p}\langle(\operatorname{div}_yv_{\mathrm{raw},i})f_j\rangle$.
Adding the slow sum gives the displayed divergence. The omitted
current mean drift and $E_p\cdot R_0$ have zero angular average.
All indices in the resulting flux are at most $p-1$, so the source
depends only on already constructed coefficients.

Integrate \eqref{ipm:expanded-hierarchy-order-count} over $y,\theta$.
Integration by parts changes its two sums into
\[
 -\sum_{i+j=p}\int(\operatorname{div}_y v_{\mathrm{raw},i})f_j,\qquad
 -\sum_{i+j=p+1}\int(m_0\cdot\partial_\theta v_{\mathrm{raw},i})f_j.
\]
Use \eqref{ipm:coefficient-divergence} in the second sum and replace
$i-1$ by $i$. Its $i=1$ contribution is zero, and the remaining sum
is the negative of the first displayed sum. Hence their sum vanishes.
The omitted current term
$(m_0\cdot \mathcal V_{\mathrm{mean}}B_p)\partial_\theta A_1$ has zero angular integral.
Also $\langle E_p\rangle=0$. It follows that
\begin{equation}\label{ipm:expanded-hierarchy-mean-compatible}
 \int_{\mathbb T_\ell^2}\langle Z_p\rangle\,dy=0.
\end{equation}
This is exactly the source condition needed by the selected mean
inverse. It is a cancellation of the full slow and fast sums, not an
assertion that each of their products separately has zero integral.

For comparison, the first nonzero mean source is explicitly
\begin{equation}\label{ipm:second-mean-source}
 \begin{aligned}
 \langle Z_2\rangle
 &=\langle A_1b_0\cdot\nabla_yA_1
            +(m_0\cdot C_1)\partial_\theta A_1\rangle\\
 &=\operatorname{div}_y\bigl(b_0\langle A_1^2\rangle\bigr).
 \end{aligned}
\end{equation}
To check the coefficient, integrate the second term by parts in
$\theta$ and use
$m_0\cdot\partial_\theta C_1=-\operatorname{div}_y(b_0A_1)$.
Its sum with the first term is the divergence of $b_0A_1^2$.
Thus no factor of one half is missing.
\end{proof}

\subsection{Coefficient bounds}

We keep the fixed Sobolev base $s=3$ and the growth constant
$R=\mathcal P^r$ chosen above. Write
\[
 E_d(n)=R^{n+d}((n+d)!)^2,\qquad
 f\in[g;d]\quad\Longleftrightarrow\quad
 |f(t)|_n\le g(t)E_d(n)\quad\hbox{for every }t,n.
\]
For a tuple, the norm means the sum of its component norms. In
particular, no unspecified tuple-dependent constant is suppressed.
We use the mean bound \eqref{ipm:mean-bound} and the pointwise,
$\gamma$-weighted scalar bound \eqref{ipm:expanded-mean-scalar-bound}.

Here is a precise convention for the finite operation costs used below.
The same-$R$ inverse estimates, coefficient multiplication estimates,
fixed-dimensional contractions, angular averaging, subtraction of
the average, and the zero-mean angular primitive have prefactors
bounded by fixed powers of $\mathcal P$. A derivative increases a
shift by one. Since
\[
 E_{d+b}(n)/E_d(n)\ge R^b\qquad(b\ge0),
\]
any such fixed prefactor is absorbed by a fixed number of shifts,
uniformly in $d,n$. Choose an integer $h\ge10$, independent of all
orders, large enough that the following conventions hold:
\begin{enumerate}
\item either inverse, including its scalar time derivative and the
mean velocity and pressure-gradient outputs when applicable, costs
at most $h$ shifts;
\item a product with one spatial or angular derivative, including
fixed coefficient contractions and a fixed number of additions,
costs at most $2h$ shifts beyond the two input shifts;
\item the entire finite curl/pressure reconstruction \eqref{ipm:curl-recurrence}
costs at most $2h$ shifts beyond each of its input bundles;
\item forming $E_p\cdot R_0$ from the preceding oscillatory bundle
costs at most $2h$ shifts.
\end{enumerate}
These are a finite list of operations; increasing $h$ covers them
simultaneously. Increase it also so that the leading scalar bundle
$(A_1,(A_1)_t)$ fits at shift $h$ with profile $\gamma$, and to cover
the fixed numerical constants in the counting estimates below. No cost
in this list depends on an expansion index or a factorial shift.
We use
\begin{equation}\label{ipm:hierarchy-shifts}
 a_p=(20p-16)h\quad(p\ge1),\qquad
 b_p=(20p-28)h\quad(p\ge2).
\end{equation}
Notice $a_1=4h$ and $b_2=12h$.

\begin{lemma}[Uniform coefficient bounds]\label{ipm:coefficients}
With $R,h$ chosen above and $a_p,b_p$ defined in \eqref{ipm:hierarchy-shifts}, every order satisfies
\begin{equation}\label{ipm:coefficient-bundles}
 \begin{aligned}
 (A_p,(A_p)_t,V_p,\pi_p,\psi_p,C_p,\nabla_y\pi_p)
       &\in[\gamma\Gamma^{2p-2};a_p],\\
 (B_p,(B_p)_t,\mathcal V_{\mathrm{mean}}B_p,\nabla_y\bar q_p)
       &\in[\Gamma^{2p-2};b_p]\qquad(p\ge2).
 \end{aligned}
\end{equation}
The same $R$ is used at every order. Each finite oscillatory coefficient
is supported in $\operatorname{supp}\chi$ in the label variable.
\end{lemma}
\begin{proof}
The proof keeps separate the derivative shifts and the time profiles.
For $p=1$, write
$A_1=\gamma\chi E_{s_0,t}f_\delta$. The centered propagator has fixed
Gevrey bounds on this support, the cutoff is fixed, and the profile
has growth polynomial in $\delta^{-1}$. The fixed $H^3$ base adds
only a fixed prefactor and a fixed derivative shift. By the preceding
choice of $R,h$, both $A_1$ and
$(A_1)_t=-c_A^{\ell}A_1$ fit at shift $h=a_1-3h$ with profile $\gamma$.
The reconstruction costs at most $2h$, leaving $h$ shifts to absorb
the finite tuple sum within $a_1$. This proves the first bundle.

Assume all bundles through $p-1$ have been constructed. Split every
slow velocity into its oscillation $V_i+C_{i-1}$ and its mean
$\mathcal V_{\mathrm{mean}}B_i$, and split $f_j=A_j+B_j$. A fixed addition is included in the
operation cost above. With the larger shift $a_i$ permitted for a mean
factor, the following are bounds for the individual known sources:
\begin{equation}\label{ipm:expanded-hierarchy-cost-table}
 \begin{array}{c|c|c}
 \text{source}&\text{shift bound}&\text{time profile}\\ \hline
 E_p\cdot R_0&a_{p-1}+2h
                     =20hp-34h&\gamma\Gamma^{2p-4}\\
 \text{slow oscillation--oscillation},\ i+j=p
     &a_i+a_j+2h=20hp-30h&\gamma^2\Gamma^{2p-4}\\
 \text{slow mixed},\ i+j=p
     &a_i+a_j+2h=20hp-30h&\gamma\Gamma^{2p-4}\\
 \text{slow mean--mean},\ i+j=p
     &a_i+a_j+2h=20hp-30h&\Gamma^{2p-4}\\
 \text{fast curl},\ (i-1)+j=p
     &a_{i-1}+a_j+2h=20hp-30h&\gamma^2\Gamma^{2p-4}\\
 \text{fast mean},\ i+j=p+1
     &b_i+a_j+2h=20hp-22h&\gamma\Gamma^{2p-2}.
 \end{array}
\end{equation}
Absent factors with index one mean zero. The linear reconstruction
uses the preceding oscillatory bundle. In a fast curl product both
indices $i-1,j$ are at most $p-1$. In a known fast mean product
$i\le p-1$ and $j\le p-1$. Thus every line is covered by the induction.

Only slow oscillation--oscillation, slow mean--mean, and fast curl
terms can contribute to $\langle Z_p\rangle$. Indeed a slow mixed
term has one zero-mean angular factor and an angle-independent
factor; slow differentiation preserves that zero mean. A fast mean
term is a coefficient times an angular derivative. The linear
reconstruction has zero mean as well.
Since $\gamma\le\Gamma$ and $\Gamma\ge1$, the profiles of all
surviving mean terms are bounded by $\Gamma^{2p-2}$.
Their shifts are at most
\[
 d_M=20hp-30h=b_p-2h.
\]
There are at most a fixed constant times $p$ terms after the slow
products have been split into four types. For $p\ge2$,
$d_M\ge5hp$. Therefore
$E_{d_M+1}(n)/E_{d_M}(n)=R(n+d_M+1)^2$
absorbs this number of terms, including the fixed angular-average
constant, after the fixed choice of $h$. The mean source belongs to
$[\Gamma^{2p-2};d_M+1]$. Applying the mean inverse, with its zero
spatial integral established above, gives its entire bundle at shift
\begin{equation}\label{ipm:expanded-hierarchy-mean-reserve}
 d_M+1+h=b_p-h+1<b_p.
\end{equation}
This application is unweighted in time; no factor $\gamma(t)$ is
claimed for a terminally selected mean.

For the oscillatory equation, subtraction of the angle average
removes every pure mean product. Each remaining term retains at least
one oscillatory factor. The table and
$\gamma^2\le\gamma\Gamma$ show that its time profile is at most
$\gamma\Gamma^{2p-2}$. The current mean drift, now known, has shift
\[
 b_p+a_1+2h=20hp-22h=a_p-6h
\]
and exactly that profile. This is also an upper bound for every
known oscillatory-source shift in the table. Summing its at most
$Cp$ terms costs one shift. The scalar inverse costs at most $h$
more, so
\begin{equation}\label{ipm:expanded-hierarchy-osc-reserve}
 (A_p,(A_p)_t)\in
 [\gamma\Gamma^{2p-2};a_p-6h+1+h]
 \subset[\gamma\Gamma^{2p-2};a_p-4h].
\end{equation}
The scalar inverse is used only after the source has been bounded
pointwise by this profile. Division by $\gamma$ in its formula
cancels that existing factor; it does not introduce an inverse seed.

To finish the reconstruction, the current scalar bundle costs at
most $2h$ more, landing at $a_p-2h$. Inputs from order $p-1$ land at
most at $a_{p-1}+2h=a_p-18h$, with a smaller power of $\Gamma$.
The remaining $2h$ shifts absorb the fixed tuple sum and these two
input groups within $a_p$.
This proves \eqref{ipm:coefficient-bundles} with unit majorant
constant, rather than a hidden order-dependent constant.
Time derivatives in the two scalar bundles are supplied by the
inverse estimates or directly by their displayed equations; no
source or recursive pressure is differentiated in time.

Finally every oscillatory term is supported in the fixed label
support. Slow derivatives, angular primitives, and coefficient
multiplication preserve support. A global mean velocity enters an
oscillatory source only multiplied by an already supported
oscillation. The scalar equation is local in $y$ with zero initial
data. These facts prove the support assertion by the same induction.
\end{proof}

\begin{lemma}[Parity of the hierarchy in the phase variable]
\label{ipm:phase-parity}
At each fixed $(t,y)$ the coefficients $A_p,V_p,v_{\mathrm{raw},p},E_p$ have parity
$(-1)^p$ under $\theta\mapsto-\theta$, while
$\pi_p,\psi_p,C_p$ have parity $(-1)^{p+1}$.
Every odd-order mean $B_p$ vanishes.
\end{lemma}
\begin{proof}
With $t,y$ fixed, $A_1$ is
odd in $\theta$. Inductively $A_p,V_p,v_{\mathrm{raw},p},E_p$ have parity $(-1)^p$,
and $\pi_p,\psi_p,C_p$ have parity $(-1)^{p+1}$. Slow sources with
$i+j=p$ and fast sources with $i+j=p+1$ both have parity $(-1)^p$,
the latter because $\partial_\theta$ reverses parity. At odd $p$ the
mean source is therefore zero, and the selected inverse returns
$B_p=0$. This observation is not needed for the estimates, but
independently checks their order indexing.
\end{proof}

\subsection{Finite approximation and exponentially small residual}

\begin{proposition}
\label{ipm:finite-approximation}
Choose the common growth constant and the coefficients above. There
are fixed $\vartheta>0$, $Q,C,k_*>0$ such that, if
$\mathcal P^Q\le k^{\vartheta/100}$ and $k\ge k_*$, the depth
$N=\lfloor k^\vartheta\rfloor$ gives a normalized approximation
$A^a=k w^a$ with the following properties at a radius
$\rho\ge\mathcal P^{-C}$.
The approximation and its lifted transport obey
\begin{equation}\label{ipm:approximation-bounds}
 \begin{aligned}
 \|A^a\|_\rho+\|DA^a\|_\rho+\|(A^a)_t\|_\rho
       &\le\mathcal P^C,\\
 \|\mathcal V_D A^a-V_1\|_\rho&\le k^{-1}\mathcal P^C,\\
 \|b^a\|_\rho&\le k^{-1}\mathcal P^C\le k^{-1/2},
 \qquad b^a=(k^{-1} \mathcal V_D A^a,m_0\cdot \mathcal V_D A^a).
 \end{aligned}
\end{equation}
Its scalar residual satisfies
\[
 \|(A^a)_t+\mathcal V_D A^a\cdot R_0+
                   \mathcal V_D A^a\cdot DA^a\|_\rho
 \le \exp(-0.6k^\vartheta\log k).
\]
The finite oscillatory tail and phase means satisfy
\begin{equation}\label{ipm:expanded-hierarchy-split-tail}
 \begin{aligned}
 \left\|\sum_{p=2}^Nk^{1-p}A_p(t)\right\|_\rho
      &\le k^{-1}\gamma(t)\mathcal P^C,\\
 \left\|\sum_{p=2}^Nk^{1-p}B_p(t)\right\|_\rho
      &\le k^{-1}\mathcal P^C.
 \end{aligned}
\end{equation}
All constants are independent of $N$ and the stage.
\end{proposition}
\begin{proof}
\textbf{Step 1. Darcy residual and velocity correction.}
We now keep a finite depth $N\ge2$. Its terminal curl term is essential:
as noted above, $v_{\mathrm{raw},N+1}=C_N$, even though there is no density
coefficient $f_{N+1}$. Summing
\eqref{ipm:expanded-hierarchy-darcy-coefficient}, and using
$\mathcal V_{\mathrm{mean}}B_p=-ca_2B_p-G\nabla_y\bar q_p$, gives the exact identity
\begin{equation}\label{ipm:darcy-defect}
 v_{\mathrm{raw}}^a+ca_2w^a+
 G(\nabla_y+km_0\partial_\theta)q^a
       =k^{-N-1}H_N,\qquad H_N=C_N+G\nabla_y\pi_N.
\end{equation}
All lower powers cancel. The scalar pressure term is in the range
of $G D$, and $v_{\mathrm{raw}}^a$ is exactly $D$-divergence free. The projection
$\mathbb P_G$ annihilates $G D$ and is the identity on
$D$-divergence-free vectors. Hence
\begin{equation}\label{ipm:exact-darcy-correction}
 \mathcal V_D(w^a)-v_{\mathrm{raw}}^a=-k^{-N-1}\mathbb P_GH_N.
\end{equation}
The same-radius vector estimate for $\mathbb P_G$ is used here.
No scalar pressure potential, inverse carrier, or small-divisor
estimate is used to obtain this velocity error.

\textbf{Step 2. A quantitative majorant at one fixed radius.}

If $f\in[g;d]$, then at $\rho\le(8R)^{-1}$,
\begin{equation}\label{ipm:shifted-to-radius}
 \|f(t)\|_\rho\le2g(t)(4R)^d(d!)^2.
\end{equation}
Indeed $(n+d)!\le2^{n+d}n!d!$ reduces the defining sum to
$g(4R)^d(d!)^2\sum_n(4\rho R)^n$.
Choose $\rho$ once smaller if needed to lie in the interval for
the exact Darcy/projection estimates. All choices are by fixed
powers of $\mathcal P$. The same argument may be applied after a
fixed number of derivatives by adding that fixed number to $d$.

By \eqref{ipm:coefficient-bundles}, since $a_p,b_p\le20hp$,
$R=\mathcal P^r$, and $\Gamma\le\mathcal P$, there are fixed constants
$A,D,C\ge1$ such
that every coefficient bundle, and its first two additional
spatial/angular derivatives when needed, obeys
\begin{equation}\label{ipm:order-coefficient-bound}
 M_p:=\bigl[\mathcal P^A(C(p+1))^D\bigr]^{p+1}.
\end{equation}
Here an oscillatory coefficient or its time derivative has the
stronger bound $\gamma(t)M_p$, whereas a mean coefficient has
bound $M_p$. All powers $\Gamma^{2p-2}$ have been absorbed into
$\mathcal P^{A(p+1)}$. To see that $A,D$ are fixed, the factors
$(4R)^{20hp+O(1)}$ are $\mathcal P^{O(p+1)}$, and
$((20hp+O(1))!)^2$ is bounded by
$(C(p+1))^{D(p+1)}$. Fixed coefficient operations and the projection
prefactor only enlarge these fixed constants.

Keep the constant $C$ in $M_p$ fixed and define the larger constant
$C_B=e^2C$. Set
\begin{equation}\label{ipm:expanded-hierarchy-BN}
 \mathcal B_N=\mathcal P^A(C_B(N+2))^D.
\end{equation}
It may be chosen so that for $1\le p\le N$,
\[
 M_p\le\mathcal B_N^{p+1},\qquad
 M_{p+1}/M_p\le\mathcal B_N\quad(p<N).
\]
For the second inequality use
\[
 (p+2)\log(p+2)-(p+1)\log(p+1)\le\log(p+2)+1;
\]
the extra fixed factor $e^D$ is absorbed by $(C_B/C)^D$.

Choose the fixed $\vartheta>0$ sufficiently small that
$D\vartheta<1/200$, and put $N=\lfloor k^\vartheta\rfloor$.
Choose $Q$ after $A,D$ and all subsequent fixed polynomial exponents.
Under $\mathcal P^Q\le k^{\vartheta/100}$, for all sufficiently
large $k$ one has
\begin{equation}\label{ipm:expanded-hierarchy-ratio-small}
 \mathcal B_N\le k^{1/50},\qquad
 k^{-1}\mathcal B_N\le k^{-49/50}\le\tfrac12.
\end{equation}
Indeed
$\log\mathcal B_N\le A\vartheta\log k/(100Q)
 +D\vartheta\log k+O(1)$.

\textbf{Step 3. Low orders and the amplitude-weighted tail.}

Set $A^a=kw^a$. The normalized $p$th term is bounded by
$k^{1-p}M_p$, with an extra factor $\gamma(t)$ in its
oscillatory part. For $p\ge2$ the ratio of these majorants is at most
$k^{-1}\mathcal B_N$. Therefore, estimating $p=1$ separately and
$p=2$ with its own fixed value of $M_2$, the oscillatory and mean tails are bounded by
$2k^{-1}\gamma(t)M_2$ and $2k^{-1}M_2$, respectively. Since
$M_2$ is a fixed power of $\mathcal P$, this proves
\eqref{ipm:expanded-hierarchy-split-tail}.
The same estimates hold for the time derivatives in the supplied
bundles and for the fixed extra derivatives already included in
$M_p$. In particular
$\|A^a\|_\rho+\|DA^a\|_\rho+\|(A^a)_t\|_\rho\le\mathcal P^{C_2}$.
No bound by $\mathcal B_N^{O(1)}$ is inserted into $M_1$ or $M_2$;
doing so would unnecessarily make the low-order estimates depend
on $N$.

From the normalized finite velocity,
\[
 kv_{\mathrm{raw}}^a
 =V_1+k^{-1} C_1+
   \sum_{p=2}^Nk^{1-p}(V_p+\mathcal V_{\mathrm{mean}}B_p)
   +\sum_{p=2}^Nk^{-p} C_p,
\]
the same geometric estimate gives
$\|kv_{\mathrm{raw}}^a-V_1\|_\rho\le k^{-1}\mathcal P^{C_3}$.
The normalized exact Darcy correction in
\eqref{ipm:exact-darcy-correction} is
$-k^{-N}\mathbb P_GH_N$ and satisfies this bound as well.
For example, its norm is at most
$k^{-N}\mathcal P^C\gamma M_N$, and
$k^{1-N}M_N\le k^{-1} M_2$ for $N\ge2$.
The additional fixed powers of $\mathcal P$ and $\gamma\le\Gamma$
are absorbed in $C_3$. We have proved \eqref{ipm:approximation-bounds}.
In the last line $m_0\cdot V_1=0$ is used exactly, not estimated.
The last inequality follows by increasing the final fixed $Q$.
These estimates retain $\gamma$ on the finite oscillatory density
tail, but do not claim it for the phase means.

\textbf{Step 4. The entire residual, including the top curl.}

Let $\mathcal S_N$ denote the left side of
\eqref{ipm:raw-scalar} evaluated at $(w^a,v_{\mathrm{raw}}^a)$.
The recurrence cancels its coefficients at powers $1,\ldots,N$.
The remaining linear term is
$k^{-N-1}C_N\cdot R_0$. Its slow products have
$1\le i\le N+1$, $1\le j\le N$ and power $q=i+j$.
Its fast products have the same ranges and power $q=i+j-1$;
the $i=1$ fast terms are zero. In particular the fast product
with $i=N+1,j=1$ starts at power $N+1$, and has not been lost.
It follows that
\begin{equation}\label{ipm:expanded-hierarchy-raw-residual}
 \mathcal S_N=\sum_{q=N+1}^{2N+1}k^{-q} S_q.
\end{equation}

We now bound each $S_q$ without assuming any further cancellation.
For $i\le N$ the norm of $v_{\mathrm{raw},i}$ is at most a fixed polynomial in
$\mathcal P$ times $\mathcal B_N^{i+1}$; the same bound holds for
$v_{\mathrm{raw},N+1}=C_N$ with $i=N+1$. Each differentiated $f_j$ has bound
at most $\mathcal P^C\mathcal B_N^{j+1}$.
A slow product of power $q$ is therefore at most
$\mathcal P^C\mathcal B_N^{q+2}$, and a fast product at most
$\mathcal P^C\mathcal B_N^{q+3}$. The factors $\gamma$ can be
bounded by $\Gamma\le\mathcal P$ here: a residual need not retain
an oscillatory profile because it includes pure mean products.
There are at most $2N+1$ terms at any power. Since
$\mathcal B_N\ge N$ after increasing the fixed constants, their
number costs at most two additional powers. The linear term has
a smaller bound. Hence
\[
 \|S_q\|_\rho\le\mathcal P^{C_4}\mathcal B_N^{q+5},
 \qquad N+1\le q\le2N+1.
\]
Products with $R_0$ use its coefficient norm. This avoids inserting
the enlarged torus volume into the estimate.
By \eqref{ipm:expanded-hierarchy-ratio-small},
\[
 \|\mathcal S_N\|_\rho
   \le2\mathcal P^{C_4}k^{-N-1}\mathcal B_N^{N+6}.
\]

Replacing $v_{\mathrm{raw}}^a$ by the exact Darcy velocity adds exactly
\[
 -k^{-N-1}\mathbb P_GH_N\cdot
       [R_0+(\nabla_y+km_0\partial_\theta)w^a]
 =-k^{-N-1}\mathbb P_GH_N\cdot[R_0+DA^a].
\]
Multiplication by $R_0$ uses its coefficient Gevrey norm, whereas
$DA^a$ uses its function Gevrey norm. The corresponding product
bound is a fixed polynomial by the low-order estimates already
proved. The correction consequently has norm
at most $k^{-N-1}\mathcal P^{C_5}\mathcal B_N^{N+1}$.
Multiply the raw density equation by $k$ and use linearity of
$\mathcal V_D$. This gives
\begin{equation}\label{ipm:approximate-residual}
 r^a=(A^a)_t+\mathcal V_D A^a\cdot R_0+
                          \mathcal V_D A^a\cdot DA^a
\end{equation}
and the bound
\begin{equation}\label{ipm:residual-before-separation}
 \|r^a\|_\rho
       \le\mathcal P^{C_6}k^{-N}\mathcal B_N^{N+6}.
\end{equation}
All constants here are fixed independently of $N,k$ and the stage.
The source time derivative never enters this computation:
$\partial_t w^a$ uses only the time derivatives already in the
coefficient bundles.

Finally take logarithms. By
\eqref{ipm:expanded-hierarchy-ratio-small}, the right side of
\eqref{ipm:residual-before-separation} has logarithm at most
\[
 C_6\log\mathcal P-N\log k+\frac{N+6}{50}\log k.
\]
The first term is at most
$C_6\vartheta\log k/(100Q)$.
For sufficiently large $k$, this upper bound is less than
$-0.9N\log k$. Also
$N=\lfloor k^\vartheta\rfloor\ge0.9k^\vartheta$ once $k$ is large.
Thus, with ample numerical margin,
\begin{equation}\label{ipm:tiny-residual}
 \|r^a\|_\rho\le\exp(-0.6k^\vartheta\log k).
\end{equation}
The radius remains at least a fixed inverse power of $\mathcal P$.
The entire finite hierarchy has used one unchanged growth constant
$R$; the fixed passage to this radius is performed only after its
construction. These are precisely the residual and approximate-field
bounds used in the next subsection.
\end{proof}

\section{From the finite approximation to an exact solution}
\label{ipm:sec-exact}

\subsection{Exact scalar correction}

We give the correction argument at finite derivative cutoffs before
constructing a solution. This order avoids assuming a Gevrey solution
in order to justify its own energy estimate. All norms in this
subsection are on the fixed lifted torus for the layer under
consideration. Their constants have the uniformity in $\ell,k$
already established for the product and Darcy estimates.

\begin{proposition}\label{ipm:correction}
Suppose the supplied coefficients and approximation are real and
$C^\infty$ in time and space, the coefficients satisfy
\eqref{ipm:derived-coefficients},
and the approximation satisfies \eqref{ipm:tiny-residual} and
\eqref{ipm:approximation-bounds} at a radius bounded below by a fixed
inverse power of $\mathcal P$. Assume the frequency separation
\begin{equation}\label{ipm:correction-frequency-separation}
 \mathcal P^Q\le k^{\vartheta/100},\qquad k\ge k_*,
\end{equation}
for sufficiently large fixed $Q$ and $k_*$. Then there is a real
smooth $e$ on $[0,S]$ with $e(0)=0$ such that
$A=A^a+e$ solves
\begin{equation}\label{ipm:exact-lifted-scalar}
 A_t+\mathcal V_D A\cdot R_0+\mathcal V_D A\cdot DA=0.
\end{equation}
For one radius $\rho_*\ge\mathcal P^{-C}$,
\begin{equation}\label{ipm:correction-bound}
 \sup_{t\in[0,S]}
 \bigl(\|e(t)\|_{\rho_*}+\|e_t(t)\|_{\rho_*}
                   +\|\mathcal V_D e(t)\|_{\rho_*}\bigr)
 \le\mathcal P^C\exp(-k^\vartheta/2).
\end{equation}
If the parent coefficients and approximation have the joint reflection
parities of the construction, the correction is odd under that reflection.
\end{proposition}

\begin{proof}
\textbf{Step 1. The exact error equation.}
Linearity of $\mathcal V_D$ gives
\begin{align*}
 \mathcal V_D(A^a+e)\cdot D(A^a+e)
       -\mathcal V_D A^a\cdot DA^a
 &=\mathcal V_D(A^a+e)\cdot De+\mathcal V_D e\cdot DA^a.
\end{align*}
Subtracting the residual of $A^a$ therefore yields exactly
\begin{equation}\label{ipm:correction-equation}
 e_t+b\cdot\nabla_{y,\theta}e+Z(e)=-r^a,\qquad e(0)=0,
\end{equation}
where
\begin{equation}\label{ipm:expanded-correction-bz}
 \begin{split}
 b&=b^a+T_{k^{-1}} e,\qquad
 T_{k^{-1}} f=(k^{-1} \mathcal V_D f,m_0\cdot \mathcal V_D f),\\
 Z(e)&=\mathcal V_D e\cdot(R_0+DA^a).
 \end{split}
\end{equation}
Since $\operatorname{div}_D \mathcal V_D f=0$,
$\operatorname{div}_{y,\theta}b=0$. In $Z(e)$, the derivative
$D$ falls on the given approximation, not on $e$. Thus $Z$ is a
linear operator of order zero on the correction.

Choose $\rho_0$ below the approximation and coefficient radii and
below the finite-cutoff Darcy threshold. There is a number
$\Lambda=\mathcal P^{C_0}\ge1$, with a fixed exponent $C_0$, such that
for every $0<\rho\le\rho_0$ and every cutoff $M$,
\begin{equation}\label{ipm:expanded-correction-operator}
 \begin{split}
 \|\mathcal V_D f\|_{\rho,M}+\|T_{k^{-1}} f\|_{\rho,M}
       &\le\Lambda\|f\|_{\rho,M},\\
 \|Z(f)\|_{\rho,M}&\le\Lambda\|f\|_{\rho,M},\\
 \|b^a\|_{\rho,M}&\le B_0:=k^{-1/2},\qquad
 \|r^a\|_{\rho,M}\le R_k:=e^{-0.6k^\vartheta\log k}.
 \end{split}
\end{equation}
Here $\|f\|_{\rho,M}=\sum_{n=0}^M\rho^n|f|_n/(n!)^2$.
For the term $R_0\cdot \mathcal V_D f$, use the coefficient norm
\[
 \sum_{l\ge0}\frac{\rho^l}{(l!)^2}
       \sum_{|I|=l}\|\partial^IR_0\|_{W^{s,\infty}}.
\]
It is bounded by a fixed power of $\mathcal P$ and multiplies the
$H^s$ norm of $\mathcal V_D f$. This avoids a volume factor. If one instead
places $R_0$ in unnormalized $H^s$, the cost is $O(\ell^{-1})$, also
allowed by $\mathcal P$. The term containing $DA^a$ uses the
$H^s$ algebra inequality. Both estimates hold at the same finite
cutoff by the factorial product inequality.

\par\smallskip\noindent\textbf{Step 2. The fixed Sobolev-base transport estimate.}
For a divergence-free vector field $b$ and a smooth scalar $v$ on
the three-dimensional lifted torus, $s=3$ gives
\begin{equation}\label{ipm:expanded-correction-base}
 \left|\sum_{|\beta|\le s}
 \langle\partial^\beta v,
            [\partial^\beta,b\cdot\nabla]v\rangle_{L^2}\right|
       \le C_s\|b\|_{H^s}\|v\|_{H^s}^2.
\end{equation}
Indeed, the term $b\cdot\nabla\partial^\beta v$ has zero pairing
with $\partial^\beta v$. In a remaining term let
$r=|\beta|\le s$ and $q\ge1$ be the number of derivatives on $b$.
For $q=1$, put $\partial b$ in $L^\infty$ and the derivative of
$v$ of order $r$ in $L^2$. For $q=r$, put $\partial^r b$ in $L^2$
and $\nabla v$ in $L^\infty$. These are controlled by $H^s$ since
$s>5/2$. For $1<q<r$, the available regularities of the two
factors are $s-q$ and $s-r+q-1$. Their sum is
$2s-r-1\ge s-1>3/2$, and both are nonnegative, so their product
is in $L^2$ with norm at most $C_s\|b\|_{H^s}\|v\|_{H^s}$.
There are finitely many base multi-indices, depending only on $s$.
This proves \eqref{ipm:expanded-correction-base}.

Apply this estimate with $v=\partial^I e$, for any additional
ordered word $I$. Its constant is independent of the length of $I$,
because only the fixed $s$ derivatives are expanded in this step.
In particular, it does not ask for $b$ in $H^{s+1}$ or for an
additional derivative of $v$ beyond $H^s$.

\par\smallskip\noindent\textbf{Step 3. All additional derivatives stay inside the cutoff.}
After applying a word $I$ of length $n$ to \eqref{ipm:correction-equation},
the transport terms other than $b\cdot\nabla\partial^I e$ are
\[
 \sum_{\varnothing\ne J\subset\{1,\ldots,n\}}
       (\partial^{I_J}b)\cdot
             \nabla\partial^{I_{J^c}}e.
\]
For $|J|=l\ge1$ put $j=n-l$. Summing over the coordinate choices
in $I$ and over subsets of size $l$ gives the bound
\begin{equation}\label{ipm:expanded-correction-wordcount}
 C_s\binom nl |b|_l\,|e|_{j+1}.
\end{equation}
The word sums already include the direction of the final gradient;
there is no extra dimension factor growing with $n$.
With $w_n=\rho^n/(n!)^2$, the exact identity is
\begin{equation}\label{ipm:transport-weight}
 \frac{w_{j+l}\binom{j+l}{l}}{w_l(j+1)w_{j+1}}
       =\rho^{-1}\frac{j+1}{\binom{j+l}{l}}
       \le\rho^{-1}.
\end{equation}
Here $\binom{j+l}{l}\ge j+1$ because $l\ge1$. Also
$j+1=n-l+1\le n$. Thus every derivative on the right of
\eqref{ipm:expanded-correction-wordcount} is already present when
$n\le M$.

Define, pointwise in time,
\[
 X_M=\sum_{n=0}^M w_n|e|_n,\qquad
 Y_M=\sum_{n=0}^M n w_n|e|_n.
\]
Summing \eqref{ipm:expanded-correction-wordcount}, using
\eqref{ipm:transport-weight}, gives
$C_s\rho^{-1}\|b\|_{\rho,M}Y_M$.
The fixed-base commutators give
$C_s\|b\|_{H^s}X_M$ after dividing each energy inequality by
the corresponding norm. The zero-order and source terms are bounded
by \eqref{ipm:expanded-correction-operator}. Since
\[
 \|b\|_{\rho,M}\le B_0+\Lambda X_M,
 \qquad w_n'=(\rho'/\rho)n w_n,
\]
enlarging the fixed power $\Lambda$ once gives
\begin{equation}\label{ipm:finite-cutoff-energy}
 X_M'\le\Lambda(X_M+X_M^2+R_k)
 +\left[\frac{\rho'}{\rho}
       +\Lambda\rho^{-1}(B_0+X_M)\right]Y_M.
\end{equation}
This is an upper Dini derivative inequality at zeros of individual
norms. To justify it there, regularize each of the finitely many
norms by $(\|v\|_{H^s}^2+\varepsilon^2)^{1/2}$ and let
$\varepsilon\downarrow0$ with $M$ fixed. No error is retained in
the initial norm.

For $M=0$, $Y_M=0$ and there are no additional-word commutators.
For $M=1$, the only such product is
$\partial b\cdot\nabla e$, bounded by $|b|_1|e|_1$; it does
not involve $|e|_2$. These endpoint cases make explicit that the
estimate is not an infinite-order cancellation requiring a nonexistent
top derivative.

\par\smallskip\noindent\textbf{Step 4. Quantitative radius and norm bootstraps.}
Put $\Delta_k=e^{-k^\vartheta/2}$ and choose
\begin{equation}\label{ipm:expanded-correction-radius}
 \rho(t)=\rho_0-2\Lambda(B_0+\Delta_k)t.
\end{equation}
The constants determining $\rho_0^{-1}$ and $\Lambda$ are fixed
powers of $\mathcal P$. Choose $Q$ so large that
\eqref{ipm:correction-frequency-separation} implies
\[
 \Lambda\rho_0^{-1}\le k^{\vartheta/100},
 \qquad \Lambda S\le k^{\vartheta/100}.
\]
Since $S\le1$ and $\vartheta$ is a fixed small number, increasing
the universal lower threshold on $k$ then gives
$2\Lambda S(B_0+\Delta_k)\le\rho_0/2$. In particular
$\rho(t)\ge\rho_0/2$.
On a bootstrap interval where $X_M\le\Delta_k\le1$,
the bracket in \eqref{ipm:finite-cutoff-energy} is nonpositive.
Starting from zero and integrating the remaining scalar inequality,
\begin{equation}\label{ipm:expanded-correction-bootstrap}
 X_M(t)\le \Lambda S e^{2\Lambda S}R_k<\Delta_k/4.
\end{equation}
To check the strict inequality, take logarithms: the left side has
logarithm at most
$\log(\Lambda S)+2\Lambda S-0.6k^\vartheta\log k$,
which is less than $-k^\vartheta/2-\log4$ for all sufficiently
large $k$. The constants are uniform in $M$.
Continuity closes this bootstrap on the whole interval where the
smooth approximation exists.

\par\smallskip\noindent\textbf{Step 5. Existence without assuming a Gevrey solution.}
Let $\mathcal J_J$ be the orthogonal Fourier projection onto a finite
symmetric box of frequencies of the lifted torus. Include the zero
frequency and choose the boxes nested and invariant under reflection.
Solve
\begin{equation}\label{ipm:galerkin-equation}
 \begin{split}
 (e_J)_t+\mathcal J_J\bigl[
 (b^a+T_{k^{-1}} e_J)\cdot\nabla e_J+Z(e_J)\bigr]
       &=-\mathcal J_J r^a,\\
 e_J(0)&=0,\qquad \mathcal J_J e_J=e_J.
 \end{split}
\end{equation}
This is a finite-dimensional ODE with smooth time coefficients and
a locally Lipschitz right side. The velocity $\mathcal V_D e_J$ is
computed by the exact Darcy operator, so its lifted divergence is
zero. It is not necessary to truncate that velocity.

When a differentiated equation is paired with a derivative of $e_J$,
move $\mathcal J_J$ to the test function. The projection commutes
with derivatives and fixes that test function; it therefore disappears
from the energy identity. All the preceding estimates hold uniformly
in $J,M$. The projection contracts every derivative $H^s$ norm.

For a fixed $J$, first use $M=0$. Estimate
\eqref{ipm:expanded-correction-bootstrap} bounds the solution's
$H^s$ norm, and hence all its finitely many Fourier coordinates,
up to time $S$. Finite-dimensional continuation therefore gives
existence on $[0,S]$. Only after this is established do we apply
the same argument for each $M\ge1$. This yields uniform bounds at
every fixed Sobolev order on an already common interval.

For each fixed integer $q$, choose a sufficiently large $M$.
The bounds give $e_J$ bounded uniformly in $C_tH^{q+2}$.
The equation, the fixed-order boundedness of $\mathcal V_D$, and
Sobolev products give a uniform bound for $(e_J)_t$ in $C_tH^q$.
At this stage the torus is fixed; compactness requires no quantitative
constant uniform over different layers. Each fixed Fourier box has
equicontinuous bounded coefficients. A diagonal subsequence converges
uniformly in time on every such box, and the two-derivative bound
controls the complementary Fourier tails. Consequently
\[
 e_J\longrightarrow e
       \quad\hbox{in }C([0,S];H^q)
       \quad\hbox{for every fixed }q.
\]
The inverse $\mathcal V_D(t)$ is uniformly bounded at each such order,
so $\mathcal V_D e_J$ converges there too. Products converge in lower
fixed norms, and $\mathcal J_J$ tends strongly to the identity;
the uniformly higher derivative bounds control its tails even on
these varying products. Passing to \eqref{ipm:galerkin-equation}
gives \eqref{ipm:correction-equation} with $e(0)=0$.

At each finite cutoff the norm estimate passes to the limit.
Monotone convergence in $M$, followed by
$\rho(t)\ge\rho_0/2$, yields
$\sup_t\|e(t)\|_{\rho_0/2}\le\Delta_k$.
The equation gives a continuous first time derivative in every
fixed Sobolev space. Further time derivatives follow inductively:
the coefficients are smooth in time, and time differentiation of
the coercive Darcy equation gives its inverse's differentiated
operator formula at every fixed order. No quantitative bound on
arbitrarily many time derivatives is required for the induction.
Thus $e$ is smooth up to both endpoints of the finite interval.

\par\smallskip\noindent\textbf{Step 6. Derivative bounds and parity.}
At $\rho_*=\rho_0/4$, one derivative costs at most
$C\rho_0^{-1}$ by \eqref{ipm:radiusderivative}.
The bounds on $b,Z,r^a$ and the equation now give
$\|e_t\|_{\rho_*}\le\mathcal P^C\Delta_k$.
The Darcy estimate gives the same bound for $\mathcal V_D e$.
This proves \eqref{ipm:correction-bound}.

All coefficients of the Darcy law are real, and the projections
preserve real functions. For the indicated joint reflection,
the scalar correction equation preserves oddness: its transport
field is odd, its scalar gradient is even, and its residual is odd.
The reflection-invariant Galerkin equations have unique solutions,
so each $e_J$ is real and odd. Both properties pass to the limit.
\end{proof}

\subsection{Physical reconstruction and gradient estimates}

Let $A=A^a+e$ be supplied by Proposition~\ref{ipm:correction} for the
approximation of Proposition~\ref{ipm:finite-approximation}, and
put $v=\mathcal V_D A$.

\begin{proposition}
\label{ipm:physical-reconstruction}
Under the integer-carrier condition \eqref{ipm:carrier-lattice}, define
\begin{equation}\label{ipm:physical-increments}
 \Theta(t,X(t,\ell y))=\ell k^{-1} A_{\rm diag}(t,y),
 \qquad
 U(t,X(t,\ell y))=\ell k^{-1} F(t,y)v_{\rm diag}(t,y),
\end{equation}
where $f_{\rm diag}(t,y):=f(t,y,km_0\cdot y)$.
Then $P_{\rm new}=P+\Theta$ is an exact odd mean-zero solution with a
smooth periodic pressure $p_{\rm new}$, and
$U=u(P_{\rm new})-u(P)$. Set $\pi=p_{\rm new}-p$.
For the leading fields in \eqref{ipm:intro-leading-matrices}, there
are global physical errors and a fixed exponent $C_*>0$ such that,
with all compositions evaluated at $X(t,\ell y)$,
\begin{equation}\label{ipm:packet-gradient-errors}
 \begin{aligned}
 (\nabla\Theta)\circ X
 &=(\nabla\Theta)_{\rm lead}\circ X+E_r\circ X,\\
 (\nabla U)\circ X
 &=(\nabla U)_{\rm lead}\circ X+E_M\circ X,\\
 (\nabla^2\pi)\circ X
 &=(\nabla^2\pi)_{\rm lead}\circ X+E_H\circ X,\\
 \sup_t(\|E_r\|_\infty+\|E_M\|_\infty+\|E_H\|_\infty)
 &\le k^{-1}\mathcal P^{C_*}.
 \end{aligned}
\end{equation}
\end{proposition}
\begin{proof}
The chain rule for graph evaluation is
\begin{equation}\label{ipm:diagonal-chain-rule}
 \nabla_y f_{\rm diag}
 =(\nabla_y f+km_0\partial_\theta f)_{\rm diag}
 =k(Df)_{\rm diag}.
\end{equation}
The lattice condition makes this a periodic function on
$\mathbb T_\ell^2$: translating $y$ by a period changes the
angle by an integer multiple of $2\pi$.

\textbf{Step 1. The vector pressure has a genuine periodic potential.}
A lifted Fourier mode has slow frequency $\ell\nu$, with
$\nu\in\mathbb Z^2$, and angular frequency $j\in\mathbb Z$.
Its $D$-frequency is
\[
 \eta=k^{-1}\ell\nu+jm_0
       =\frac{\ell}{k}(\nu+jn_0),\qquad
 n_0=(k/\ell)m_0\in\mathbb Z^2.
\]
If $\eta\ne0$, then $|\eta|\ge\ell/k$. A vector in the closed
$D$-gradient space has its Fourier coefficient parallel to $\eta$
and has zero coefficient at each resonant frequency $\eta=0$.
For the smooth vector pressure $p$ in the Darcy construction define
\[
 \widehat\zeta(\nu,j)=
       \frac{\eta\cdot\widehat p(\nu,j)}{i|\eta|^2}
       \quad(\eta\ne0),\qquad
 \widehat\zeta(\nu,j)=0\quad(\eta=0).
\]
Then $D\zeta=p$ and
$|\widehat\zeta|\le(k/\ell)|\widehat p|$. Thus $\zeta$ is a smooth
periodic potential. Its possible $k/\ell$ norm cost is used only
to establish existence of a smooth potential; all quantitative
operator estimates concerned the vector $p$.

\textbf{Step 2. Darcy law, divergence, and the scalar equation.}
For the increments \eqref{ipm:physical-increments},
let $\pi(t,X(t,\ell y))=\ell^2k^{-2}\zeta_{\rm diag}(t,y)$.
By \eqref{ipm:diagonal-chain-rule},
\[
 (\nabla_x\pi)\circ X(t,\ell y)
       =\ell k^{-1} F^{-T}p_{\rm diag}.
\]
Since $FG=F^{-T}$ and $Fa_2=e_2$, the exact lifted Darcy law gives
$U=-c\Theta e_2-\nabla_x\pi$. This pressure is periodic, because
both the graph function and the inverse parent torus map are periodic
in the required sense.

For divergence, Piola's identity with $\det F=1$ gives
\begin{equation}\label{ipm:physical-piola}
 (\operatorname{div}_xU)\circ X
   =\operatorname{div}_z(F^{-1}U\circ X)
   =k^{-1}\operatorname{div}_y v_{\rm diag}
   =(\operatorname{div}_D v)_{\rm diag}=0.
\end{equation}
In the second equality $z=\ell y$ and $F^{-1}U\circ X=\ell k^{-1} v_{\rm diag}$.
Along the parent map,
$(D_t^P\Theta)\circ X=\ell k^{-1}(A_t)_{\rm diag}$ and
$F^T(\nabla P)\circ X=R_0$. Therefore
\[
 \begin{split}
 \bigl[\partial_t(P+\Theta)+(u(P)+U)\cdot\nabla(P+\Theta)\bigr]
                 \circ X(t,\ell y)
 &=\ell k^{-1}\bigl[A_t+v\cdot R_0+v\cdot DA\bigr]_{\rm diag}\\
 &=0.
 \end{split}
\]
This verifies every term and scale factor in the physical equation.

\textbf{Step 3. The physical zero mode.}
Under the joint reflection $(y,\theta)\mapsto(-y,-\theta)$,
$F,G,a_2,R_0,L,b_0,c_A^{\ell}$ are even. The leading $A_1$ is odd.
The mean inverse returns an odd scalar for odd data; the centered
scalar inverse preserves the same parity. In the recurrence the
scalar and velocity coefficients $A_p,V_p,C_p$ are odd under this
joint reflection, while their angular primitives $\pi_p,\psi_p$
are even. Spatial or angular differentiation reverses the joint
parity, which checks the recurrence one term at a time. This joint
parity is distinct from the order-dependent parity under reflection
of $\theta$ alone established in Lemma~\ref{ipm:phase-parity}.

The exact correction is odd by its construction. The physical graph
and the odd parent map carry joint oddness to spatial oddness.
Consequently $\Theta$ and $P+\Theta$ have zero spatial integral.
Integrating the periodic Darcy law now gives $\int U=0$.
Thus the physical velocity is exactly the normalized Darcy velocity.
No zero-mean condition is imposed on the other auxiliary phase leaves.

\textbf{Step 4. Global first-derivative and Hessian errors.}
At a fixed smaller radius, the hierarchy and correction give
\[
 A=A_1+\mathcal R_A,\qquad
 \mathcal V_D A=V_1+\mathcal R_V,\qquad
 \|\mathcal R_A\|_\rho+\|\mathcal R_V\|_\rho
                       \le k^{-1}\mathcal P^C.
\]
The exponentially small correction is included in these inequalities.
One fixed radius decrease supplies the same bound for any fixed
number of lifted derivatives, with a larger fixed polynomial.
Physical differentiation of the scalar in
\eqref{ipm:physical-increments} gives
\[
 (\nabla_x\Theta)\circ X=F^{-T}(DA)_{\rm diag}.
\]
The fast part of $DA_1$ is $m_0\partial_\theta A_1$; its slow
part is $k^{-1}\nabla_yA_1$. Hence
$(\nabla\Theta)\circ X
=m(\partial_\theta A_1)_{\rm diag}+O(k^{-1}\mathcal P^C)$.
For the velocity, differentiate $Fv$ on the graph. Derivatives of
$F$ and other slow coefficients acquire a factor $k^{-1}$, whereas
the leading fast derivative is
\[
 (Fb_0)\otimes(F^{-T}m_0)\,(\partial_\theta A_1)_{\rm diag}
       =\mathfrak m(m)\otimes m\,(\partial_\theta A_1)_{\rm diag}.
\]
Here $Fb_0=\mathfrak m(F^{-T}m_0)$ follows directly by substituting
$G=F^{-1}F^{-T}$ into $L$. All remaining terms are
$O(k^{-1}\mathcal P^C)$ globally.
Finally,
\[
 \nabla^2\pi=-\nabla U-ce_2(\nabla\Theta)^T,\qquad
 \mathfrak m(m)+ce_2=c\,m_2m/|m|^2.
\]
This yields exactly \eqref{ipm:packet-gradient-errors}.
The estimates are global because the exact Darcy projection and
nonlinear correction have been estimated globally. Only the leading
terms are localized.
\end{proof}

\section{The child particle map and recovery of the parent class}
\label{ipm:sec-map}

\begin{proposition}
\label{ipm:parent-class-return}
For the child of Proposition~\ref{ipm:physical-reconstruction}, let
$\Phi$ be the flow of $b=(k^{-1}v,m_0\cdot v)$ with
$\Phi(0)=\mathrm{id}$, and let $\pi_y$ project onto the first two components.
Define
\begin{equation}\label{ipm:perturbed-label-flow}
 Y(t,z)=\ell\,\pi_y
       \Phi(t,z/\ell,(k/\ell)m_0\cdot z).
\end{equation}
Then $Y$ is volume preserving and $X_{\rm new}=X\circ Y$ is the
actual initial-label flow. The initial density and $X_{\rm new}$,
together with its first two time derivatives, satisfy
\eqref{ipm:parent-class} with $K_{\rm new}\le(K\mathcal P k/\ell)^{C_*}$
for one fixed exponent $C_*$.
\end{proposition}
\begin{proof}
It remains to show that the exact child provides the hypotheses for
the next layer, with one stage-independent exponent.

\textbf{Step 1. Small lifted transport and its time derivative.}
Use the sharper bound $\|b^a\|_{\rho_0}\le k^{-1}\mathcal P^C$
from the truncated hierarchy, before replacing it by $B_0$ in the
energy estimate. The uniform Sobolev embedding constant is fixed.
Frequency separation therefore gives
$C_sk^{-1}\mathcal P^C\le k^{-1/2}$. The exact correction contributes
at most $C_s\mathcal P^C e^{-k^\vartheta/2}\le k^{-1/2}$.
Together with the Darcy bounds these estimates give
\begin{equation}\label{ipm:small-lifted-transport}
 \sup_t\sum_{|I|=n}\|\partial^Ib(t)\|_\infty
       \le B R_b^n(n!)^2,\qquad
 B\le2k^{-1/2},\quad R_b\le\mathcal P^C.
\end{equation}
The scalar $A=A^a+e$ and its time derivative have polynomial Gevrey
norms at one fixed smaller radius. Differentiating the exact Darcy
equation, as in \eqref{ipm:darcy-time-derivative}, uses only
$G_t,(a_2)_t,A,A_t$. Consequently
\begin{equation}\label{ipm:lifted-time-coefficient}
 \sup_t\sum_{|I|=n}\|\partial^Ib_t(t)\|_\infty
       \le\mathcal P^C(\mathcal P^C)^n(n!)^2.
\end{equation}
No second scalar time derivative or time differentiation of the
formal pressure arrays is used.

Frequency separation implies $SBR_b$ is smaller than the fixed
threshold in Lemma~\ref{ipm:smallflow}. Thus \eqref{ipm:smallflowbound} gives
\begin{equation}\label{ipm:lifted-flow-estimate}
 \sup_t\sum_{|I|=n}
       \|\partial^I(\Phi-\mathrm{id})\|_\infty
       \le CBS(CR_b)^n(n!)^2.
\end{equation}
That lemma proves this for finite word cutoffs using a positive
composition majorant, then passes to all orders. Equations
$\Phi_t=b\circ\Phi$ and
$\Phi_{tt}=(b_t+(\nabla_{y,\theta}b)b)\circ\Phi$ give fixed polynomial size and
growth constants for the first two time derivatives as well, by
\eqref{ipm:smallflow-first-time}--\eqref{ipm:smallflow-second-time}.
All derivatives here are ordinary derivatives in the lifted
variables, not the degenerate derivative $D$.

\textbf{Step 2. The physical relabeling and its Jacobian.}
For the relabeling \eqref{ipm:perturbed-label-flow}, since $b=(k^{-1} v,m_0\cdot v)$, the quantity
$\theta-km_0\cdot y$ is constant along its trajectories.
Starting on the zero-offset graph therefore gives
\begin{equation}\label{ipm:label-velocity}
 Y_t(t,z)=\ell k^{-1}
       v(t,Y(t,z)/\ell,km_0\cdot Y(t,z)/\ell).
\end{equation}
The induced vector field has zero divergence in $z$, because
its divergence is $(\operatorname{div}_D v)_{\rm diag}=0$ after
the scaling and graph chain rule. Therefore
$\det D_zY=1$ and $Y(0,z)=z$.
Periodicity of the lifted field and the integer carrier imply
$Y(t,z+2\pi n)=Y(t,z)+2\pi n$. Thus $Y$ is a smooth torus
flow, not only a formal graph map.

Define
\begin{equation}\label{ipm:new-physical-flow}
 X_{\rm new}(t,z)=X(t,Y(t,z)).
\end{equation}
Differentiating in time gives
\begin{equation}\label{ipm:child-flow-identity}
 (X_{\rm new})_t
       =X_t(t,Y)+D_zX(t,Y)Y_t
       =(u(P)+U)(t,X_{\rm new}).
\end{equation}
The last identity uses exactly \eqref{ipm:physical-increments}.
Together with $X_{\rm new}(0,z)=z$, it identifies the map as the
actual initial-label flow of the constructed child.

\textbf{Step 3. Quantitative closure of all four fields.}
Graph differentiation in \eqref{ipm:perturbed-label-flow} multiplies
the derivative growth constant by at most $Ck/\ell$.
The factor $\ell$ outside the flow and the bounds for $\Phi$,
$\Phi_t,\Phi_{tt}$ show that $Y-\mathrm{id},Y_t,Y_{tt}$ have
Gevrey size at most $\mathcal P^C$ and growth at most
$(k/\ell)\mathcal P^C$.
Apply Lemma~\ref{ipm:composition} to
\[
 X_{\rm new}-\mathrm{id}
       =(X-\mathrm{id})\circ Y+(Y-\mathrm{id}).
\]
This expression uses periodic displacement fields, so their zeroth
norms are finite even for a lift of the torus flow.
The composition growth is a fixed polynomial in $K,\mathcal P,k/\ell$.
For the time derivatives use the explicit identities
\[
 \begin{split}
 (X\circ Y)_t
       &=X_t\circ Y+(DX\circ Y)Y_t,\\
 (X\circ Y)_{tt}
       &=X_{tt}\circ Y+2(DX_t\circ Y)Y_t\\
       &\quad +(D^2X\circ Y)[Y_t,Y_t]+(DX\circ Y)Y_{tt}.
 \end{split}
\]
Only the four fields in \eqref{ipm:parent-class} and at most two
additional spatial derivatives appear. A fixed derivative shift
changes their Gevrey size and growth by fixed polynomial factors.
Products and the composition lemma therefore give the same kind
of bound for both time derivatives of $X_{\rm new}$.

At time zero the physical labels are the physical coordinates.
The explicitly prescribed density is $P(0)$ plus the graph of
the finite hierarchy, with no exact correction. Its Gevrey norm
before graph evaluation is polynomial in $\mathcal P$ by the
tail estimate \eqref{ipm:expanded-hierarchy-split-tail}; its graph growth
costs at most $Ck/\ell$.
Hence the initial density satisfies the same fixed polynomial
bound. Choose one exponent $C_*$ larger than all these finitely
many size and growth exponents. Taking
$K_{\rm new}=(K\mathcal P k/\ell)^{C_*}$ then bounds each field's
$n$th derivative by $K_{\rm new}^{n+1}(n!)^2$.
This proves \eqref{ipm:packet-output-class} with exactly the input
class required at the next stage.
\end{proof}

\section{The initial density and its all-time increment bounds}
\label{ipm:sec-initial-increment}

\begin{proposition}
\label{ipm:initial-increment-bounds}
For the exact child of Proposition~\ref{ipm:physical-reconstruction},
the leading density increment is
\begin{equation}\label{ipm:leading-density}
 \Theta_{\rm lead}(t,X(t,\ell y))
 =\frac{\ell}{k}\alpha\chi(y)a(t,y)f_\delta(km_0\cdot y).
\end{equation}
The exact increment decomposes as
$\Theta=\Theta_{\rm osc}+\Theta_{\rm mean}+\Theta_{\rm corr}$, with
finite oscillatory and phase-independent parts and a global exact
correction. For each fixed integer $j\ge0$ and every $t\in[0,S]$,
\begin{equation}\label{ipm:packet-history}
 \begin{aligned}
 \|\Theta_{\rm osc}(t)\|_{C^j}
 &\le C_j\ell^{1-j}k^{j-1}\mathcal P^{C_j}\gamma(t),\\
 \|\Theta_{\rm mean}(t)\|_{C^j}
 &\le C_j\ell^{1-j}k^{-2}\mathcal P^{C_j},\\
 \|\Theta_{\rm corr}(t)\|_{C^j}
 &\le C_j\ell^{1-j}k^{j-1}\mathcal P^{C_j}
       e^{-k^\vartheta/2}.
 \end{aligned}
\end{equation}
Its oscillatory initial density is exactly $\Theta_{\rm lead}(0)$;
all higher oscillatory coefficients vanish initially, and
$\Theta_{\rm corr}(0)=0$. The remaining initial increment consists
of the phase means selected by Proposition~\ref{ipm:meaninverse}.
\end{proposition}
\begin{proof}
For each fixed $j$, differentiating a graph function $j$ times costs
at most a constant times $k^j$. Scaling $z=\ell y$ costs
$\ell^{-j}$. The identity
$\nabla_x=F^{-T}\nabla_z$, applied recursively, involves only a
fixed number of derivatives of $F^{-1}$ at a fixed $j$.
Their sizes are fixed powers of $K$. Define the physical evaluation
operator by $(\mathcal T_tf)(X(t,\ell y))=\ell k^{-1} f_{\rm diag}(y)$.
The physical scalar prefactor $\ell/k$ gives
\[
 \|\mathcal T_tf\|_{C^j}
       \le C_j\ell^{1-j}k^{j-1}\mathcal P^{C_j}\|f\|_\rho.
\]
This inequality is precisely the iterated chain rule for
\eqref{ipm:physical-increments}.
Sobolev embedding in the lifted variables supplies evaluation on the
graph without an additional frequency factor.

By \eqref{ipm:expanded-hierarchy-split-tail} and the leading coefficient
bound, the finite oscillatory hierarchy has norm at most
$\mathcal P^C\gamma(t)$ after summation; its small factor is retained
before time suprema are taken. The mean density is
$\ell\sum_{p=2}^Nk^{-p}B_p$. It has no angular derivatives and
therefore costs no $k^j$, yielding
$C_j\ell^{1-j}k^{-2}\mathcal P^{C_j}$.
Apply \eqref{ipm:correction-bound} separately to obtain the final
line of \eqref{ipm:packet-history}; no factor $\gamma(t)$ is asserted
for that correction. At time zero, all $A_p$ for $p\ge2$ and $e$
vanish. The only initial oscillation is $A_1(0)$, which gives
\eqref{ipm:leading-density} at $t=0$, and the selected initial means are precisely those included in the mean bound.
\end{proof}

\section{The conditional exact child}
\label{ipm:sec-packet}

We collect the analytic construction into the statement used in
the induction. The parent, carrier, amplitude, and size parameter are
those of Sections~\ref{ipm:sec-history} and \ref{ipm:calculus}.
The pressure bound is assumed only for the parent; the geometric estimates
will supply it for the child.

\begin{proposition}\label{ipm:packet}
There are fixed positive constants $r,Q,C_*$, a fixed integer $h$,
and a fixed sufficiently small $\vartheta>0$ with the following property.
Suppose \eqref{ipm:parent-class} holds, $\nabla^2p\le I$ on $[0,S]$, and
\begin{equation}\label{ipm:packet-separation}
 \ell K^{2A+3}\le\varepsilon_*.
\end{equation}
Assume also the frequency separation \eqref{ipm:correction-frequency-separation}.
Then an exact odd mean-zero unforced solution $P_{\rm new}$ exists on
$[0,S]$. Write $\Theta=P_{\rm new}-P$, $m=F^{-T}m_0$, and evaluate lifted
quantities at $(t,y,\theta)=(t,y,km_0\cdot y)$ when composing with
$x=X(t,\ell y)$. Its leading density increment is given by \eqref{ipm:leading-density}.
The global gradient and Hessian errors satisfy
\eqref{ipm:packet-gradient-errors} of Proposition~\ref{ipm:physical-reconstruction}.
The exact increment has the decomposition and all-time bounds
\eqref{ipm:packet-history} of Proposition~\ref{ipm:initial-increment-bounds}.
The oscillatory initial density is exactly \eqref{ipm:leading-density}
at $t=0$, and $\Theta_{\rm corr}(0)=0$.
Finally, $P_{\rm new}(0)$ and the initial-label flow $X_{\rm new}$,
including its first two time derivatives, satisfy
\eqref{ipm:parent-class} with
\begin{equation}\label{ipm:packet-output-class}
 K_{\rm new}\le(K\mathcal P k/\ell)^{C_*}.
\end{equation}
All exponents are the same at every invocation. The proposition does not
assert $\nabla^2p_{\rm new}\le I$; that conclusion requires the geometric
sign and size estimates for the leading pressure-Hessian increment in \eqref{ipm:intro-leading-matrices}.
\end{proposition}

\subsection{Dependency order of the fixed exponents}

\begin{proof}[Proof of Proposition~\ref{ipm:packet}]
Lemma~\ref{ipm:packet-calculus} and Corollary~\ref{ipm:common-growth} supply
the coefficient bounds and the three inverses.
Proposition~\ref{ipm:hierarchy-construction} and
Lemma~\ref{ipm:coefficients} construct and estimate the hierarchy.
Proposition~\ref{ipm:finite-approximation}, Proposition~\ref{ipm:correction},
and Proposition~\ref{ipm:physical-reconstruction} give the exact
physical child and its gradient estimates.
Proposition~\ref{ipm:parent-class-return} recovers the parent class,
and Proposition~\ref{ipm:initial-increment-bounds} supplies the initial data
and all-time increment bounds.
It remains to choose their constants in a common order.
First fix the dimension, $s$, cutoff, and profile constants.
The algebra of the parent class determines $A_0$ and the exponent $A$ in
Lemma~\ref{ipm:tube-localization}; the localization exponent $2A+3$ suffices for
\eqref{ipm:centered-propagator}.
Next derive the polynomial thresholds and prefactors in the mean,
scalar, and lifted Darcy inverses, including
\eqref{ipm:darcy-time-derivative}. This is a finite list independent
of the hierarchy depth. Choose $r$ so $R=\mathcal P^r$ exceeds those
thresholds and the leading profile growth, and then choose $h$ to
cover fixed shift costs. The hierarchy determines the fixed exponents
$A,D$ in \eqref{ipm:order-coefficient-bound} and the low-order sizes in
\eqref{ipm:approximation-bounds}. Choose $\vartheta>0$ with
$D\vartheta<1/200$, as in the truncation estimate.
The exact correction and flow calculation now determine their fixed
polynomial exponents and a fixed $C_*$ dominating the finitely many
outputs of Proposition~\ref{ipm:packet}.
Finally choose $Q$ to enforce all frequency separations used for
truncation, radius retention, Gronwall domination, and small lifted
flow. Increasing $Q$ changes none of the preceding exponents.
Choose the universal threshold $k_*$ last.

All applications of the mean, scalar, and Darcy inverses in the
hierarchy accept and return the same $R$ and arbitrary input shift.
The only radius decrease depending on evolution is the single one
in Proposition~\ref{ipm:correction}, performed after the hierarchy has been
truncated. Its initial and final radii are fixed inverse powers of
$\mathcal P$. No inverse amplitude seed enters any choice. This
completes the proof of Proposition~\ref{ipm:packet}.
\end{proof}

\section{From the growing wave to the layer inputs}
\label{ipm:seventh-cascade-interface}

The analytic theorem supplies global value remainders. We now
verify the geometric inputs needed to use it repeatedly: the five transfer
errors, the positive phase direction on the supported tube, and the global
upper pressure bound.

For clarity, the outputs of Proposition~\ref{ipm:packet} used in this section are
the following. The parent is exact on $[0,S]$, $S\le1$, has
$H(P):=\nabla^2p(P)\le I$, and has the stated Gevrey bounds, with parameter
$K$, on the four fields
\begin{equation}
 P(0),\qquad X-\operatorname{id},\qquad X_t,\qquad X_{tt}
 \label{ipm:cascade-four-fields}
\end{equation}
in initial physical labels. For a layer let $\ell,k,\delta$ be its scales,
$m_0$ its unit initial phase vector, $m=F^{-T}m_0$, and let
$\gamma(t)=\alpha a(t,0)$ be the actual scalar amplitude at the origin.
Use the size parameter $\mathcal P$ defined in \eqref{ipm:packet-parameter}.
There are fixed constants $Q,C_{\rm out},\vartheta>0$ such that the
theorem applies when $\ell K^{2A+3}$ is sufficiently small and
$\mathcal P^Q\le k^{\vartheta/100}$. It returns an exact child, with the
global gradient, velocity-gradient, and pressure-Hessian expansions stated
there, with error at most
\begin{equation}
 \epsilon=k^{-1}\mathcal P^{C_{\rm out}}.
 \label{ipm:cascade-epsilon}
\end{equation}
Its new four-field parameter is at most
$(K\mathcal P k/\ell)^{C_{\rm out}}$. We may enlarge fixed exponents and
constants once and for all.

The quantifiers on these constants will matter. The analytic proposition,
the fixed profile and cutoff, and the geometric lemmas determine a
finite list of universal constants and exponents. We choose that list first.
We then choose one inherited regularity exponent, one starting index $J$,
and finally one base size. These choices produce the entire infinite
sequence of exact stages. Only when proving convergence in a particular
$C^m$ or $C^\eta$ norm do we choose a further tail index depending on
$m$ or $\eta$. Such a tail index discards finitely many already constructed
smooth increments; it does not change the stages, the starting index,
or the base data.
\label{ipm:expanded-cascade-quantifiers}

We also use the proposition's all-time decomposition into a finite oscillatory
part, a part independent of the phase, and an exact correction which
vanishes initially.
For each fixed $m\ge0$, its physical density increment $\theta$ obeys
\begin{align}
 \|\theta_{\rm osc}(t)\|_{C^m}
   &\le C_m\ell^{1-m}k^{m-1}\mathcal P^{C_m}\gamma(t),
       \label{ipm:cascade-osc-bound}\\
 \|\theta_{\rm mean}(t)\|_{C^m}
   &\le C_m\ell^{1-m}k^{-2}\mathcal P^{C_m},
       \label{ipm:cascade-mean-bound}\\
 \|\theta_{\rm corr}(t)\|_{C^m}
   &\le C_m(k/\ell)^{m+1}\mathcal P^{C_m}
                   \exp(-k^\vartheta/2).
       \label{ipm:cascade-correction-bound}
\end{align}
The last, deliberately weakened, bound is convenient for every fixed $m$.
At $t=0$ the exact correction and all higher oscillatory coefficients vanish.
The oscillatory factor $\gamma(t)$ is not asserted for the exact correction.

\subsection{Bounds on the transfer errors}

Suppose an odd exact older solution $O$ satisfies
$H(O)\le I$ and $\|\nabla O\|_\infty\le C_{\nabla}h_{\rm old}$,
where $h_{\rm old}\ge1$. Let $P=O+\Theta$ be an exact layer
child whose density-gradient, velocity-gradient, and pressure-Hessian
remainders $E_r,E_M,E_H$ in \eqref{ipm:packet-gradient-errors} (with its parent $P$ replaced by $O$ and its child $P_{\rm new}$ replaced by $P$) satisfy
\[
 \sup_t\bigl(\|E_r(t)\|_\infty+\|E_M(t)\|_\infty
                    +\|E_H(t)\|_\infty\bigr)\le\epsilon\le1,
\]
with spatial norms over the whole physical torus.
Write the leading density-gradient increment at the origin as
$Gn$, where $G>0$ and $n=(\cos\alpha,\sin\alpha)$.
Linearity of Darcy's law and the layer expansions \eqref{ipm:packet-gradient-errors} give
\begin{equation}\label{ipm:seventh-cascade-central-decomposition}
 \begin{aligned}
 r_P&=Gn+R_r,& R_r&=r_O+E_r,\\
 M_P&=-cG\cos\alpha\,n^\perp\otimes n+R_M,
     &R_M&=M_O+E_M,\\
 H_P&=-cG\sin\alpha\,n\otimes n+R_H,
     &R_H&=H_O+E_H.
 \end{aligned}
\end{equation}
\begin{lemma}
\label{ipm:five-transfer-errors}
Under this setup, the remainders and ray variables satisfy
\[
 |R_r|+\|R_M\|+\|R_H\|+|\alpha'|+|(\log G)'|
 \le C_E(1+h_{\rm old}),\qquad C_E:=5+C_{\nabla}(1+6c).
\]
\end{lemma}
\begin{proof}
Oddness of the flow of $O$ fixes the origin. The angular identity in \eqref{ipm:intro-angular-speed} and the amplitude identity \eqref{ipm:rayidentity} therefore give
\[
 |\alpha'|\le\|M_O\|,\qquad
 |(\log G)'|=|(n^\perp)^TH_On^\perp|\le\|H_O\|.
\]
From $M_O=-H_O-ce_2r_O^T$ and
$\operatorname{tr}H_O=-c(r_O)_2$, the upper bound $H_O\le I$ implies
\[
 \|H_O\|\le1+cC_{\nabla}h_{\rm old},\qquad
 \|M_O\|\le1+2cC_{\nabla}h_{\rm old}.
\]
Adding the five quantities required by Lemma~\ref{ipm:amplification} gives the
fully explicit bound
\begin{equation}\label{ipm:seventh-cascade-five-errors}
 \begin{aligned}
 |R_r|+\|R_M\|+\|R_H\|+|\alpha'|+|(\log G)'|
 &\le |r_O|+2\|M_O\|+2\|H_O\|+\epsilon\\
 &\le4+C_{\nabla}(1+6c)h_{\rm old}+\epsilon\\
 &\le C_E(1+h_{\rm old}).
 \end{aligned}\qedhere
\end{equation}
\end{proof}

\subsection{History, future, and Hessian budgets}

\begin{lemma}
\label{ipm:explicit-packet-budget}
Let a layer over an exact parent $P$ have activation $s$, target
$t_*$, retained horizon $S=t_*+2W_{\rm next}\le1$, target gradient
$h$, and target angle $\pi/2+1/x_{\rm tar}$, where
$x_{\rm tar}\ge x\ge4$. Suppose its angle at the origin increases
on $[s,t_*]$ from at least $1/(2x)$. Set
$g(t)=\alpha a(t,0)|m(t,0)|/\delta$, so $g(t_*)=h$, and assume
\eqref{ipm:cascade-future-localization} and
\begin{equation}\label{ipm:seventh-cascade-target-gain}
 \frac{g(t_*)}{g(s)}\ge e^{c_0x},\qquad
 H(P)\le I,\qquad \|\nabla P\|_\infty\le C_{\nabla}h_{\rm par}.
\end{equation}
Then the retained-angle bounds proved in
Proposition~\ref{ipm:packet-history-geometry} hold, and the three
time-range bounds and positive supported phase satisfy
\eqref{ipm:seventh-cascade-three-time-ranges}--\eqref{ipm:seventh-cascade-supported-positive}.
For the cutoff and profile of Lemma~\ref{ipm:profilelemma}, with
the constant $C_f$ in \eqref{ipm:profile-bounds}, the global pressure and gradient estimates
\eqref{ipm:seventh-cascade-global-budget}--\eqref{ipm:seventh-cascade-gradient-bounds}
hold with the explicit constants displayed there. If
$|m(t,0)|^{-1}\le B_mK^A$, then the seed bounds
\eqref{ipm:seventh-cascade-seed-parameter} also hold.
\end{lemma}
\begin{proof}
The first condition in \eqref{ipm:seventh-cascade-target-gain} is supplied by
\eqref{ipm:transfer-gain} in Lemma~\ref{ipm:amplification};
the other two are inherited exact-parent bounds.

The first bound in \eqref{ipm:cascade-future-localization} gives
$|\beta(t)-\beta(t_*)|\le1/(4x_{\rm tar})$ for $t_*\le t\le S$,
as in the proof of Proposition~\ref{ipm:packet-history-geometry}.
Moreover, with $n(t)=m(t,0)/|m(t,0)|$,
\[
 \frac d{dt}\log g=-(n^\perp)^TH(P)n^\perp
 \le1+cC_{\nabla}h_{\rm par}.
\]
Thus \eqref{ipm:cascade-future-localization} bounds its forward
growth exponent by $1/(4x_{\rm tar})$, and in particular
$g(t)\le2h$ for $t_*\le t\le S$.
The lower logarithmic derivative bound $-1$ gives the two backward
bounds
\begin{equation}\label{ipm:seventh-cascade-three-time-ranges}
 \begin{aligned}
 g(t)&\le e h e^{-c_0x},&&0\le t\le s,\\
 g(t)&\le e h,&&s\le t\le t_*,\\
 g(t)&\le2h,&&t_*\le t\le S.
 \end{aligned}
\end{equation}
Here $S\le1$ was used to replace each backward factor by $e$.
The first line uses the gain in
\eqref{ipm:seventh-cascade-target-gain}; it is not a forward
stability assertion.

The second bound in \eqref{ipm:cascade-future-localization} ensures
$\ell K^{2A+3}\le1/(20\pi)$. Hence
\eqref{ipm:cascade-tube-comparison} gives, for
$y\in\operatorname{supp}\chi$ and $t\in[0,S]$,
\begin{equation}\label{ipm:seventh-cascade-tube-hypothesis}
 \frac{a(t,\ell y)|m(t,\ell y)|}{a(t,0)|m(t,0)|}\le2,
 \qquad
 \left|\frac{m(t,\ell y)}{|m(t,\ell y)|}-n(t)\right|
 \le\frac1{10\pi x_{\rm tar}}.
\end{equation}

On the active interval $[s,t_*]$, the angle range \eqref{ipm:transfer-angle-range} in Lemma~\ref{ipm:amplification} gives
\[
 \sin\beta(t)\ge\sin(1/(2x))\ge\frac1{\pi x},\qquad s\le t\le t_*.
\]
The direction error in
\eqref{ipm:seventh-cascade-tube-hypothesis} is at most
$1/(10\pi x)$, so the second phase component on the entire active
supported tube is positive. On the retained interval, the angle estimate following from
\eqref{ipm:cascade-future-localization}, as in
Proposition~\ref{ipm:packet-history-geometry}, gives
$\sin\beta(t)\ge1/2$, with an even smaller relative direction error. Thus
\begin{equation}\label{ipm:seventh-cascade-supported-positive}
 m_2(t,\ell y)>0\quad\hbox{on the supported tube for }s\le t\le S.
\end{equation}
No assertion about the sign of this component is made on $[0,s]$.

Use $C_\chi$ from Corollary~\ref{ipm:packet-budget} and the profile bounds
\eqref{ipm:profile-bounds}.
Use the leading pressure-Hessian increment in \eqref{ipm:intro-leading-matrices},
evaluated at $X(t,\ell y)$.
For a unit vector $v$, on $[s,S]$ the sign just proved gives
\[
 v^T(\nabla^2\pi)_{\rm lead}v
 \le cC_f\alpha\chi a|m|
 \le2e c C_\chi C_f\delta h.
\]
On $[0,s]$ we instead use its absolute value and the first line of
\eqref{ipm:seventh-cascade-three-time-ranges}:
\[
 \|(\nabla^2\pi)_{\rm lead}\|
 \le cC_f^2\frac{\alpha\chi a|m|}{\delta}
 \le2e c C_\chi C_f^2h e^{-c_0x}.
\]
Let the sum of the three global layer remainder norms be at most
$\epsilon$. The time split and the fact that the leading pressure-Hessian
increment vanishes outside the tube give, globally,
\begin{equation}\label{ipm:seventh-cascade-global-budget}
 H(P+\Theta)\le H(P)
 +\left[C_H(\delta h+h e^{-c_0x})+\epsilon\right]I,
 \qquad C_H=2e c C_\chi(C_f+C_f^2).
\end{equation}
In particular a global phase mean is covered by the global remainder;
it has not been assigned the oscillatory history factor.

With $C_{\rm lead}$ defined in \eqref{ipm:cascade-gradient-increment},
the same estimates give the two gradient bounds used in the induction
and limit:
\begin{equation}\label{ipm:seventh-cascade-gradient-bounds}
 \begin{aligned}
 \|\nabla\Theta(t)\|_\infty
 &\le C_{\rm lead}h+\epsilon,
       &&0\le t\le S,\\
 \|\nabla\Theta(t)\|_\infty+
       \|\nabla u(\Theta)(t)\|_\infty
 &\le(1+c)C_{\rm lead}h e^{-c_0x}+\epsilon,
       &&0\le t\le s.
 \end{aligned}
\end{equation}
The velocity bound uses $|\mathfrak m(\xi)|\le c$ and the global
velocity-gradient remainder. These constants are fixed before
choosing the scales.

Finally, if the parent class gives
$|m(t,0)|^{-1}\le B_mK^A$, then
\begin{equation}\label{ipm:seventh-cascade-seed-parameter}
 \sup_{[0,S]}\gamma(t)\le eB_mK^A\delta h,
 \qquad
 \gamma(t)\le eB_mK^A\delta h e^{-c_0x}\quad(t\le s).
\end{equation}
This follows from $\gamma=\delta g/|m|$. Neither formula requires a
lower bound for $\alpha$ or for $\gamma$.
\end{proof}

\section{Choosing the stage parameters}
\label{ipm:sec-parameters}

First fix every universal exponent in Proposition~\ref{ipm:packet} and in the
preceding lemmas. Increase $L$ to cover
\eqref{ipm:cascade-future-localization}. Fix
$C_{\nabla}:=4\max(C_{\rm lead},C_f^2,1)$, with $C_{\rm lead}$ from
\eqref{ipm:cascade-gradient-increment}. This is the global gradient constant.
Fix also $C_E:=5+C_{\nabla}(1+6c)$ from
Lemma~\ref{ipm:five-transfer-errors}.
Fix the numerical thresholds for transfer, supported phase positivity,
and the future window, as specified above.
Choose a fixed inherited exponent
\begin{equation}
 A_K>2C_{\rm out}+2.
 \label{ipm:cascade-inherited-exponent}
\end{equation}
Next choose a sufficiently large integer $J$, depending only on these
fixed constants. The size of the base is chosen last.

For $j\ge J$, set
\begin{align}
 x_j&=j^2x_{j-1},&
 h_j&=\exp(x_{j-1}/j^5),&
 \widehat k_j&=\exp(x_{j-1}/j^2),
       \label{ipm:cascade-scales}\\
 \delta_j&=\exp(-x_{j-1}/j^3),&
 \ell_j&=\exp(-x_{j-1}/j^{7/2}),&
 W_j&=\frac{2x_{j-1}^2}{c h_{j-1}}.
       \nonumber
\end{align}
The actual $k_j$ differs from $\widehat k_j$ by at most $C\ell_j$ to
enforce the lattice condition below. The base values are prescribed by
\begin{equation}
 h_{J-1}=\exp\bigl(x_{J-1}/(J-1)^7\bigr),\qquad
 k_{J-1}=\exp\bigl(x_{J-1}/(J-1)^4\bigr).
 \label{ipm:cascade-base-bookkeeping}
\end{equation}
The latter is a bookkeeping bound, not a Fourier carrier for the base.

\begin{lemma}[Uniform domination of polynomial factors by the scales]
\label{ipm:scale-envelopes}
Put $X=x_{J-1}\ge1$ and $a_j=((j-1)!/(J-1)!)^2$.
For fixed $p,r,q\ge0$, $b>0$, and $J\ge2q+4$, the sequence
$x_{j-1}=Xa_j$ satisfies
\begin{equation}
 \frac{a_{j+1}/(j+1)^q}{a_j/j^q}
       =j^2\left(\frac j{j+1}\right)^q\ge2,\qquad j\ge J.
 \label{ipm:expanded-cascade-doubling}
\end{equation}
For a finite constant $D_{p,r,q,J}$ independent of $X$,
\begin{equation}
 \sup_{j\ge J}
 \frac{p\log j+r\log(Xa_j)}{Xa_j/j^q}
       \le D_{p,r,q,J}\frac{1+\log X}{X}
       \longrightarrow0\quad(X\to\infty).
 \label{ipm:expanded-cascade-polynomial-uniformity}
\end{equation}
Moreover,
\[
 \sum_{j\ge J}j^p x_{j-1}^r e^{-b x_{j-1}/j^q}
       \longrightarrow0\qquad(X\to\infty).
\]
\end{lemma}
\begin{proof}
The definition of $a_j$ gives
\eqref{ipm:expanded-cascade-doubling}, since $J\ge2q+4$.
Thus $x_{j-1}/j^q\ge (X/J^q)2^{j-J}$.
To prove \eqref{ipm:expanded-cascade-polynomial-uniformity}, note that
$\log a_j\le2j\log j$ and $a_j\ge J^{2(j-J)}$.
Hence $j^q(p\log j+r\log a_j)/a_j$ is bounded. The coefficient
of $\log X$ is bounded by $rJ^q2^{-(j-J)}$ using
\eqref{ipm:expanded-cascade-doubling}.
Consequently, for fixed $b>0$,
\begin{equation}
 \sum_{j\ge J}j^p x_{j-1}^r e^{-b x_{j-1}/j^q}
       \longrightarrow0\quad(X\to\infty).
 \label{ipm:expanded-cascade-uniform-series}
\end{equation}
To prove the summation assertion, use
\eqref{ipm:expanded-cascade-polynomial-uniformity} to absorb the
polynomial by half the exponent. With $A=bX/(2J^q)$, the remaining
series is bounded by
$\sum_{n\ge0}e^{-A2^n}\le e^{-A}/(1-e^{-A})$.
Only $q\in\{0,2,3,7/2,5,7\}$ is needed for the construction.
The polynomial degrees used later for a fixed norm can be arbitrary
without changing this finite list or the chosen $J$.
\end{proof}

\begin{lemma}[Simultaneous scale inequalities]
\label{ipm:scales}
After choosing $J$ sufficiently large and then $x_{J-1}$ sufficiently
large, the following inequalities hold at every stage. Suppose
$K_{j-1}\le k_{j-1}^{A_K}$, and choose $\Gamma_j$ by
\eqref{ipm:cascade-gamma-choice}. Then
\begin{align}
 x_j\ell_jK_{j-1}^{2A+3}&\ll1,&
 \mathcal P_j^Q&\le k_j^{\vartheta/100},
       \label{ipm:cascade-scale-interface}\\
 k_j^{-1}\mathcal P_j^{C_{\rm out}}&\le k_j^{-1/2},&
 (K_{j-1}\mathcal P_j k_j/\ell_j)^{C_{\rm out}}
       &\le k_j^{A_K}.
       \label{ipm:cascade-scale-output}
\end{align}
For $j>J$, the geometric errors $E_j\le C(1+h_{j-2})$ satisfy
\begin{equation}
 \frac{E_jj^2x_{j-1}^3}{h_{j-1}}\ll1.
 \label{ipm:cascade-scale-geometry}
\end{equation}
Furthermore,
\begin{equation}
 \frac{W_{j+1}}{W_j}\le\frac1{16},\qquad
 (1+h_{j-1})W_{j+1}\ll x_j^{-1},
 \label{ipm:cascade-scale-horizons}
\end{equation}
and each of the series
\begin{equation}
 \sum_{j\ge J}\delta_jh_j,\qquad
 \sum_{j\ge J}h_je^{-c_0x_{j-1}},\qquad
 \sum_{j\ge J}k_j^{-1/2}
 \label{ipm:cascade-budget-series}
\end{equation}
can be made arbitrarily small by increasing the base size.
\end{lemma}

\begin{proof}
Write $x=x_{j-1}$. For $j>J$, the definitions give
\begin{equation}
 \log h_{j-1}=\frac{x}{(j-1)^7},\qquad
 \log h_{j-2}=\frac{x}{(j-1)^2(j-2)^7},\qquad
 \log\widehat k_{j-1}=\frac{x}{(j-1)^4}.
 \label{ipm:cascade-prior-scales}
\end{equation}
The first and third formulas hold at $j=J$ by
\eqref{ipm:cascade-base-bookkeeping}; the second at $j=J+1$
uses the prescribed base value and is also exact.
To handle rounding explicitly, the mean-value theorem gives
$|\log k_j-\log\widehat k_j|
\le2C\ell_j/\widehat k_j$ once $C\ell_j/\widehat k_j\le1/2$.
Thus, for the fixed starting index $J$, one sufficiently large base size
$x_{J-1}$ ensures, uniformly in $j\ge J$,
\begin{equation}
 \frac9{10}\frac{x}{j^2}\le\log k_j
       \le\frac{11}{10}\frac{x}{j^2},\qquad
 \log K_{j-1}\le\frac{11A_K}{10}\frac{x}{(j-1)^4}.
 \label{ipm:expanded-cascade-rounded-logs}
\end{equation}
The second bound uses the inherited hypothesis, with its exact
bookkeeping version at the first step.

For a sufficiently large fixed $J$, denominator comparison gives
\[
 \log(h_j/h_{j-1})\ge\frac{x}{2j^5},\qquad
 \log(h_{j-1}/h_{j-2})\ge\frac{x}{2j^7}\quad(j>J).
\]
For example, the second difference equals
$x(j-1)^{-7}[1-(j-1)^5/(j-2)^7]$, whose bracket is at least
$1/2$ when $J$ is large. Since $h_{j-2}\ge1$, the geometric error
in \eqref{ipm:cascade-scale-geometry} is at most
$2C_Ej^2x^3e^{-x/(2j^7)}$. Formula
\eqref{ipm:expanded-cascade-uniform-series} makes its supremum
smaller than the prescribed transfer threshold. The first step has
exactly frozen geometry and no such error.

For the horizons use
\[
 \frac{W_{j+1}}{W_j}=j^4\frac{h_{j-1}}{h_j},\qquad
 (1+h_{j-1})W_{j+1}x_j
       \le C j^6x^3\frac{1+h_{j-1}}{h_j}.
\]
Both right sides are bounded by fixed polynomials times
$e^{-x/(2j^5)}$. Their uniform smallness gives the precise factor
$1/16$ and the chosen future-angle threshold. It also gives
$h_{j-1}/h_j\le1/4$, as required for the global gradient induction.

The localization logarithm satisfies
\[
 \log(x_j\ell_jK_{j-1}^{2A+3})
   \le \log(j^2x)-\frac{x}{j^{7/2}}
                   +\frac{2(2A+3)A_Kx}{(j-1)^4}.
\]
Choose $J$ so large that the last term is at most
$x/(2j^{7/2})$. Formula
\eqref{ipm:expanded-cascade-polynomial-uniformity} then makes the
entire expression arbitrarily negative, proving both the analytic
and the stronger supported-direction localization.

From $\Gamma_j=\max(1,CK_{j-1}^A\delta_jh_j)$,
$\log\Gamma_j\le C+A\log K_{j-1}+\log h_j$.
For a sum of five positive terms, its logarithm is at most $\log5$
plus the largest logarithm. Thus
\[
 \log\mathcal P_j
   \le C\left(1+\frac{A_Kx}{(j-1)^4}
                   +\frac{x}{j^3}+\frac{x}{j^5}\right)
   \le D\frac{x}{j^3},
 \qquad
 \frac{\log\mathcal P_j}{\log k_j}\le\frac{2D}{j}.
\]
Choose $D$ using only the fixed constants and $A_K$, before the final
choice of $J$. This is possible, for example, by bounding
$(j/(j-1))^4\le16$ for all $j\ge2$. Terms independent of $x$
are absorbed by increasing the base. The choices
$J\ge200QD/\vartheta$ and $J\ge4C_{\rm out}D$ respectively give
the frequency and remainder bounds in the statement.

The logarithm of the proposed new $K$, divided by $\log k_j$, is
bounded explicitly by
\begin{equation}
 \frac{\log[(K_{j-1}\mathcal P_jk_j/\ell_j)^{C_{\rm out}}]}
             {\log k_j}
 \le C_{\rm out}
       \left(1+\frac{3A_K}{j^2}
                 +\frac{2D}{j}+\frac2{j^{3/2}}\right).
 \label{ipm:expanded-cascade-regularity-recurrence}
\end{equation}
Here the first fraction follows from
\eqref{ipm:expanded-cascade-rounded-logs}, after making $J$ larger
than one fixed absolute constant. Choose $J$ so that the three
fractions sum to at most $1/2$. The right side is then
$3C_{\rm out}/2<A_K$, proving the output bound with the very same
$A_K$ at every stage. A fixed power of the previous $K$ is affordable;
a new regularity exponent is not selected at later stages.

The three summands in \eqref{ipm:cascade-budget-series} are bounded by
fixed constants times, respectively,
\[
 e^{-x/(2j^3)},\qquad e^{-c_0x/2},\qquad e^{-x/(3j^2)}.
\]
The inequalities follow respectively from
$j^{-5}\le(2j^3)^{-1}$, $j^{-5}\le c_0/2$, and the lower
rounded carrier bound. Their sums tend to zero by
\eqref{ipm:expanded-cascade-uniform-series}.
There have been finitely many lower bounds on $J$, all involving
constants fixed earlier. Fix such a $J$. All subsequent requirements
are uniform smallness consequences of increasing the single base $X$,
so one base meets them simultaneously. This proves the claimed
order of quantifiers.
\end{proof}

\section{The stationary base and the first stage}
\label{ipm:sec-base}

With the universal constants and numerical scales fixed, we
construct the initial exact parent. The first layer is then obtained
from this parent by the same induction step as every later layer.

\begin{lemma}
\label{ipm:stationary-base}
For every sufficiently large integer $N$, set
$\nu=(-1,N)$ and $x_{J-1}=1/\arctan(1/N)$, and use the base
levels in \eqref{ipm:cascade-base-bookkeeping}. There is an odd
mean-zero stationary solution $P_{J-1}$ with gradient at the origin
$h_{J-1}\nu/|\nu|$, upper Hessian at most $I/4$, and global
gradient at most $Ch_{J-1}$. Its initial density and explicit
particle map satisfy the four-field class with
$K_{J-1}\le k_{J-1}^{A_K}$. One may take
$t_{J-1}=0$, $S_{J-1}=2W_J\le1$; the first transfer has zero
geometric remainder.
\end{lemma}
\begin{proof}
To realize the base, choose an integer $N$ and put
\begin{equation}
 \nu=(-1,N),\qquad
 x_{J-1}=\frac1{\arctan(1/N)},\qquad n_{\rm b}=\nu/|\nu|.
 \label{ipm:cascade-base-angle}
\end{equation}
The angle of $n_{\rm b}$ is exactly $\pi/2+1/x_{J-1}$, and
$|\nu|\asymp x_{J-1}$ because $\arctan(1/N)\asymp1/N$.
The admissible base sizes tend to infinity as $N$ runs through the
integers. The uniform scale conclusions therefore remain available
with this arithmetic restriction.
Fix a small constant $d_0>0$ such that
$cC_fd_0\le1/4$. With $h=h_{J-1}$ define
\begin{equation}
 \delta_{\rm b}=d_0/h,\qquad b=d_0/|\nu|,\qquad
 P_{J-1}(z)=b f_{\delta_{\rm b}}(\nu\cdot z).
 \label{ipm:cascade-base}
\end{equation}
Its velocity is $b\mathfrak m(\nu)f_{\delta_{\rm b}}(\nu\cdot z)$ and
$\mathfrak m(\nu)\cdot\nu=0$, so this is an exact stationary solution.
Indeed every nonzero Fourier mode of the density is a multiple of
$\nu$, and the degree-zero even symbol has the same value on all
such multiples. Its mean mode is zero. Thus the displayed velocity
is exact, and its dot product with
$b\nu f'_{\delta_{\rm b}}(\nu\cdot z)$ vanishes identically.
It is odd and mean zero, and
\begin{equation}
 \nabla P_{J-1}(0)=h n_{\rm b},\qquad
 H(P_{J-1})=-cb\nu_2\frac{\nu\otimes\nu}{|\nu|^2}
                           f'_{\delta_{\rm b}}(\nu\cdot z)
                  \le\tfrac14 I.
 \label{ipm:cascade-base-hessian}
\end{equation}
By \eqref{ipm:profile-bounds}, its global density gradient satisfies
$\|\nabla P_{J-1}\|_\infty\le b|\nu|C_f^2/\delta_{\rm b}=C_f^2h$.
Its initial-label flow is the explicit shear
\begin{equation}
 X_{J-1}(t,a)=a+t u(P_{J-1})(a),\qquad
 (X_{J-1})_t=u(P_{J-1})(a),\qquad
 (X_{J-1})_{tt}=0.
 \label{ipm:cascade-base-flow}
\end{equation}
Indeed $\nu\cdot X$ is constant along each particle. The rank-one
matrix $D u(P_{J-1})$ squares to zero, whence
$F^{-1}=I-tD u(P_{J-1})$ and $\det F=1$.

We check the entire inherited class, including derivative counts,
rather than only the first few norms. For any ordered word $I$ of
length $n$, differentiating the plane wave multiplies
$f_{\delta_{\rm b}}^{(n)}(\nu\cdot a)$ by a product of $n$ components
of $\nu$. Summing these products in absolute value gives
$(|\nu_1|+|\nu_2|)^n\le(\sqrt2|\nu|)^n$. Consequently
\[
 [P_{J-1}]_n+[u(P_{J-1})]_n
       \le C b\left(\frac{C|\nu|}{\delta_{\rm b}}\right)^n(n!)^2.
\]
The same estimate holds for $X_{J-1}-\operatorname{id}$ on $[0,1]$
and for $(X_{J-1})_t$, while $(X_{J-1})_{tt}=0$.
Taking
$K_{J-1}=C(1+x_{J-1}h/d_0)$ with a sufficiently large fixed $C$
therefore bounds all four fields in exactly the convention of
\eqref{ipm:parent-class}. In particular
\[
 \log K_{J-1}
    \le C+\log x_{J-1}+\frac{x_{J-1}}{(J-1)^7}
    \le A_K\log k_{J-1}
\]
for the already large $J$ and then sufficiently large $N$.
For example the coefficient inequality
$(J-1)^{-7}\le A_K[2(J-1)^4]^{-1}$ holds at all our allowed $J$,
and the remaining $C+\log x_{J-1}$ is absorbed by the other half
as $N\to\infty$. No physical carrier at the bookkeeping scale
$k_{J-1}$ is needed, and no exponential deformation bound for the
flow has been used.

Choose $N$ large enough for all scale requirements and for
$2W_J=4x_{J-1}^2/(ch_{J-1})\le1$.
This is possible since its denominator is
$\exp[x_{J-1}/(J-1)^7]$.
Also $\delta_{\rm b}\le\delta_0$ for such $N$.
Set $t_{J-1}=0$ and $S_{J-1}=2W_J$.
The first transfer has exactly frozen parent geometry and $E=0$.
\end{proof}

\section{Closing the stage induction}
\label{ipm:sec-cascade}

\subsection{Order of stage choices and preservation of the bounds}

\begin{proposition}
\label{ipm:ordered-stage-closure}
Under the scale conditions of Lemma~\ref{ipm:scales}, an exact
parent satisfying the induction bounds determines the carrier,
target time, retained horizon, and seed before any child mean is
selected. The resulting horizon lies inside the parent's interval.
With $C_{\nabla}=4\max(C_{\rm lead},C_f^2,1)$, the conditions
$h_{j-1}/h_j\le1/4$, $\epsilon_j\le1$, and $h_j\ge1$ preserve
the global gradient bound. A base upper Hessian $I/4$ and a total
budget at most $1/2$ preserve $H(P_j)\le3I/4$ at every stage.
\end{proposition}
\begin{proof}
The foregoing estimates can be used with the following strictly ordered
choices at stage $j$.
\begin{enumerate}
\item The numerical levels $x_i,h_i,\ell_i,\delta_i,\widehat k_i,W_i$
are prescribed from $J$ and the base size. The parent $P_{j-1}$ is
already known on $[0,t_{j-1}+2W_j]$.
\item Use its known value $F_{j-1}(t_{j-1},0)$ to choose and round the
initial phase vector. Solve its parent-transported phase ODE and
determine the first target time $t_j$. This calculation uses no
coefficient of the new layer.
\item Set $S_j=t_j+2W_{j+1}$. The next width is a prescribed number,
not a function of the child solution. Restrict the known parent to
$[0,S_j]$, normalize the seed at $t_j$, and choose $\Gamma_j$ using
\eqref{ipm:seventh-cascade-seed-parameter}.
\item Solve the finite triangular hierarchy on this fixed interval,
including its terminally selected means at time $S_j$, and then make
the exact scalar correction. These mean initial data may depend on
the known parent's future, but they do not alter that parent or the
already selected $t_j,S_j$.
\item Apply \eqref{ipm:seventh-cascade-global-budget} and
\eqref{ipm:seventh-cascade-gradient-bounds} to the returned exact
child. Its Hessian bound is an output of this step, not an input to
its own mean inverse.
\end{enumerate}
The inequality $W_{j+1}\le W_j/16$ and the transfer bound
\eqref{ipm:transfer-time} give
\[
 S_j\le t_{j-1}+\tfrac34W_j+\tfrac18W_j
     <t_{j-1}+2W_j=S_{j-1}.
\]
Thus the parent data used in every step exist on the required interval.

Using the fixed choice $C_{\nabla}=4\max(C_{\rm lead},C_f^2,1)$, suppose
$h_{j-1}/h_j\le1/4$, $\epsilon_j\le1$, and $h_j\ge1$.
Then the full global gradient induction is the explicit inequality
\begin{equation}\label{ipm:seventh-cascade-gradient-induction}
 \|\nabla P_j\|_\infty
 \le C_{\nabla}h_{j-1}+C_{\rm lead}h_j+\epsilon_j
 \le\tfrac34C_{\nabla}h_j<C_{\nabla}h_j.
\end{equation}
If the base upper Hessian is at most $I/4$ and the sum of the budget
increments in \eqref{ipm:seventh-cascade-global-budget} is at most
$1/2$, every exact stage has $H(P_j)\le3I/4<I$. The constants
$C_{\rm lead},C_H,C_{\nabla},C_E$ have now all been fixed in a finite order.
The scale lemma may consequently choose one $J$ and one base size to
meet their finitely many tolerances. Estimates in higher fixed
$C^m$ norms are used only later to choose convergence tails; they do
not enter this stage closure.
\end{proof}

\subsection{Construction of the exact stage sequence}

We now perform the induction with the constants already fixed.
The scale choice ensures $h_{j-1}/h_j\le1/4$.
Increase the base once, before constructing any child, so that
\begin{equation}
 \sum_{j\ge J}
    \bigl[C_H\delta_jh_j+C_Hh_je^{-c_0x_{j-1}}
                         +k_j^{-1/2}\bigr]\le\frac12.
 \label{ipm:expanded-cascade-summed-budget}
\end{equation}
The uniform envelopes in the scale lemma prove this for every
allowed lattice rounding, so the choice does not depend on any
unconstructed ray.
Define $b_{J-1}=1/4$ and let $b_j$ be $1/4$ plus this sum through
stage $j$. Then $b_j\le3/4<1$; in the induction
we will retain the sharper inequality $H(P_j)\le b_jI$.
The strict margin ensures that using $H\le I$ in the analytic
proposition never relies on the child bound before it is proved.

\begin{proposition}
\label{ipm:stage-induction}
For the fixed scales and stationary base above, there are exact odd
mean-zero solutions $P_j$ on $[0,S_j]$, $j\ge J$, and target
times $t_j$ such that
\[
 \frac14W_j\le t_j-t_{j-1}\le\frac34W_j,
 \qquad S_j=t_j+2W_{j+1}<S_{j-1}.
\]
Every stage satisfies
\[
 H(P_j)\le b_jI\le\tfrac34I,
 \qquad \|\nabla P_j\|_\infty\le C_{\nabla}h_j,
 \qquad K_j\le k_j^{A_K}.
\]
The $j$th leading density-gradient increment at the origin has magnitude
$h_j$ and angle $\pi/2+1/x_j$ at $t_j$. Its carrier satisfies the integer-lattice
condition. Its initial density contains the oscillatory seed and
the selected phase means; its exact residual correction vanishes
at time zero.
\end{proposition}
\begin{proof}
Inductively the exact odd stage $P_{j-1}$ is defined on
$[0,S_{j-1}]$, where
\begin{equation}
 S_{j-1}=t_{j-1}+2W_j,\qquad
 H(P_{j-1})\le I,\qquad
 \|\nabla P_{j-1}\|_\infty\le C_{\nabla} h_{j-1},\qquad
 K_{j-1}\le k_{j-1}^{A_K}.
 \label{ipm:cascade-induction}
\end{equation}
For $j>J$, the previous leading density-gradient increment at the origin
has magnitude $h_{j-1}$ and angle $\pi/2+1/x_{j-1}$ at $t_{j-1}$.
Although this leading ray is not transported by $P_{j-1}$, it is
transported by the exact older solution $P_{j-2}$, which is available
on the entire current interval. The phase evolution after
\eqref{ipm:intro-phase}, the identity \eqref{ipm:rayidentity}, and the
matrix bounds \eqref{ipm:matrixbounds} imply
\[
 |(\log G_{j-1})'|+|\alpha_{j-1}'|
       \le C(1+h_{j-2}).
\]
The global layer errors and the older density-gradient bound give
\eqref{ipm:transfer-decomposition}--\eqref{ipm:transfer-error} for
$P_{j-1}$, with
\begin{equation}
 E_j\le C(1+h_{j-2})+C\epsilon_{j-1}
            \le C'(1+h_{j-2}).
 \label{ipm:cascade-error-generation}
\end{equation}
Here are the individual terms in this error estimate. With
$O=P_{j-2}$ and the previous leading density-gradient increment at the origin denoted by
$G_{j-1}n_{j-1}$, linearity of the Darcy operator gives
\[
 \begin{aligned}
 r_{P_{j-1}}-G_{j-1}n_{j-1}&=r_O+E_r,\\
 M_{P_{j-1}}+cG_{j-1}(n_{j-1})_1
                    n_{j-1}^\perp\otimes n_{j-1}&=M_O+E_M,\\
 H_{P_{j-1}}+cG_{j-1}(n_{j-1})_2
                    n_{j-1}\otimes n_{j-1}&=H_O+E_H.
 \end{aligned}
\]
The three $E$ terms have norm sum at most $\epsilon_{j-1}$.
Furthermore
$\|r_O\|\le C_{\nabla}h_{j-2}$,
$\|M_O\|\le1+2cC_{\nabla}h_{j-2}$, and
$\|H_O\|\le1+cC_{\nabla}h_{j-2}$.
The derivatives of $\alpha_{j-1}$ and $\log G_{j-1}$ are bounded
by the latter two norms because their ray is transported by $O$.
Adding these estimates proves
$E_j\le C_E(1+h_{j-2})$, since $\epsilon_{j-1}\le1$.
Thus the constant $C_E$ used when choosing the scales depends only
on previously fixed constants.
\label{ipm:expanded-cascade-two-generation-error}

Set $s=t_{j-1}$ and select the child phase. Let
$v$ be the unit vector with angle $1/x_{j-1}$, perpendicular to the
previous leading density-gradient increment in the desired orientation. The ideal unit
initial vector is
\[
 m_0^*=\frac{F_{j-1}(s,0)^Tv}{|F_{j-1}(s,0)^Tv|}.
\]
Round $(\widehat k_j/\ell_j)m_0^*$ to a nearest integer vector
$\nu_j\in\mathbb Z^2$, and define
\begin{equation}
 k_j=\ell_j|\nu_j|,\qquad m_{0,j}=\nu_j/|\nu_j|.
 \label{ipm:cascade-rounding}
\end{equation}
Then $|k_j-\widehat k_j|\le C\ell_j$ and
$|m_{0,j}-m_0^*|\le C\ell_j/\widehat k_j$.
To check the loss under $F^{-T}$, use
$|v/|v|-w/|w||\le2|v-w|/|w|$ for nonzero $v,w$,
with $v=F^{-T}m_{0,j}$ and $w=F^{-T}m_0^*$.
The denominator is at least $\|F\|^{-1}$.
The resulting unit-vector difference is at most
$2\|F^{-1}\|\|F\||m_{0,j}-m_0^*|$.
For differences below $1/2$, the angular difference is at most
twice the unit-vector difference. Both matrix norms have fixed
polynomial bounds in $K_{j-1}$, proving
\begin{equation}
 |\tau_j|\le CK_{j-1}^A\ell_j/\widehat k_j,\qquad
 |\tau_j|j^2x_{j-1}\ll1.
 \label{ipm:cascade-rounding-error}
\end{equation}
The logarithm of the right side in the second assertion is bounded
above by
$C+\log(j^2x_{j-1})+A\log K_{j-1}
 -x_{j-1}/j^{7/2}-x_{j-1}/j^2$.
It tends to a large negative value uniformly with the base by
\eqref{ipm:expanded-cascade-rounded-logs} and
\eqref{ipm:expanded-cascade-polynomial-uniformity}.
Thus it satisfies the fixed transfer tolerance.
This establishes the exact periodicity condition
$(k_j/\ell_j)m_{0,j}\in\mathbb Z^2$ with an admissible angular error.

The parent is already known, so Lemma~\ref{ipm:amplification} now
determines a first target time $t_j$ satisfying \eqref{ipm:transfer-time}: 
\begin{equation}
 \beta_j(t_j)=\pi/2+1/x_j,\qquad
 W_j/4\le t_j-t_{j-1}\le3W_j/4.
 \label{ipm:cascade-target-time}
\end{equation}
Fix the new horizon
\begin{equation}
 S_j=t_j+2W_{j+1}.
 \label{ipm:cascade-new-horizon}
\end{equation}
By \eqref{ipm:cascade-scale-horizons},
$S_j\le t_{j-1}+(3/4+1/8)W_j<S_{j-1}$. Normalize the seed by
\eqref{ipm:cascade-target-normalization}, with target $h_j$, and choose
$\Gamma_j$ by \eqref{ipm:cascade-gamma-choice}.
The inherited gradient bound supplies the future estimates required
for that choice. All the hypotheses of Proposition~\ref{ipm:packet} now
hold, by Lemma~\ref{ipm:scales}. Apply it on the already fixed interval
$[0,S_j]$ to obtain $P_j$.

In particular the terminal condition used to select the layer's phase
means is imposed only after $S_j$ is known. Neither $t_j$ nor $S_j$
depends on the child mean being solved: both use only the existing
parent and the prescribed next width. This fixes the dependency order
without a shooting or fixed-point problem.

Combining the leading pressure-Hessian bound \eqref{ipm:cascade-h-budget} with the
pressure-Hessian remainder in \eqref{ipm:packet-gradient-errors} gives
\[
 H(P_j)\le H(P_{j-1})+
 \bigl[C_H(\delta_jh_j+h_je^{-c_0x_{j-1}})+\epsilon_j\bigr]I
 \le b_jI\le(3/4)I<I,
\]
by \eqref{ipm:expanded-cascade-summed-budget}.
The global gradient bound $\|\nabla P_j\|_\infty\le C_{\nabla}h_j$
follows from \eqref{ipm:seventh-cascade-gradient-induction}.
The layer output estimate \eqref{ipm:packet-output-class}, with
$C_{\rm out}\ge C_*$, and the second inequality in
\eqref{ipm:cascade-scale-output} give
\[
 K_j\le(K_{j-1}\mathcal P_jk_j/\ell_j)^{C_{\rm out}}
     \le k_j^{A_K}.
\]
The prescribed leading normalization \eqref{ipm:cascade-target-normalization}
and target angle at the origin \eqref{ipm:cascade-target-time} also hold,
so the induction is complete.
\end{proof}

\section{The accumulation time and the limiting smooth initial datum}
\label{ipm:sec-initial-limit}

\subsection{The accumulation time}

\begin{lemma}
\label{ipm:common-time}
The times from Proposition~\ref{ipm:stage-induction} have a limit
$T\in(0,1)$ and satisfy
\[
 0<T-t_j\le\tfrac45W_{j+1}<2W_{j+1},\qquad T<S_j.
\]
Every finite stage is therefore smooth on a neighborhood of $[0,T]$.
\end{lemma}
\begin{proof}
The target times are strictly increasing and their increments are
comparable to $W_j$. Since $W_{j+1}\le W_j/16$, their limit
\[
 T=\lim_{j\to\infty}t_j
\]
is finite and positive. More precisely,
\[
 0<T-t_j\le\frac34\sum_{i>j}W_i
      \le\frac{3/4}{1-1/16}W_{j+1}=\frac45W_{j+1}<2W_{j+1},\qquad T<S_j\le1.
\]
The final inequality is strict. Thus every finite stage is smooth on a
neighborhood of $[0,T]$.
\end{proof}

\subsection{The limiting smooth initial datum}

\begin{proposition}
\label{ipm:initial-data-convergence}
For every fixed integer $m\ge0$, the increments
$\theta_j=P_j-P_{j-1}$ satisfy
\[
 \sum_{j\ge J}\sup_{0\le t\le t_{j-1}}
                   \|\theta_j(t)\|_{C^m}<\infty.
\]
Consequently $P_j(0)$ converges in $C^\infty$ to a single odd mean-zero
density $\rho_{\rm in}$. The stage sequence and its initial datum are
independent of the derivative order $m$.
\end{proposition}
\begin{proof}
Set $\theta_j=P_j-P_{j-1}$. For $t\le t_{j-1}$,
\eqref{ipm:cascade-gamma-history} and
\eqref{ipm:cascade-osc-bound} give, at each fixed derivative order $m$,
\begin{align}
 \|\theta_{j,\rm osc}(t)\|_{C^m}
    &\le C_m\ell_j^{1-m}k_j^{m-1}\mathcal P_j^{C_m}
                     h_j e^{-c_0x_{j-1}}
       \le C_m e^{-c_0x_{j-1}/2},
       \label{ipm:cascade-history-cm-osc}\\
 \|\theta_{j,\rm mean}(t)\|_{C^m}
    &\le C_m\ell_j^{1-m}k_j^{-2}\mathcal P_j^{C_m}
       \le C_m e^{-x_{j-1}/j^2},
       \label{ipm:cascade-history-cm-mean}
\end{align}
for all sufficiently large $j$ depending on $m$. We have absorbed the
fixed power of $K$ and the factor $\delta_j\le1$ into the displayed
$\mathcal P_j$ power. To verify these inequalities, take logarithms:
apart from $\log C_m$, their respective upper bounds are
\[
 -c_0x_{j-1}+\frac{2m x_{j-1}}{j^2}
                   +\frac{D_mx_{j-1}}{j^3},\qquad
 -\frac{9x_{j-1}}{5j^2}+\frac{D_mx_{j-1}}{j^3}.
\]
Here $D_m$ is finite, independent of the stage, and may depend on
$m$ and the already fixed constants. These bounds follow by inserting
the rounded carrier inequalities, $\log\ell_j^{-1}=x_{j-1}/j^{7/2}$,
$\log\mathcal P_j\le D x_{j-1}/j^3$, and
$\log h_j=x_{j-1}/j^5$.
For $m=0$ any negative carrier or localization exponents can be
dropped when forming an upper bound.
Choose a tail index $j_m\ge J$ such that
$2m/j_m^2\le c_0/4$,
$D_m/j_m^3\le c_0/4$, and $D_m/j_m\le1/2$.
The same inequalities then hold for all $j\ge j_m$ and imply
\eqref{ipm:cascade-history-cm-osc} and
\eqref{ipm:cascade-history-cm-mean}.

For the exact correction, its logarithmic prefactor is bounded by
$D'_m(\log k_j+\log\ell_j^{-1}+\log\mathcal P_j)$.
This grows only linearly in $\log k_j$ for fixed $m$, while
$k_j^\vartheta=\exp(\vartheta\log k_j)$.
Increase $j_m$, if necessary, so that this prefactor is at most
$k_j^\vartheta/4$. The correction is then bounded by
$C_m\exp(-k_j^\vartheta/4)$, whose sum converges. For example
$k_j^\vartheta/4\ge x_{j-1}/j^2$ eventually, by the lower carrier
bound and elementary exponential growth, and
\eqref{ipm:expanded-cascade-doubling} applies.
The two other right sides in
\eqref{ipm:cascade-history-cm-osc}--\eqref{ipm:cascade-history-cm-mean}
are summable by the same doubling argument. The finitely many terms
with $J\le j<j_m$ are smooth on their respective compact intervals.
Consequently
\begin{equation}
 \sum_{j\ge J}\sup_{0\le t\le t_{j-1}}
                      \|\theta_j(t)\|_{C^m}<\infty
 \qquad\text{for every fixed }m.
 \label{ipm:cascade-full-history-convergence}
\end{equation}
In particular the initial data converge in $C^\infty$ to an odd
mean-zero function $\rho_{\rm in}$. This uses the actual small seed
$\gamma_j(0)$; a bound by the maximum amplitude up to the target would
not suffice.
The constants $C_m,D_m,j_m$ are chosen only in this convergence
argument, after the entire sequence has been constructed. Thus
convergence in all fixed $C^m$ norms is asserted for a single
initial datum. No scale or base choice depends on infinitely many
derivative-order constants.
\label{ipm:expanded-cascade-fixed-order-tails}
\end{proof}

\section{The smooth limiting solution and finite-time blow-up}
\label{ipm:sec-smooth-limit}

\subsection{The smooth limiting solution}

\begin{proposition}
\label{ipm:smooth-convergence}
On every compact subinterval of $[0,T)$, the stages converge with all
mixed space-time derivatives to an exact unforced solution
$\rho\in C^\infty([0,T)\times\mathbb T^2)$ with
$\rho(0)=\rho_{\rm in}$, where $\rho_{\rm in}$ is the datum of
Proposition~\ref{ipm:initial-data-convergence}.
\end{proposition}
\begin{proof}
Fix $t_0<T$. Eventually $t_{j-1}>t_0$, so
\eqref{ipm:cascade-full-history-convergence} shows that $P_j$ converges
on $[0,t_0]$ in every spatial $C^m$ norm. Denote the limit by $\rho$.
Here is the passage through time derivatives in fixed Banach spaces.
For every integer $s\ge2$, spatial convergence gives
$P_j\to\rho$ in $C([0,t_0];H^s)$.
The bounded Fourier symbol gives
$\|u(f)\|_{H^s}\le c\|f\|_{H^s}$.
Since $H^s(\mathbb T^2)$ is an algebra for $s\ge2$,
the bilinear map
\[
 (f,g)\longmapsto u(f)\cdot\nabla g
\]
is continuous from $H^{s+1}\times H^{s+1}$ to $H^s$.
Therefore the exact stage equations imply
$\partial_tP_j\to-u(\rho)\cdot\nabla\rho$
in $C([0,t_0];H^s)$ for every such $s$.
Passing to the integral equation
$P_j(t)=P_j(0)+\int_0^t\partial_tP_j(\sigma)\,d\sigma$
proves that $\rho$ has this derivative and solves
\eqref{ipm:darcy}--\eqref{ipm:unforced}.

For the induction step, suppose convergence of
$\partial_t^aP_j$ has been proved in
$C([0,t_0];H^s)$ for all $a\le q$ and every fixed $s\ge2$.
Differentiate the classical stage equation $q$ times:
\begin{equation}
 \partial_t^{q+1}P_j
  =-\sum_{a=0}^q\binom qa
      u(\partial_t^aP_j)\cdot\nabla\partial_t^{q-a}P_j.
 \label{ipm:expanded-cascade-time-recursion}
\end{equation}
The induction hypothesis at $s+1$ and the bilinear continuity just
proved make the right side converge in $C([0,t_0];H^s)$.
The fundamental theorem of calculus applied to
$\partial_t^qP_j$ then shows that its limit is differentiable in
$H^s$ with this limiting derivative. This proves the induction
for $q+1$, simultaneously for every fixed $s$.
The velocity commutes with time differentiation because it is a
bounded linear map on each of these spaces.
Finally Sobolev embedding at arbitrarily large $s$ gives convergence
and continuity of every mixed space-time derivative.
Thus
\begin{equation}
 \rho\in C^\infty([0,T)\times\mathbb T^2),\qquad
 \rho(0)=\rho_{\rm in},
 \label{ipm:cascade-smooth-limit}
\end{equation}
and $\rho$ is an exact odd mean-zero unforced solution on the whole
interval $[0,T)$. No stability assertion across the accumulation time
is required.
\end{proof}

\subsection{Finite-time blow-up}

\begin{proposition}
\label{ipm:component-divergence}
The solution of Proposition~\ref{ipm:smooth-convergence} satisfies
\[
 \lim_{t\uparrow T}\partial_2\rho(t,0)
 =\lim_{t\uparrow T}\partial_1u_1(\rho)(t,0)=+\infty.
\]
More precisely, there are fixed $a_1,a_2>0$ and a sequence
$D_j\to0$ such that, for every $t\in[t_j,T)$,
\begin{equation}
 \partial_2\rho(t,0)\ge a_1h_j-2/c-D_j,\qquad
 \partial_1u_1(\rho)(t,0)\ge a_2h_j/x_j^2-1-D_j.
 \label{ipm:cascade-limit-lower}
\end{equation}
In particular, the solution has no
smooth continuation through $T$.
\end{proposition}
\begin{proof}
Let $g_j$ and $\beta_j$ denote the magnitude and angle of the leading
density-gradient increment of layer $j$ at the origin. On its active and retained interval
$[t_{j-1},S_j]$ its phase has positive second component. Its two useful
leading contributions are
\begin{equation}
 (\partial_2\theta_j)_{\rm lead}=g_j\sin\beta_j\ge0,\qquad
 (\partial_1u_1(\theta_j))_{\rm lead}
          =c g_j\cos^2\beta_j\sin\beta_j\ge0.
 \label{ipm:cascade-positive-components}
\end{equation}
The square of the first phase component in the second formula is
essential. These signs hold even while a layer crosses the vertical
direction.
To derive the second component without an orientation convention,
the leading velocity-gradient increment is
$c g_j((n_j)_1(n_j)_2,-(n_j)_1^2)\otimes n_j$ by
\eqref{ipm:velocity-symbol}; its $(1,1)$ entry is exactly
$c g_j(n_j)_1^2(n_j)_2$.
Oddness of the parent flow sends label zero to spatial zero, so the
phase in these formulas is exactly zero at the observation point.
The factor $f'_{\delta_j}(0)/\delta_j^{-1}$ is one, and no
oscillatory sign has been discarded.

At the target $g_j(t_j)=h_j$. By
\eqref{ipm:rayidentity} and $H(P_{j-1})\le I$,
\begin{equation}
 g_j(t)\ge e^{-(t-t_j)}h_j\ge e^{-1}h_j,\qquad
 t_j\le t\le S_j.
 \label{ipm:cascade-persistent-amplitude}
\end{equation}
The parent gradient bound and the matrix bounds \eqref{ipm:matrixbounds} yield
\[
 |\beta_j(t)-\beta_j(t_j)|
       \le C(1+h_{j-1})W_{j+1}\le\frac1{4x_j}.
\]
Thus, throughout $[t_j,S_j]$, its angle lies between
$\pi/2+3/(4x_j)$ and $\pi/2+5/(4x_j)$. For the chosen large base,
there are fixed $a_1,a_2>0$ such that
\begin{equation}
 (\partial_2\theta_j)_{\rm lead}\ge a_1h_j,\qquad
 (\partial_1u_1(\theta_j))_{\rm lead}\ge a_2h_j/x_j^2.
 \label{ipm:cascade-persistent-components}
\end{equation}
For example $\sin\beta_j\ge1/2$ on this interval.
Since $|\cos\beta_j|=\sin(\beta_j-\pi/2)$ and
$0<\beta_j-\pi/2\le1$ for the chosen base,
$\sin v\ge v/2$ gives $|\cos\beta_j|\ge3/(8x_j)$.
Together with $g_j\ge e^{-1}h_j$, these yield the displayed
bounds with $a_1=1/(2e)$ and $a_2=9c/(128e)$, or with
any smaller fixed constants.

For times before activation, \eqref{ipm:cascade-small-history} and the
global expansions instead give
\begin{equation}
 \|\nabla\theta_j(t)\|_\infty
       +\|\nabla u(\theta_j)(t)\|_\infty
       \le d_j+C\epsilon_j,\qquad
 d_j=Ch_je^{-c_0x_{j-1}},\qquad t\le t_{j-1}.
 \label{ipm:cascade-preactivation-error}
\end{equation}
Both $\sum d_j$ and $\sum\epsilon_j$ converge. Define
$D_j=C\sum_{i\ge j}(d_i+\epsilon_i)$, so $D_j\to0$.

Every earlier exact parent satisfies the pointwise bounds
\begin{equation}
 \partial_2P_{j-1}=-\operatorname{tr}H(P_{j-1})/c\ge-2/c,\qquad
 \partial_1u_1(P_{j-1})=-H_{11}(P_{j-1})\ge-1.
 \label{ipm:cascade-parent-lower}
\end{equation}
Fix $j$ and any $t\in[t_j,T)$. Telescope from $P_{j-1}$ through finitely
many later stages. Layer $j$ has the lower bounds
\eqref{ipm:cascade-persistent-components}; every later active layer
has nonnegative leading contributions by
\eqref{ipm:cascade-positive-components}; every later layer in history
costs at most \eqref{ipm:cascade-preactivation-error}. Sum the global
remainder bounds and then pass to the smooth limit at this fixed time.
In detail, for every finite $n\ge j$, $P_n(t)$ and every increment
in $P_n=P_{j-1}+\sum_{i=j}^n\theta_i$ are defined because
$t<T<S_n$ and the horizons decrease. If $i>j$ and $t\ge t_{i-1}$,
the leading contribution is nonnegative and its error is at most
$\epsilon_i$. If $t<t_{i-1}$, its entire contribution has lower
bound $-d_i-C\epsilon_i$. These two alternatives exhaust all later
layers. The sum of their lower errors is bounded by $D_j$,
independently of both $n$ and $t$. At the fixed time $t<T$, the
smooth convergence already proved allows $n\to\infty$ in these
two derivative components, including the velocity component.
This proves \eqref{ipm:cascade-limit-lower} uniformly for all
$t\in[t_j,T)$.
Both $h_j$ and $h_j/x_j^2$ tend to infinity, since
$\log h_j=x_{j-1}/j^5$ dominates every fixed multiple of $\log x_j$.
For the latter assertion, insert $x_j=j^2x_{j-1}$ and apply
\eqref{ipm:expanded-cascade-polynomial-uniformity}, or its
fixed-base asymptotic as $j\to\infty$, to
$(2\log x_{j-1}+4\log j)/(x_{j-1}/j^5)$.
The ratio tends to zero because of the factorial growth.
Given any finite lower threshold, choose $j$ large in
\eqref{ipm:cascade-limit-lower}. The threshold then holds at every
subsequent time below $T$, not merely at the target times. Therefore
\begin{equation}
 \lim_{t\uparrow T}\partial_2\rho(t,0)
   =\lim_{t\uparrow T}\partial_1u_1(\rho)(t,0)=+\infty.
 \label{ipm:cascade-component-limits}
\end{equation}
In particular both $L^\infty$ gradient norms tend to infinity. Combined
with \eqref{ipm:cascade-smooth-limit}, this excludes any smooth
continuation through $T$.
\end{proof}

\section{Uniqueness and the terminal limit}
\label{ipm:seventh-physical}

The component limits fix the maximal smooth lifespan. We complete
the argument by proving smooth uniqueness and convergence at the terminal
time for the actual physical solution.

\subsection{Smooth uniqueness and maximal smooth lifespan of the limit}

\begin{lemma}
\label{ipm:seventh-physical-uniqueness}
Two smooth solutions of the normalized unforced IPM equation on
$[0,\tau]\times\mathbb T^2$ with identical initial density agree on that
interval. More quantitatively, if they are $\rho,\sigma$ and
$w=\rho-\sigma$, then
\begin{equation}\label{ipm:seventh-physical-stability}
 \|w(t)\|_2^2
 \le\|w(0)\|_2^2
 \exp\left(2c\int_0^t\|\nabla\sigma(s)\|_\infty\,ds\right).
\end{equation}
\end{lemma}
\begin{proof}
Linearity of the Darcy map gives the exact difference equation
\[
 w_t+u(\rho)\cdot\nabla w+u(w)\cdot\nabla\sigma=0.
\]
Pair with $w$ in $L^2$. The first transport term integrates to zero by
periodicity and $\operatorname{div}u(\rho)=0$. Thus
\[
 \frac12\frac d{dt}\|w\|_2^2
 =-\int_{\mathbb T^2}w\,u(w)\cdot\nabla\sigma
 \le c\|\nabla\sigma\|_\infty\|w\|_2^2
\]
because the Darcy symbol in \eqref{ipm:velocity-symbol} satisfies
$|\mathfrak m|\le c$. Integrating this scalar differential
inequality proves \eqref{ipm:seventh-physical-stability}. On a compact
smooth interval its exponent is finite. Zero initial difference
therefore implies zero difference throughout.
\end{proof}

\begin{remark}
This also verifies precisely the initial-data issue in removing a
force. If a correction erased the entire perturbation residual while
preserving the parent's initial density, its exact smooth output would
equal the parent by the lemma. Our initial density includes the primary
oscillatory seed and the selected phase means before the zero-initial
residual correction is applied. That distinction is an exact
initial-data distinction; no failure of uniqueness is being used.
\end{remark}

\begin{corollary}
\label{ipm:maximal-lifespan}
The solution of Proposition~\ref{ipm:smooth-convergence} is the
unique smooth solution with initial density $\rho_{\rm in}$ on
every compact subinterval of $[0,T)$. Its maximal smooth lifespan
is $T$.
\end{corollary}
\begin{proof}
The smooth limit in \eqref{ipm:cascade-smooth-limit} is consequently the
unique smooth solution with its limiting initial density on every
compact subinterval of $[0,T)$. Its maximal smooth lifespan is $T$:
the construction supplies existence up to every smaller time, whereas
any smooth extension to a time greater than $T$ would coincide with it
below $T$ and have bounded first derivatives near $T$, contradicting
\eqref{ipm:cascade-component-limits}. This last argument uses the
proved component limits, without an additional continuation criterion.
\end{proof}

\subsection{H\"older continuity of the limit}

\begin{proposition}
\label{ipm:terminal-holder}
For each $0\le\eta<1$, the stage sequence converges in
$C([0,T];C^\eta(\mathbb T^2))$. These limits agree with the
smooth solution on $[0,T)$ and have one common terminal value
$\rho_*\in\bigcap_{0\le\eta<1}C^\eta(\mathbb T^2)$.
The terminal value is odd and mean zero, and
$\rho(t)\to\rho_*$ in every such norm as $t\uparrow T$.
\end{proposition}
\begin{proof}
The amplitude choice \eqref{ipm:cascade-gamma-choice} and the all-time
layer bounds give, uniformly for $0\le t\le T$,
\begin{equation}
 \|\theta_j(t)\|_{C^0}\le C\ell_jk_j^{-1}\mathcal P_j^C,\qquad
 \|\theta_j(t)\|_{C^1}\le\mathcal P_j^C.
 \label{ipm:cascade-alltime-low}
\end{equation}
Here the exact correction is smaller than the displayed quantities for
large $j$; finitely many initial stages cause no problem. For the second
bound one can either use the layer estimate with $m=1$, or the global
gradient estimate and
$h_j=\gamma_j(t_j)|m_j(t_j,0)|/\delta_j\le\mathcal P_j^C$.
For the first bound, the oscillatory term has $m=0$ in
\eqref{ipm:cascade-osc-bound} and $\gamma_j\le\Gamma_j\le\mathcal P_j$.
The mean term is bounded by
$C\ell_jk_j^{-2}\mathcal P_j^C\le
C\ell_jk_j^{-1}\mathcal P_j^C$.
Even using the weakened correction bound, its ratio to
$\ell_jk_j^{-1}\mathcal P_j^C$ is at most a fixed polynomial in
$k_j/\ell_j,\mathcal P_j$ times $e^{-k_j^\vartheta/2}$, which
tends to zero. For $m=1$, the same reasoning gives the second
bound. Thus only the fixed constants for derivative orders zero
and one occur in \eqref{ipm:cascade-alltime-low}; their values do
not have to be uniform over all derivative orders.

For $0<\eta<1$, the elementary interpolation inequality
\begin{equation}
 \|f\|_{C^\eta}
      \le C_\eta\|f\|_{C^0}^{1-\eta}\|f\|_{C^1}^{\eta}
 \label{ipm:cascade-interpolation}
\end{equation}
follows by bounding a difference quotient with the minimum of
$2\|f\|_\infty$ and $\|\nabla f\|_\infty|x-y|$, and splitting at
$|x-y|=2\|f\|_\infty/\|\nabla f\|_\infty$.
If the gradient norm is zero, the function is constant and the
inequality is immediate, with the full $C^1$ norm understood to
include the supremum norm. The same supremum term bounds the
zeroth-order part of the $C^\eta$ norm.
Consequently
\begin{equation}
 \sup_{0\le t\le T}\|\theta_j(t)\|_{C^\eta}
       \le C_\eta k_j^{-(1-\eta)}\mathcal P_j^{C_\eta}
       \le C_\eta k_j^{-(1-\eta)/2}
 \label{ipm:cascade-holder-sum}
\end{equation}
for all sufficiently large $j$ depending on $\eta$. We used
$\ell_j^{1-\eta}\le1$ and
$\log\mathcal P_j/\log k_j\to0$. The right side is summable.
Explicitly, since $\log\mathcal P_j/\log k_j\le2D/j$, take
$j_\eta\ge4DC_\eta/(1-\eta)$, after including any fixed
prefactor in the harmless outside constant. Then the exponent
from $\mathcal P_j^{C_\eta}$ is at most
$(1-\eta)\log k_j/2$. For this fixed $\eta$,
\eqref{ipm:expanded-cascade-doubling} with $q=2$ gives summability
of the remaining carrier powers. The choice of $j_\eta$ does not
alter the construction when $\eta$ approaches one.
The case $\eta=0$ follows directly from
\eqref{ipm:cascade-alltime-low}.

Each finite stage is smooth past $T$ by
Lemma~\ref{ipm:common-time}. The telescoping density series therefore
converges in $C([0,T];C^\eta)$ for every $0\le\eta<1$. Its restriction
to $[0,T)$ is the previously constructed smooth solution. In particular
there is a terminal density, common to all these spaces, such that
\begin{equation}
 \rho_*\in\bigcap_{0\le\eta<1}C^\eta(\mathbb T^2),\qquad
 \rho(t)\longrightarrow\rho_*\quad\hbox{in }C^\eta
                         \quad(t\uparrow T).
 \label{ipm:cascade-terminal-density}
\end{equation}
The terminal densities agree because their $C^0$ limits agree. Oddness
and zero mean pass to this limit.
For completeness, the space $C^\eta(\mathbb T^2)$ is Banach in the
usual supremum-plus-H\"older-seminorm norm. A series whose suprema
of these norms are summable is uniformly Cauchy in
$C([0,T];C^\eta)$, and its uniform limit is continuous. This proves
the continuity assertion at $T$ directly. It requires neither a
uniform $C^1$ bound on the total solution nor an equation for a
smooth extension past $T$; the component limits rule out such an
extension.
\end{proof}

\begin{theorem}
\label{ipm:cascade}
The finite-stage construction of Proposition~\ref{ipm:packet} yields an odd,
mean-zero initial density $\rho_{\rm in}\in C^\infty(\mathbb T^2)$ and an
exact unforced IPM solution smooth on $[0,1)\times\mathbb T^2$ such that
\[
 \lim_{t\uparrow1}\partial_2\rho(t,0)
   =\lim_{t\uparrow1}\partial_1u_1(\rho)(t,0)=+\infty.
\]
It has a terminal density in every $C^\eta$, $0\le\eta<1$, with
convergence in each such norm.
\end{theorem}

\begin{proof}
Proposition~\ref{ipm:stage-induction} constructs the stages, and
Lemma~\ref{ipm:common-time} supplies their accumulation time.
Propositions~\ref{ipm:smooth-convergence},
\ref{ipm:component-divergence}, and \ref{ipm:terminal-holder}
give the asserted solution and limits with terminal time $T>0$.
Define $\widetilde\rho(t,x)=T\rho(Tt,x)$. Since the Darcy velocity is
linear in its density,
\[
 \partial_t\widetilde\rho
       +u(\widetilde\rho)\cdot\nabla\widetilde\rho
       =T^2\bigl(\partial_t\rho+u(\rho)\cdot\nabla\rho\bigr)(Tt,x)=0.
\]
This moves the singular time to one. Multiplication by the positive
constant $T$ preserves both component limits, smoothness of the initial
density, oddness, zero mean, and every stated terminal convergence.
\end{proof}

\appendix
\section{Uniform operator estimates}
\label{ipm:seventh-appendices}

This appendix gives the uniform operator estimates used in the main argument.

\subsection{Uniform Sobolev constants on rectangular tori}
\label{ipm:seventh-foundations-section}

The following estimates are uniform as the torus periods increase.

\begin{lemma}
\label{ipm:seventh-foundations-sobolev}
Let
$\mathbb T_L^d=\prod_{i=1}^d\mathbb R/(L_i\mathbb Z)$, with
$L_i\ge L_0>0$, and use unnormalized Lebesgue measure.
For fixed $s>d/2$ the $H^s$ algebra constant is bounded independently
of the $L_i$. For $s>d/2+1$ the same is true of the embedding
$H^s\longrightarrow W^{1,\infty}$.
For an integer $s\ge0$, multiplication from
$W^{s,\infty}\times H^s$ to $H^s$ also has a uniform constant.
\end{lemma}
\begin{proof}
Put $V=\prod_iL_i$, and normalize Fourier coefficients by $V^{-1}$.
The frequency lattice is
$\Xi_L=\prod_i(2\pi/L_i)\mathbb Z$, and
\[
 \|f\|_{H^s}^2
   =V\sum_{\xi\in\Xi_L}\langle\xi\rangle^{2s}|\widehat f(\xi)|^2.
\]
For every $a>d/2$,
\begin{equation}\label{ipm:seventh-foundations-lattice-sum}
 V^{-1}\sum_{\xi\in\Xi_L}\langle\xi\rangle^{-2a}
       \le C_{a,d,L_0}.
\end{equation}
Indeed the rectangular cell centered at $\xi$ has volume
$(2\pi)^d/V$, and its diameter is bounded by a number depending only
on $d,L_0$. Throughout that cell $\langle x\rangle$ and
$\langle\xi\rangle$ are comparable with a uniform constant.
Multiply the sum by its cell volume and compare it with
$\int_{\mathbb R^d}\langle x\rangle^{-2a}\,dx$.

Cauchy--Schwarz and
\eqref{ipm:seventh-foundations-lattice-sum} give
\[
 \sum_\xi|\widehat f(\xi)|\le C_{s,d,L_0}^{1/2}\|f\|_{H^s}
 \quad(s>d/2).
\]
Applying the same calculation to the first derivatives gives the
$W^{1,\infty}$ estimate when $s-1>d/2$.
For the algebra estimate use $\langle\xi+\zeta\rangle
\le\langle\xi\rangle+\langle\zeta\rangle$ and
\[
 \langle\xi+\zeta\rangle^s
   \le2^{\max(s-1,0)}(\langle\xi\rangle^s+\langle\zeta\rangle^s)
\]
(the factor is $2^{s-1}$ for $s\ge1$)
in the convolution formula for $\widehat{fg}$ and then the
$\ell^2*\ell^1\to\ell^2$ inequality on the counting lattice.
Multiplying by $V^{1/2}$ and using the preceding Fourier
$\ell^1$ bound yields
$\|fg\|_{H^s}\le C_s\|f\|_{H^s}\|g\|_{H^s}$.
All factors of $V$ cancel. Finally, for the multiplier estimate,
expand only the fixed number $s$ of derivatives by Leibniz's rule,
put the coefficient derivatives in $L^\infty$, and put the function
derivatives in $L^2$. Equivalence between this derivative norm and
the displayed Fourier norm depends only on $s,d$.
\end{proof}

\subsection{Time dependence of a projected inverse}

\begin{lemma}
\label{ipm:time-projected-inverse}
Let $\Pi$ be a fixed orthogonal projection on
$L^2(\mathbb T_L^d;\mathbb R^r)$ commuting with translations, and set
$\mathcal H=\operatorname{ran}\Pi$. Fix an integer $s\ge0$.
Let $A(t,x)$ be a smooth symmetric matrix field in $x$, continuous
in $t$ with values in $W^{s,\infty}$, and suppose
$A(t,x)\ge\lambda I$ for a fixed $\lambda>0$.
Define $T_{A(t)}v=\Pi(A(t)v)$ on $\mathcal H$.
Then $T_{A(t)}$ is invertible on $\mathcal H\cap H^s$, and
$t\mapsto T_{A(t)}^{-1}$ is continuous in operator norm there.
If $A\in C_t^1W^{s,\infty}$ and $h\in C_t^1H^s$, then
$v=T_A^{-1}\Pi h\in C_t^1H^s$ satisfies
\begin{equation}\label{ipm:seventh-foundations-time-inverse}
 v_t=T_A^{-1}\Pi(h_t-A_tv).
\end{equation}
\end{lemma}
\begin{proof}
Coercivity and self-adjointness give an $L^2$ inverse of norm at
most $\lambda^{-1}$. Commuting spatial difference quotients with
$\Pi$, as in the proof of Proposition~\ref{ipm:meaninverse}, gives
bounded inversion on $\mathcal H\cap H^s$, locally uniformly in time.
The resolvent identity
\[
 T_{A(t)}^{-1}-T_{A(u)}^{-1}
   =T_{A(t)}^{-1}\Pi(A(u)-A(t))T_{A(u)}^{-1}
\]
holds on the projected $H^s$ space and proves operator continuity.
For differentiable $A,h$, taking a difference quotient gives
\eqref{ipm:seventh-foundations-time-inverse}.
\end{proof}

\section*{Acknowledgments}
The author is supported by the Ram\'on y Cajal grant RYC2022-035363-I.


\begin{thebibliography}{99}
\small
\raggedright
\setlength{\itemsep}{0pt}
\bibitem{lit:ABB} L. Alp\"oge and T. Buckmaster, \emph{Blowup for the Boussinesq equations with smooth forcing}, manuscript (2026). \url{https://cims.nyu.edu/~tristanb/boussinesq.pdf}.
\bibitem{lit:ABE} L. Alp\"oge and T. Buckmaster, \emph{Blowup for the Euler equations with smooth forcing}, manuscript (2026). \url{https://cims.nyu.edu/~tristanb/euler.pdf}.
\bibitem{lit:ABC} L. Alp\"oge, T. Buckmaster, and M. P. Coiculescu, \emph{Extending the C\'ordoba--Mart\'inez-Zoroa IPM blow-up to uniformly space-time smooth forcing}, preprint (2026). \url{https://arxiv.org/abs/2609.16470}.
\bibitem{lit:BKM} J. T. Beale, T. Kato, and A. Majda, \emph{Remarks on the breakdown of smooth solutions for the 3-D Euler equations}, Comm. Math. Phys. \textbf{94} (1984), 61--66. \url{https://doi.org/10.1007/BF01212349}.
\bibitem{lit:BCM} R. Bianchini, D. C\'ordoba, and L. Mart\'inez-Zoroa, \emph{Non existence and strong ill-posedness in $H^2$ for the stable IPM equation}, preprint (2024). \url{https://arxiv.org/abs/2410.01297}.
\bibitem{lit:BCP} R. Bianchini, T. Crin-Barat, and M. Paicu, \emph{Relaxation approximation and asymptotic stability of stratified solutions to the IPM equation}, Arch. Ration. Mech. Anal. \textbf{248} (2024), 2. \url{https://arxiv.org/abs/2210.02118}.
\bibitem{lit:BJPW} R. Bianchini, M. J. Jo, J. Park, and S. Wang, \emph{Sharp asymptotic stability of the incompressible porous media equation}, preprint (2025). \url{https://arxiv.org/abs/2505.05165}.
\bibitem{lit:BL} J. Bourgain and D. Li, \emph{Strong ill-posedness of the incompressible Euler equation in borderline Sobolev spaces}, Invent. Math. \textbf{201} (2015), 97--157. \url{https://arxiv.org/abs/1307.7090}.
\bibitem{lit:Brenier} Y. Brenier, \emph{Generalized solutions and hydrostatic approximation of the Euler equations}, Physica D \textbf{237} (2008), 1982--1988. \url{https://doi.org/10.1016/j.physd.2008.02.026}.
\bibitem{lit:CCGO} \'{A}. Castro, D. C\'ordoba, F. Gancedo, and R. Orive, \emph{Incompressible flow in porous media with fractional diffusion}, Nonlinearity \textbf{22} (2009), 1791--1815. \url{https://doi.org/10.1088/0951-7715/22/8/002}.
\bibitem{lit:CCL} \'{A}. Castro, D. C\'ordoba, and D. Lear, \emph{Global existence of quasi-stratified solutions for the confined IPM equation}, Arch. Ration. Mech. Anal. \textbf{232} (2019), 437--471. \url{https://arxiv.org/abs/1804.08490}.
\bibitem{lit:Chen} J. Chen, \emph{Asymptotically self-similar blowup for 3D incompressible Euler with $C^{1,1/3-}$ velocity II: 3D profiles, blowup, and limiting behavior}, preprint (2026). \url{https://arxiv.org/abs/2605.15130}.
\bibitem{lit:CH1} J. Chen and T. Y. Hou, \emph{Stable nearly self-similar blowup of the 2D Boussinesq and 3D Euler equations with smooth data I: Analysis}, preprint (2022). \url{https://arxiv.org/abs/2210.07191}.
\bibitem{lit:CH2} J. Chen and T. Y. Hou, \emph{Stable nearly self-similar blowup of the 2D Boussinesq and 3D Euler equations with smooth data II: Rigorous numerics}, Multiscale Model. Simul. \textbf{23} (2025), 25--130. \url{https://doi.org/10.1137/23M1580395}.
\bibitem{lit:Cheverry} C. Cheverry, \emph{Cascade of phases in turbulent flows}, Bull. Soc. Math. France \textbf{134} (2006), 33--82. \url{https://doi.org/10.24033/bsmf.2501}.
\bibitem{lit:CPT} C. Collot, C. Prange, and J. Tan, \emph{Stable self-similar singularity formation for infinite energy solutions of the incompressible porous medium equations}, Comm. Math. Phys. (2026). \url{https://doi.org/10.1007/s00220-026-05688-0}.
\bibitem{lit:CFG} D. C\'ordoba, D. Faraco, and F. Gancedo, \emph{Lack of uniqueness for weak solutions of the incompressible porous media equation}, Arch. Ration. Mech. Anal. \textbf{200} (2011), 725--746. \url{https://arxiv.org/abs/0912.3210}.
\bibitem{lit:CGO} D. C\'ordoba, F. Gancedo, and R. Orive, \emph{Analytical behavior of two-dimensional incompressible flow in porous media}, J. Math. Phys. \textbf{48} (2007), 065206. \url{https://doi.org/10.1063/1.2404593}.
\bibitem{lit:CMEuler} D. C\'ordoba and L. Mart\'inez-Zoroa, \emph{Blow-up for the incompressible 3D-Euler equations with uniform $C^{1,\frac12-\varepsilon}\cap L^2$ force}, preprint (2023). \url{https://arxiv.org/abs/2309.08495}.
\bibitem{lit:CMIPM} D. C\'ordoba and L. Mart\'inez-Zoroa, \emph{Finite time singularities of smooth solutions for the 2D incompressible porous media (IPM) equation with a smooth source}, preprint (2024), revised 2025. \url{https://arxiv.org/abs/2410.22920}.
\bibitem{lit:CMZ} D. C\'ordoba, L. Mart\'inez-Zoroa, and F. Zheng, \emph{Finite time singularities to the 3D incompressible Euler equations for solutions in $C^\infty(\mathbb R^3\setminus\{0\})\cap C^{1,\alpha}\cap L^2$}, Ann. PDE \textbf{11} (2025), 19. \url{https://arxiv.org/abs/2308.12197}.
\bibitem{lit:CMZHypo} D. C\'ordoba, L. Mart\'inez-Zoroa, and F. Zheng, \emph{Finite time blow-up for the hypodissipative Navier Stokes equations with a force in $L^1_t C_x^{1,\epsilon}\cap L^\infty_t L_x^2$}, Arch. Ration. Mech. Anal. \textbf{250} (2026), 38. \url{https://doi.org/10.1007/s00205-026-02198-0}.
\bibitem{lit:CC} A. D. D. Craik and W. O. Criminale, \emph{Evolution of wavelike disturbances in shear flows: a class of exact solutions of the Navier--Stokes equations}, Proc. R. Soc. Lond. A \textbf{406} (1986), 13--26. \url{https://doi.org/10.1098/rspa.1986.0061}.
\bibitem{lit:DLS} C. De Lellis and L. Sz\'ekelyhidi, Jr., \emph{The Euler equations as a differential inclusion}, Ann. of Math. \textbf{170} (2009), 1417--1436. \url{https://doi.org/10.4007/annals.2009.170.1417}.
\bibitem{lit:Dembski} K. H. Dembski, \emph{Singularity formation in the incompressible porous medium equation without boundary mass}, preprint (2025). \url{https://arxiv.org/abs/2511.01827}.
\bibitem{lit:ElgindiIPM} T. M. Elgindi, \emph{On the asymptotic stability of stationary solutions of the inviscid incompressible porous medium equation}, Arch. Ration. Mech. Anal. \textbf{225} (2017), 573--599. \url{https://arxiv.org/abs/1411.6958}.
\bibitem{lit:ElgindiEuler} T. M. Elgindi, \emph{Finite-time singularity formation for $C^{1,\alpha}$ solutions to the incompressible Euler equations on $\mathbb R^3$}, Ann. of Math. \textbf{194} (2021), 647--727. \url{https://doi.org/10.4007/annals.2021.194.3.2}.
\bibitem{lit:EGM} T. M. Elgindi, T.-E. Ghoul, and N. Masmoudi, \emph{On the stability of self-similar blow-up for $C^{1,\alpha}$ solutions to the incompressible Euler equations on $\mathbb R^3$}, Camb. J. Math. \textbf{9} (2021), 1035--1075. \url{https://arxiv.org/abs/1910.14071}.
\bibitem{lit:EJ} T. M. Elgindi and I.-J. Jeong, \emph{Ill-posedness for the incompressible Euler equations in critical Sobolev spaces}, Ann. PDE \textbf{3} (2017), 7. \url{https://doi.org/10.1007/s40818-017-0027-7}.
\bibitem{lit:EP} T. M. Elgindi and F. Pasqualotto, \emph{From instability to singularity formation in incompressible fluids}, preprint (2023). \url{https://arxiv.org/abs/2310.19780}.
\bibitem{lit:FH} B. R. Fabijonas and D. D. Holm, \emph{Multi-frequency Craik--Criminale solutions of the Navier--Stokes equations}, J. Fluid Mech. \textbf{506} (2004), 207--215. \url{https://doi.org/10.1017/S0022112004008511}.
\bibitem{lit:FV} S. Friedlander and M. M. Vishik, \emph{Instability criteria for the flow of an inviscid incompressible fluid}, Phys. Rev. Lett. \textbf{66} (1991), 2204--2206. \url{https://doi.org/10.1103/PhysRevLett.66.2204}.
\bibitem{lit:Isett} P. Isett, \emph{A proof of Onsager's conjecture}, Ann. of Math. \textbf{188} (2018), 871--963. \url{https://arxiv.org/abs/1608.08301}.
\bibitem{lit:JK} M. J. Jo and J. Kim, \emph{Quantitative asymptotic stability of the quasi-linearly stratified densities in the IPM equation with the sharp decay rates}, J. Funct. Anal. \textbf{286} (2024), 110401. \url{https://arxiv.org/abs/2210.11437}.
\bibitem{lit:KY} A. Kiselev and Y. Yao, \emph{Small scale formations in the incompressible porous media equation}, Arch. Ration. Mech. Anal. \textbf{247} (2023), 1. \url{https://arxiv.org/abs/2102.05213}.
\bibitem{lit:LL} S. Le Diz\`es and S. Leblanc, \emph{Note on `Multi-frequency Craik--Criminale solutions of the Navier--Stokes equations' by B. R. Fabijonas and D. D. Holm}, J. Fluid Mech. \textbf{550} (2006), 43--50. \url{https://doi.org/10.1017/S0022112005007949}.
\bibitem{lit:LH} A. Lifschitz and E. Hameiri, \emph{Local stability conditions in fluid dynamics}, Phys. Fluids A \textbf{3} (1991), 2644--2651. \url{https://doi.org/10.1063/1.858153}.
\bibitem{lit:OpenAI} OpenAI, \emph{Finite time blowup for the Euler equation}, manuscript (2026). \url{https://cdn.openai.com/pdf/315b36cd-ec98-4023-8342-93345194ece1/euler.pdf}.
\bibitem{lit:Park} J. Park, \emph{Stability analysis of the incompressible porous media equation and the Stokes transport system via energy structure}, preprint (2024). \url{https://arxiv.org/abs/2403.14187}.
\bibitem{lit:SWZZ} F. Shao, D. Wei, P. Zhang, and Z. Zhang, \emph{Self-similar blow-up solutions of $d$-dimensional incompressible Euler equations with $C^{1,(1-2/d)-}$ velocity}, preprint (2026), version 2, 22 July 2026. \url{https://arxiv.org/abs/2605.19716v2}.
\bibitem{lit:Shkoller} S. Shkoller, \emph{Incompressible Euler blowup at the $C^{1,1/3}$ threshold}, preprint (2026). \url{https://arxiv.org/abs/2603.10945}.
\bibitem{lit:ZhengEuler} F. Zheng, \emph{Finite-time blowup of the Euler equation without forcing}. \url{https://arxiv.org/abs/2610.02509}.
\end{thebibliography}
\end{document}